\documentclass[11pt, toc=flat]{article}

\usepackage[utf8]{inputenc}
\usepackage[colorlinks = true]{hyperref}
\hypersetup{citecolor=gray,
linkcolor = gray,
urlcolor=gray}
\usepackage{etex}
\usepackage[shortlabels]{enumitem}
\usepackage{subfiles}
\usepackage{lmodern}
\usepackage{titlesec}
\usepackage{wrapfig}
\usepackage{csquotes}
\usepackage[T1]{fontenc}
\setenumerate[1]{label={(\arabic*)}}

\usepackage{upgreek}
\usepackage{tikz-cd}
\usepackage{mathrsfs}
\usepackage{sseq}
\usepackage{mathdots}
\usepackage{thm-restate}
\usepackage{amsmath}
\usepackage{amsxtra}
\usepackage{amscd}
\usepackage{amsthm}
\usepackage{amsfonts}
\usepackage{amssymb}
\usepackage{mathtools}
\usepackage[scr=rsfs]{mathalfa}
\usepackage{eucal}
\usepackage{bbm}
\usepackage[all,cmtip]{xy}

\usepackage{graphicx}
\usepackage{tikz}
\usepackage{morefloats}
\usepackage{pdfpages}
\usetikzlibrary{matrix,arrows,decorations.pathmorphing,quotes,angles,decorations.markings,shapes.misc}
\RequirePackage{color}
\definecolor{lavender}{rgb}{0.59, 0.44, 0.84}
\usepackage{microtype}
\usepackage{epigraph}
\usepackage{tocloft}
\usepackage[colorinlistoftodos, textsize=tiny]{todonotes}

\usepackage[headsep=1cm,headheight=2cm,margin=1in]{geometry}
\usepackage[framemethod=TikZ]{mdframed}
\usepackage{setspace}
\titleformat{\subsection}[hang]{\large\bfseries}{\thesubsection\hsp\textcolor{gray!75}{|}\hsp}{0pt}{\large\bfseries}
\titleformat{\subsubsection}[hang]{\bfseries}{\thesubsubsection\hsp\textcolor{gray!75}{|}\hsp}{0pt}{\bfseries}
\titleformat{\part}[display]{\Huge\bfseries}{\partname~\thepart:}{20pt}{}{}
\newcommand{\hsp}{\hspace{20pt}}
\titleformat{\section}[hang]{\Large\bfseries}{\thesection\hsp\textcolor{gray!75}{|}\hsp}{0pt}{\Large\bfseries}

\usepackage[nottoc]{tocbibind}

\usepackage{makeidx}
\newcommand{\defn}[1]{\textbf{\emph{#1}}\index{#1}}
\makeindex

\usepackage{lipsum}
\usepackage{cleveref}
\crefname{section}{\S \!\!}{\S\S \!\!}
\crefname{paragraph}{\S \!\!}{\S\S \!\!}
\crefname{appendix}{\S \!\!}{\S\S \!\!}
\crefname{equation}{}{}
\crefname{enumi}{}{}

\allowdisplaybreaks

\theoremstyle{plain}
\newtheorem{theorem}{Theorem}[subsection]
\newtheorem{mainthm}{Theorem}

\newtheorem{proposition}[theorem]{Proposition}
\newtheorem{lemma}[theorem]{Lemma}

\newtheorem{corollary}[theorem]{Corollary}

\theoremstyle{definition}

\newtheorem{definition}[theorem]{Definition}

\newtheorem{localass}[theorem]{Local Assumption}

\newtheorem{remark}[theorem]{Remark}

\newtheorem{example}[theorem]{Example}

\newtheorem{observation}[theorem]{Observation}
\newtheorem{convention}[theorem]{Convention}

\newtheorem{notation}[theorem]{Notation}

\theoremstyle{remark}

\newcommand\scl{\mathscr L}

\newcommand\sco{\mathscr O}

\newcommand\fr{ { \mathfrak r}}

\newcommand\fB{ { \mathfrak B}}

\DeclareMathOperator\Spec{Spec}

\renewcommand{\lim}{{\sf lim}}

\newcommand{\cat}[1]{\mathsf{#1}}

\newcommand{\Fun}{\cat{Fun}}
\newcommand{\Alg}{\cat{Alg}}

\newcommand{\Spaces}{\mathcal{S}}
\newcommand{\spectra}{{\mathcal{S}\cat{p}}}
\newcommand{\Spectra}{{\mathcal{S}\cat{p}}}
\newcommand{\Sp}{{\Spectra}}

\newcommand{\chom}{\cat{hom}}
\renewcommand{\hom}{\cat{hom}}
\renewcommand{\lim}{{\sf lim}}
\newcommand{\Mot}{\cat{Mot}}

\newcommand{\fgt}{\cat{fgt}}
\newcommand{\St}{\cat{St}}
\newcommand{\Cat}{\cat{Cat}}
\newcommand{\op}{\cat{op}}
\newcommand{\fin}{\cat{fin}}
\newcommand{\Yo}{\cat{Yo}}
\newcommand{\ex}{\cat{ex}}
\newcommand{\twoex}{{2\text{-}\ex}}

\newcommand{\bDelta}{\mathbf{\Delta}}

\newcommand{\cotensor}{\pitchfork}
\newcommand{\tensor}{\odot}

\newcommand{\Ab}{{\sf Ab}}
\newcommand{\CMon}{{\sf CMon}}
\newcommand{\THH}{{\sf THH}}
\newcommand{\enr}{{\text{-}{\sf enr}}}

\def\cA{\mathcal A}\def\cB{\mathcal B}\def\cC{\mathcal C}\def\cD{\mathcal D}
\def\cE{\mathcal E}\def\cF{\mathcal F}
\def\cI{\mathcal I}\def\cK{\mathcal K}\def\cL{\mathcal L}
\def\cM{\mathcal M}\def\cO{\mathcal O}\def\cP{\mathcal P}
\def\cT{\mathcal T}
\def\cV{\mathcal V}\def\cW{\mathcal W}\def\cX{\mathcal X}
\def\cY{\mathcal Y}\def\cZ{\mathcal Z}

\def\CC{\mathbb C}
\def\EE{\mathbb E}
\def\KK{\mathbb K}
\def\NN{\mathbb N}
\def\QQ{\mathbb Q}\def\SS{\mathbb S}\def\TT{\mathbb T}

\def\sB{\mathsf B}\def\sC{\mathsf C}

\def\fB{\mathfrak B}

\newcommand{\ms}{\mathscr}

\newcommand{\K}{\cat{K}}
\newcommand{\Kto}{\K^{(2,1)}}
\newcommand{\Ktwo}{\K^{(2)}}
\newcommand{\oneLoc}{{\sf{1Loc}}}
\newcommand{\Adj}{{\sf{Adj}}}
\newcommand{\zeroLoc}{{\sf 0Loc}}
\newcommand{\Perf}{{\sf Perf}}
\newcommand{\Ind}{{\sf Ind}}
\newcommand{\cofib}{{\sf cofib}}
\newcommand{\fib}{{\sf fib}}
\newcommand{\CAlg}{{\sf CAlg}}

\newcommand{\dzbl}{{\sf dzbl}}

\newcommand{\TC}{\cat{TC}}
\newcommand{\tr}{{\sf tr}}

\newcommand{\QCtwo}{\QC^{(2)}}
\newcommand{\THHtwo}{\THH^{(2)}}
\newcommand{\TCtwo}{\TC^{(2)}}
\newcommand{\Perftwo}{{\sf Perf^{(2)}}}
\newcommand{\Cyclo}{{\sf Cyc}}
\newcommand{\Cyclic}{{\sf C}}
\newcommand{\trtwo}{\tr^{(2)}}

\newcommand{\colim}{{\sf colim}}
\renewcommand{\Spec}{{\sf Spec}}

\newcommand{\SH}{{\sf SH}}

\newcommand{\ra}{\rightarrow}

\newcommand{\la}{\leftarrow}

\newcommand{\longhookra}{\lhook\joinrel\longrightarrow}

\newcommand{\adjarr}{\rightleftarrows}
\newcommand{\xra}{\xrightarrow}
\newcommand{\xla}{\xleftarrow}

\newcommand{\longla}{\longleftarrow}
\newcommand{\longra}{\longrightarrow}

\newcommand{\xlongra}[1]{\stackrel{#1}{\longra}}
\newcommand{\xlongla}[1]{\stackrel{#1}{\longla}}

\newcommand{\Env}{{\sf Env}}

\newcommand{\adj}{\dashv}

\newcommand{\ul}[1]{\underline{#1}}

\newcommand{\bit}[1]{\textbf{\textit{#1}}}

\newcommand{\End}{{\sf end}}

\newcommand{\fCat}{\cat{fCat}}
\newcommand{\Mod}{{\sf Mod}}
\newcommand{\univ}{\cL}
\newcommand{\const}{{\sf const}}
\newcommand{\uno}{\mathbbm{1}}
\newcommand{\Add}{\sf Add}
\newcommand{\ev}{{\sf ev}}
\newcommand{\idem}{{\sf idem}}
\newcommand{\id}{{\sf id}}
\newcommand{\da}{\downarrow}
\newcommand{\coCart}{{\sf coCart}}
\newcommand{\htpy}{{\sf h}}

\newcommand{\Gr}{{\sf Gr}}
\newcommand{\Fin}{{\sf Fin}}

\renewcommand{\th}{^{\text{th}}}

\newcommand{\longmapsfrom}{\raisebox{3.5pt}{\rotatebox{180}{$\longmapsto$}}}

\newcommand{\QC}{{\sf QC}}

\newcommand{\PrLSt}{{\sf Pr}^{L,{\sf st}}}
\newcommand{\PrL}{{\sf Pr}^L}
\newcommand{\PrR}{{\sf Pr}^R}
\newcommand{\PrRSt}{{\sf Pr}^{R,{\sf st}}}
\newcommand{\what}[1]{\widehat{#1}}

\newcommand{\oneop}{{1\text{-}\sf{op}}}
\newcommand{\twoop}{{2\text{-}\sf op}}
\newcommand{\kop}{k\text{-}\op}
\newcommand{\onetwoop}{{1\&2\text{-}\sf{op}}}
\newcommand{\LMod}{{\sf LMod}}

\newcommand{\pt}{{\sf pt}}

\newcommand{\weak}{\flat}
\newcommand{\sadd}{{\sf semiadd}}

\newcommand{\Aut}{{\sf aut}}

\renewcommand{\fr}{{\sf fr}}
\newcommand{\sat}{{\sf sat}}

\newcommand{\inn}{{\sf inn}}

\title{
\huge \thetitle
}

\author{\theauthor}

\date{\thedate}

\newcommand{\thetitle}{A universal characterization of noncommutative motives and secondary algebraic K-theory}
\newcommand{\thedate}{\today}
\newcommand{\theauthor}{Aaron Mazel-Gee and Reuben Stern}

\begin{document}

\maketitle

\begin{abstract}
We provide a universal characterization of the construction taking a scheme $X$ to its stable $\infty$-category $\Mot(X)$ of noncommutative motives, patterned after the universal characterization of algebraic K-theory due to Blumberg--Gepner--Tabuada. As a consequence, we obtain a corepresentability theorem for secondary K-theory. We envision this as a fundamental tool for the construction of trace maps from secondary K-theory.

Towards these main goals, we introduce a preliminary formalism of ``stable $(\infty, 2)$-categories''; notable examples of these include (quasicoherent or constructible) sheaves of stable $\infty$-categories. We also develop the rudiments of a theory of presentable enriched $\infty$-categories -- and in particular, a theory of presentable $(\infty, n)$-categories -- which may be of independent interest.
\end{abstract}

\setcounter{tocdepth}{2}
\tableofcontents

\setcounter{section}{-1}

\section{Introduction}

\subsection{Overview}
\label{subsection.overview}

\bit{Secondary algebraic K-theory} is an algebro-geometric analogue of elliptic cohomology. It was introduced by To{\"e}n--Vezzosi in \cite{TV2009} as a categorification of ordinary (``primary'') algebraic K-theory: whereas primary K-theory is built from sheaves of vector spaces, secondary K-theory is built from sheaves of categories.

In this paper, we prove a corepresentability result for secondary K-theory (\Cref{mainthm:locating.K.2}). The main ingredient is a corepresentability result for noncommutative motives (\Cref{mainthm:universal.property.of.K.2.1}). These results pave the way for the construction and study of a \bit{secondary cyclotomic trace} from secondary K-theory.

Our theorems are supported by a number of foundational results on enriched and higher categories that we establish herein. Notably, we develop the rudiments of a theory of \textit{presentable} enriched $\infty$-categories, which specializes in particular to a theory of presentable $(\infty, n)$-categories.

The remainder of this introductory section is organized as follows.
\begin{itemize}

\item[\Cref{subsection.intro.background}:] We give some background and motivation for secondary K-theory.
\item[\Cref{subsection.intro.main.results}:] We describe our main results.

\item[\Cref{subsection.intro.horizons}:] We discuss a number of directions for future research that stem from the present work: the notion of stability for $(\infty,2)$-categories (\Cref{subsubsection.true.defn}), types of secondary K-theory (\Cref{subsubsection.flavors.of.Kthy}), a secondary $S_\bullet$-construction (\Cref{subsubsection.S.dot.constrn}), and secondary trace maps (\Cref{subsubsection.secondary.cylo.trace}).

\item[\Cref{subsection.intro.cat.thy}:] We survey the foundational results on enriched and higher categories that we establish in this paper. (This material can be read independently of the material preceding it.)

\item[\Cref{subsection.intro.outline}:] We give a linear overview of the paper.

\item[\Cref{subsection.intro.notn}:] We discuss our notation and conventions.

\item[\Cref{subsection.intro.ack}:] We give acknowledgments.

\end{itemize}

\noindent Henceforth, we take the ``implicit $\infty$ convention''; for instance, we simply use the term ``$n$-category'' to mean an $(\infty,n)$-category. Moreover, we use the term ``category'' to mean ``$1$-category'' (meaning $(\infty,1)$-category).

\subsection{Background and motivation}
\label{subsection.intro.background}

\bit{Chromatic homotopy theory} is a highly successful framework for organizing cohomological invariants of spaces. Specifically, it organizes cohomology theories into a hierarchy of increasing complexity according to their \textit{chromatic heights}. The canonical examples of cohomology theories at heights 0, 1, and 2 are respectively rational cohomology, complex K-theory, and elliptic cohomology. 

The first two terms in this hierarchy admit direct analogs in algebraic geometry: \'etale cohomology plays the role of rational cohomology, while algebraic K-theory plays the role of complex K-theory. This led To{\"e}n--Vezzosi to introduce \bit{secondary algebraic K-theory} as an analog of elliptic cohomology \cite{TV2009,toenDGcatlectures, ToenDGcats}. Their construction was inspired by the correspondence between chromatic height and categorical dimension that is indicated in the table below.\footnote{Although cohomology theories exist at all heights, it is a longstanding open problem to describe them in terms of cocycles at heights 2 and above \cite{StolzTeichner-ell}.}
\begin{figure}[h]
\def\arraystretch{1.2}
\begin{tabular}{ @{\quad} c @{\quad} | @{\quad} c @{\quad} | @{\quad} c @{\quad} | @{\quad} c @{\quad} }
At chromatic
&
... the prototypical
&
... cocycles for which
&
... which are
\\
height
&
cohomology theory is
&
are families of
&
objects of
\\ \hline \hline
0,
&
rational cohomology,
&
rational numbers,
&
the 0-category $\QQ$.
\\ \hline
1,
&
complex K-theory,
&
$\CC$-vector spaces,
&
the 1-category ${\sf Vect}_\CC$.
\\ \hline
2,
&
elliptic cohomology,
&
[...as-yet unknown...],
&
[...some 2-category...].
\end{tabular}
\end{figure}
Namely, given a scheme $X$, whereas its primary algebraic K-theory $\K(X)$ is constructed from its 1-category $\QC(X)$ of quasicoherent sheaves of (complexes of) vector spaces, they constructed its \textit{secondary} algebraic K-theory $\Ktwo(X)$ from its 2-category $\QCtwo(X)$ of quasicoherent sheaves of (small stable) \textit{categories} (\Cref{examples.of.st.2}\Cref{item.ShvCat.is.a.stable.two.cat}). A fundamental source of such sheaves is schemes over $X$: a map $Y \ra X$ determines the quasicoherent sheaf $U \mapsto \Perf(Y_{|U}) \in \Mod_{\Perf(U)}(\St)$.

We note here that secondary K-theory, categorified sheaf theory, and stable 2-categories (see \Cref{defn.stable.2.cat.in.intro} below) connect with a wide range of areas of mathematics in algebraic geometry and beyond: motivic measures and the K-theory of varieties \cite{BondalLarsenLunts, Campbell2019, CampbellZakharevichDevissage,looijenga-motivic,Poonen-notadomain}; secondary traces and higher character theory \cite{BZNsecondarytraces,CampbellPontoTraces,GanterKapranov2008,BZNnonlineartraces,Bartlett-geometry}; chromatic redshift \cite{AusRog-rational, Rog-White,KleinRog-fib,hahn2020redshift,Baas2004};\footnote{Note that there is a canonical map $\K \circ \K \ra \Ktwo$ from iterated K-theory to secondary K-theory, which is nontrivial \cite[Remark 6.23]{HStwo}. Indeed, it is expected that secondary K-theory is a substantially richer invariant than iterated K-theory, which is itself already very interesting (e.g.\! that of a separably closed field is related to the K-theory of topological K-theory \cite{Sus-algebraically, AusRog-topological}).} Floer theory and Fukaya categories \cite{NadZas,Douglas-thesis,Bottman-assoc,BC-Ainftytwocats}; and quantum algebra and categorification \cite{Lauda-intro,Mazurchuk-lectures,nakano2021noncommutative,DR-fusiontwocats}. Secondary K-theory itself is the topic of much recent work, including \cite{HStwo, HSthree, tabuada2016note, tabuada2016note2, tabuada-embedding}.

To{\"e}n--Vezzosi originally defined secondary K-theory as the primary K-theory of the category $\Perftwo(X)$ of fully dualizable quasicoherent sheaves of categories,\footnote{The category $\Perftwo(X)$ is often denoted $\Cat^\sat(X)$ (as its objects are often referred to as ``saturated'' sheaves of categories).} which carries a natural Waldhausen structure. This was refined by Hoyois--Scherotzke--Sibilla in \cite{HStwo}, who instead studied the K-theory of the stable category $\Mot(X) := \Mot(\Perftwo(X))$ of \bit{noncommutative motives} over $X$ (using \Cref{defn.motives.over.a.stable.two.cat}).\footnote{The two definitions are closely related; see \cite[Remark 6.18]{HStwo}.} It is this latter notion of secondary K-theory that we study here (\Cref{def.secondary.K.theory}).

We are motivated by a desire to compute secondary K-theory via trace methods. Specifically, we envision a \bit{secondary cyclotomic trace} map
\[
\Ktwo
\longra
\TCtwo
\]
to \bit{secondary topological cyclic homology}. As explained in \Cref{subsubsection.secondary.cylo.trace} below, our work provides a key step towards constructing such a map. Namely, as we explain in \Cref{subsection.intro.main.results}, our \Cref{mainthm:locating.K.2} is a categorification of the following universal characterization of primary algebraic K-theory of Blumberg--Gepner--Tabuada \cite{BGT} (see also \cite{Tabuada-HigherKTheory, CisTab-Nonconnective,BarwickAlgKTheoryofHigherCats}), which determines the primary cyclotomic trace map
\[
\K
\longra
\TC
\]
by Yoneda, as explained in \cite[\S 10.3]{BGT}.

\begin{theorem}
\label{thm:bgt-main}
Let $\cA,\cB \in \St$ be small stable categories, and let $\univ(\cA),\univ(\cB) \in \Mot$ denote their corresponding (additive) motives.\footnote{Using \Cref{defn.motives.over.a.stable.two.cat}, we have $\Mot := \Mot(\St^\omega)$. Note that \cite{BGT} states this theorem for idempotent-complete stable categories, but this restriction is unnecessary (see \Cref{mainthm.homs.in.Mot.of.X.are.K.theory.are.K.theory}).} Assuming $\cA$ is compact, there is a canonical equivalence of spectra 
\[
\hom_{\Mot}(\univ(\cA), \univ(\cB)) \simeq \K(\Fun^\ex(\cA, \cB))
~.
\]
In particular, taking $\cA = \uno_{\St} = \Spectra^\fin$ gives a canonical equivalence 
\[
\hom_{\Mot}(\univ(\uno_\St), \univ(\cB)) \simeq \K(\cB)
~.
\]
\end{theorem}

\noindent More precisely, the equivalence of \Cref{thm:bgt-main} arises from the canonical functor
\[
\iota_1 \St
\xlongra{\univ}
\Mot
\]
from the 1-category of small stable categories.\footnote{Our choice of notation $\univ$ is explained in \Cref{rmk.K.and.Kto.and.Ktwo.as.linearizations}. Note that \Cref{thm:bgt-main} may be interpreted as the assertion that K-theory categorifies stable homotopy: we have the categorification
\[
\iota_1 \St
\xlongra{\univ}
\Mot
\qquad
\text{of}
\qquad
\Spaces
\xra{\Sigma^\infty_+}
\Spectra
~,
\]
and thereafter we have the categorification
\[
\iota_1 \St \xra{\hom_\Mot(\univ(\uno_\St),\univ(-))} \Spectra
\qquad
\text{of}
\qquad
\Spaces \xra{\hom_\Spectra(\Sigma^\infty_+ (\uno_\Spaces) , \Sigma^\infty_+(-))} \Spectra
~.
\]
}

\subsection{Main results}
\label{subsection.intro.main.results}

In this work, we provide a universal mapping property for secondary K-theory analogous to that for primary K-theory of \Cref{thm:bgt-main}, as we now explain.

The first step is to establish a categorical context for considering sheaves of categories, analogous to the context that the 2-category $\St$ of stable categories provides for quasicoherent sheaves. For this, we introduce the following.

\begin{definition}
\label{defn.stable.2.cat.in.intro}
A \bit{stable 2-category} is a 2-category $\cX$ satisfying the following conditions.
\begin{enumerate}
    \item\label{condition.in.intro.defn.X.is.enriched.in.stable.cats} $\cX$ is a \textit{stably-enriched} 2-category: its hom-categories are stable and its composition bifunctors are biexact.
    \item $\cX$ is \textit{semiadditive}: it admits finite coproducts and finite products, and these agree.\footnote{In fact, condition \Cref{condition.in.intro.defn.X.is.enriched.in.stable.cats} implies that finite coproducts are automatically finite products and reversely (\Cref{obs.products.and.coproducts.in.stably.enriched.two.cat}).}
\end{enumerate}
Small stable 2-categories assemble into a 3-category $\St_2$, which is closed symmetric monoidal (\Cref{lem.omnibus.in.section.one.point.one}). In particular, for any stable 2-categories $\cX, \cY \in \St_2$, we have an internal hom $\Fun^\twoex(\cX, \cY) \in \St_2$, the stable 2-category of \bit{2-exact functors}.
\end{definition}

\noindent Namely, whereas quasicoherent sheaves on $X$ assemble into a stable category $\QC(X)$, quasicoherent sheaves of categories on $X$ assemble into a stable 2-category $\QCtwo(X)$. We list some further basic examples of stable 2-categories in \Cref{examples.of.st.2}; many more examples arise throughout the literature, as discussed in \Cref{subsection.intro.background}.

\begin{remark}
\label{rmk.intro.defn.of.stable.2.cat.is.provisional}
We consider \Cref{defn.stable.2.cat.in.intro} to be provisional; we discuss possible refinements in \Cref{subsubsection.true.defn}. However, we find it extremely likely that any improved definition of stable 2-categories will define a full sub-3-category $\St_2' \subseteq \St_2$ whose inclusion admits a symmetric monoidal left adjoint that preserves compact objects, and in this situation our results apply equally well to $\St_2'$.
\end{remark}

The provisionality of \Cref{defn.stable.2.cat.in.intro} is illustrated by the following.

\begin{example}
\label{example.unit.in.St.two}
The unit object of $\St_2$ is the full sub-2-category $\uno_{\St_2} \subset \St$ on the objects $\{ (\Spectra^\fin)^{\oplus n} \}_{n \in \NN}$.\footnote{This identification follows easily from \Cref{lem.omnibus.in.section.one.point.one}: by construction it is given by $\Env_2(\uno_{\Cat(\St)}) = \Env_2(\fB \Spectra^\fin)$, an explicit description of which follows from the proof of \Cref{prop.Env.V}.} More informally, it can be described as follows: its objects are natural numbers; its hom-object $\hom_{\uno_{\St_2}}(m,n) \in \St$ is the stable category of $n$-by-$m$ matrices of finite spectra; composition is multiplication of matrices; and the symmetric monoidal structure is given on objects by multiplication of natural numbers.
\end{example}

Primary K-theory interacts closely both with 0-localization sequences (i.e.\! co/fiber sequences) \textit{within} a stable category as well as with 1-localization sequences (i.e.\! split Verdier sequences) \textit{among} stable categories (see \Cref{def.0.loc.seq,def:localization.sequence}). These relationships categorify: secondary K-theory interacts closely both with 1-localization sequences \emph{within} a stable 2-category as well as with 2-localization sequences \textit{among} stable 2-categories (see \Cref{def:2.localization.sequence}).

Given any small stable 2-category $\cX \in \St_2$, we can form the stable 1-category $\Mot(\cX)$ of (\bit{noncommutative}) \bit{motives} over $\cX$: by definition, this is the target of the universal functor
\[
\iota_1 \cX
\xra{\univ_\cX}
\Mot(\cX)
\]
to a presentable stable 1-category that carries 1-localization sequences to exact sequences. Our first main result is a version of \Cref{thm:bgt-main} that applies to any small stable 2-category.

\begin{mainthm}[Theorem \ref{thm.homs.in.Mot.of.X.are.K.theory}]
	\label{mainthm.homs.in.Mot.of.X.are.K.theory.are.K.theory}
	Fix a small stable 2-category $\cX \in \St_2$. For any objects $A, B \in \mathcal{X}$, there is a canonical equivalence of spectra
	\[
	\chom_{\Mot(\mathcal{X})}(\univ_\cX(A), \univ_\cX(B)) \simeq \K(\chom_{\mathcal{X}}(A, B))
	~.
	\]
\end{mainthm}

Now, as an auxiliary notion, we define the \bit{(2,1)-ary K-theory} of $\cX \in \St_2$ to be
\[
\Kto(\cX) := \Mot(\cX)^\omega,
\]
the small stable category of compact objects in $\Mot(\cX)$ (note that $\Mot(\cX)$ is compactly generated). So by definition, we have a factorization 
\[ \begin{tikzcd}
\iota_1 \St_2
\arrow{rr}{\Ktwo}
\arrow{rd}[sloped, anchor=north, swap]{\iota_1\Kto}
&
&
\Spectra
\\
&
\iota_1 \St
\arrow{ru}[sloped, anchor=north, swap]{\K}
\end{tikzcd}
~.\footnote{Our notation is such that the functor $\K^{(n,m)}$ carries stable $n$-categories to stable $m$-categories, where we write $\K^{(n)} := \K^{(n,0)}$ and $\K := \K^{(1)}$ for simplicity (and take the convention that a stable 0-category is a spectrum).}
\]
Our second main result is a corepresentability theorem for (2,1)-ary K-theory.\footnote{Via the commutative diagram
\[ \begin{tikzcd}[ampersand replacement=\&, row sep=0.25cm]
\&
\PrL_\omega
\arrow[leftarrow, xshift=-0.9ex]{dd}[sloped, anchor=south]{\sim}[swap]{\Ind}
\arrow[xshift=0.9ex]{dd}{(-)^\omega}
\\
\iota_1 \St_2
\arrow{ru}[sloped, anchor=south]{\Mot}
\arrow{rd}[sloped, anchor=north, swap]{\Kto}
\\
\&
\St^\idem
\end{tikzcd}~,\]
this may be equivalently thought of as a corepresentability theorem for motives.} For this, we introduce the stable 2-category $\Mot_{2,1}$ of \bit{(2,1)-motives} (\Cref{def.2.1.motives}): this is the target of the universal functor
\[
\iota_2\St_2 \xra{\univ_{2,1}} \Mot_{2,1}
\]
to a presentable stable 2-category that carries 2-localization sequences to 1-localization sequences.\footnote{We are confident that this characterization of $\Mot_{2,1}$ is correct, but we do not prove it (see \Cref{rmk.U.2.1.is.univ.2.1.additive.invt}).}

\begin{mainthm}[Theorem \ref{thm:universal.property.of.K.2.1}]
	\label{mainthm:universal.property.of.K.2.1}
	Let $\cX,\cY \in \St_2$ be small stable 2-categories, and assume that $\cX$ is compact. There is a canonical equivalence of stable categories
	\[
	\hom_{\Mot_{2,1}}(\univ_{2,1}(\cX), \univ_{2,1}(\cY)) \simeq \Kto(\Fun^\twoex(\cX, \cY))
	~.
	\]
\end{mainthm}

Our corepresentability theorem for secondary K-theory arises by daisy-chaining Theorems \ref{mainthm.homs.in.Mot.of.X.are.K.theory.are.K.theory} and \ref{mainthm:universal.property.of.K.2.1} as follows. We define the category $\Mot_2 := \Mot(\Mot^\omega_{2,1})$ of \bit{2-motives}; by construction, this comes equipped with a functor
\[
\iota_1\St_2^\omega \xra{\univ_2} \Mot_2
\]
from the 1-category of compact stable 2-categories.

\begin{mainthm}[Theorem \ref{thm:locating.K.2}]
	\label{mainthm:locating.K.2}
	Let $\cX,\cY \in \St_2^\omega$ be compact stable 2-categories.
	There is a canonical equivalence of spectra
	\[
	\hom_{\Mot_2}(\univ_2(\cX), \univ_2(\cY)) \simeq \Ktwo(\Fun^\twoex(\cX, \cY))
	~.
	\]
	In particular, taking $\cX = \uno_{\St_2}$ gives a canonical equivalence
	\[
	\hom_{\Mot_2}(\univ_2(\uno_{\St_2}),\univ_2(\cY)) \simeq \Ktwo(\cY)
	~.
	\]
\end{mainthm}

\begin{remark}
\label{rmk.K.and.Kto.and.Ktwo.as.linearizations}
Our choice to denote various morphisms by $\univ$ stems from the fact that these are \textit{linearization} maps, in two senses. Most primitively, given a stable category $\cC \in \St$, (the infinite loopspace of) its K-theory spectrum receives a canonical map
\[
\iota_0 \cC
\xra{\univ_\cC}
\K(\cC)
\]
from its maximal subgroupoid: the universal Euler characteristic for objects of $\cC$. One categorical dimension higher, given a stable 2-category $\cX \in \St_2$, (the underlying category of) its (2,1)-ary K-theory stable category receives a canonical map
\[
\iota_1 \cX
\xra{\univ_\cX}
\Kto(\cX)
\]
from its maximal sub-1-category, and thereafter (the infinite loopspace of) its secondary K-theory spectrum receives a canonical composite map
\[
\iota_0 \cX
\xra{\iota_0 \univ_\cX}
\iota_0 \Kto(\cX)
\xra{\univ_{\Kto(\cX)}}
\Ktwo(\cX)
\]
from its maximal subgroupoid. Of course, the latter maps may also be thought of as assigning universal Euler characteristics to objects of $\cX$. At the same time, the functor $\univ_\cX$ carries each object of $\iota_1 \cX$ to a sort of linearization thereof; for instance, in the case that $\cX = \St$ it carries each object $\cC \in \St$ to an object $\univ(\cC) \in \Mot$ that (by \Cref{thm:bgt-main}) records the K-theory of $\cC$.\footnote{Note that in all cases these linearization operations make reference to higher-categorical structure that has been discarded from the source (e.g.\! the construction of $\K(\cC)$ makes reference to bifiber (i.e.\! fiber and cofiber) sequences in $\cC$), in contrast with those indicated in \Cref{rmk.sigma.infty.n.plus}. Indeed, using the notation suggested there, by adjunction we obtain canonical morphisms
\[
\Sigma^{(\infty,0)}_+ \iota_0 \cC
\longra
\K(\cC)
~,
\qquad
\Sigma^{(\infty,1)}_+ \iota_1 \cX
\longra
\Kto(\cX)
~,
\qquad
\text{and}
\qquad
\Sigma^{(\infty,2)}_+ \iota_2 \St_2^\omega
\longra
\Mot_{2,1}
\]
of stable $n$-categories for $0 \leq n \leq 2$.}
\end{remark}

\subsection{Questions and further directions}
\label{subsection.intro.horizons}

In this highly speculative subsection, we lay out a number of directions for further research that our work raises. It is organized as follows.
\begin{itemize}

\item[\Cref{subsubsection.true.defn}:] We discuss the notion of stability for 2-categories.

\item[\Cref{subsubsection.flavors.of.Kthy}:] We discuss types of localization sequences and types of secondary K-theory.

\item[\Cref{subsubsection.S.dot.constrn}:] We discuss the possibility of a secondary $S_\bullet$-construction.

\item[\Cref{subsubsection.secondary.cylo.trace}:] We discuss secondary traces.

\end{itemize}
We note once and for all that there are a number of obvious connections between these discussions, which we leave implicit for lack of anything useful to say.

\subsubsection{The notion of stability for 2-categories}
\label{subsubsection.true.defn}

Recall from \Cref{rmk.intro.defn.of.stable.2.cat.is.provisional} that we consider our definition of stable 2-categories to be provisional. We would be interested to see a refined definition,\footnote{In \cite[\S 5]{LinSatWes-Twisted}, Lind--Sati--Westerland give a hypothetical definition of cocomplete stable $n$-categories as those objects of $\Cat_n(\Spectra)$ that admit all $\Cat_{n-1}(\Spectra)$-weighted colimits.} which we would hope to have the following features and consequences.
\begin{enumerate}

    \item It closely reflects specific properties of the stable 2-category $\St$ of stable categories, particularly those involving fiber and cofiber sequences. Moreover, the fundamental examples given in \Cref{examples.of.st.2} remain examples of stable 2-categories.
        
    \item It does not refer to an \emph{a priori} enrichment in $\St$, analogously to how the definition of a stable category does not refer to an \emph{a priori} enrichment in $\Spectra$.
    
    \item It does not refer to specific diagram shapes but instead is ``coordinate-free'', corresponding to how a 1-category is stable if and only if it admits finite limits and finite colimits and moreover these commute \cite{Groth-stable}.
    
    \item The unit object is closely related to $\St$ (recall \Cref{example.unit.in.St.two}), analogously to how the unit object of $\St$ is the category $\Spectra^\fin$ of finite spectra.
    
    \item The Cauchy-complete stable envelope (akin to \Cref{defn.Env.two}) of a stably-enriched 2-category (\Cref{defn.stably.enriched.two.cat}) is its category of (say right-)adjointable modules \cite{Heine-Weighted}.\footnote{Here we refer to adjoints in a Morita 3-category, whose objects are stably-enriched 2-categories and whose morphisms from $\cX$ to $\cY$ are 2-exact functors $\cX^\op \otimes \cY \ra \St$.}$^,$\footnote{We expect that a higher-categorical notion of idempotent-completeness should be invoked here \cite{DR-fusiontwocats,GJF-Condensations}.}
    
\end{enumerate}

\subsubsection{Variants of secondary K-theory}
\label{subsubsection.flavors.of.Kthy}

Primary algebraic K-theory has both connective and nonconnective variants, which are characterized in terms of their interactions with different types of localization sequences \cite{BGT}. Similarly, secondary K-theory admits a host of possible variants, and we would be interested in seeing a systematic treatment of these. To explain them, we discuss variants of primary K-theory in \Cref{flavors.of.primary.K.thys}, after a preliminary discussion of localization sequences in \Cref{pgh.localizn.seqs}; those discussions indicate that the notion of secondary K-theory that we study here is merely the simplest nontrivial variant. We would also be interested to see these notions related via descent, analogously to how nonconnective K-theory arises by enforcing (noncommutative) Nisnevich descent on connective K-theory \cite[Theorem 4.4]{Robalo-bridge}.

\paragraph{Localization sequences.}
\label{pgh.localizn.seqs}

As indicated in \Cref{subsection.intro.main.results}, in this paper we study notions of $n$-localization sequences -- that is, higher-categorical versions of bifiber sequences -- for varying $n$. Here we discuss possible variants, which are characterized in terms of their degrees of \textit{splitness}. We indicate the general patterns, and then explain them in low dimensions.\footnote{For simplicity, here we suppress all distinctions between data and conditions.}
\begin{itemize}

\item One may speak of an \bit{$n$-localization sequence} in a stable $(n+1)$-category for any $n \geq -1$. We take the convention that a stable 0-category is a spectrum: $\St_0 := \Spectra$.\footnote{Thinking even more negatively, one may define a stable $(-1)$-category to be a formal difference of finite sets, i.e.\! a point in the sphere spectrum: $\St_{-1} := \SS$.}

\item By definition, an $n$-localization sequence is \bit{$i$-split} for some $0 \leq i \leq n+2$. This notion becomes more restrictive as $i$ increases; a 0-split $n$-localization sequence is an arbitrary bifiber sequence, an $(n+1)$-split $n$-localization sequence is a direct sum decomposition, and only an $n$-localization sequence among zero objects is $(n+2)$-split.

\item For $n \geq 0$ and $i \geq 1$, an $i$-split $n$-localization sequence has a \bit{middle $(n-1)$-localization sequence}, which is definitionally $(i-1)$-split. Moreover, for any $j \geq i$, the former is $j$-split iff the latter is $(j-1)$-split.

\end{itemize}
Using this terminology, the $n$-localization sequences studied in the main body of this paper (for $0 \leq n \leq 2$) are $n$-split.

The low-dimensional examples of $i$-split $n$-localization sequences are summarized in the table below, which we explain row by row.
\begin{figure}[h]
\hspace{-1.2cm}
\def\arraystretch{1.2}
\begin{tabular}{ @{\quad} c @{\quad} | @{\quad} c @{\quad} || @{\quad} c @{\quad} | @{\quad} c @{\quad} | @{\quad} c @{\quad} | @{\quad} c @{\quad} }
$n$
&
stable $(n+1)$-category
&
$(n+2)$-split
&
$(n+1)$-split
&
$n$-split
&
$(n-1)$-split
\\ \hline \hline
$-1$
&
spectrum
&
$0 = 0+0$
&
$b = a+c$
&
&
\\ \cline{1-6}
$0$
&
stable category
&
$0 \rightleftarrows 0 \rightleftarrows 0$
&
$A \rightleftarrows A \oplus C \rightleftarrows C$
&
$A \ra B \ra C$
&
\\ \hline
$1$
&
stable 2-category
& 
$0 \rightleftarrows 0 \rightleftarrows 0$
&
$\cA \rightleftarrows \cA \oplus \cC \rightleftarrows \cC$
&
$\cA \rightleftarrows \cB \rightleftarrows \cC$
&
$\cA \ra \cB \ra \cC$
\end{tabular}
\end{figure}
\begin{enumerate}
    \item[$n = -1$:] One may speak of $(-1)$-localization sequences among objects of a spectrum, i.e.\! among points of its infinite loop space. A 0-split $(-1)$-localization sequence is an equation $b=a+c$, which is 1-split iff all of its terms are zero.
    
    \item[$n = 0$:] One may speak of 0-localization sequences among objects of a stable category. A 0-split 0-localization sequence is a bifiber sequence $A \ra B \ra C$, and it is 1-split iff it is split in the usual sense. A 1-split 0-localization sequence
    \[ \begin{tikzcd}
    A
    \arrow{r}{i}
    &
    B
    \arrow{r}{L}
    \arrow[bend left]{l}{R}
    &
    C
    \arrow[bend left]{l}{j}
    \end{tikzcd} \]
    has a middle $0$-split $(-1)$-localization sequence $\id_B = iR + jL$ in $\End_\Spectra(B)$.
    
    \item[$n = 1$:] One may speak of 1-localization sequences among objects of a stable 2-category. We discuss these in the canonical example of the stable 2-category $\St$ of small stable categories. A 0-split 1-localization sequence is a bifiber sequence $\cA \hookrightarrow \cB \ra \cC$ (i.e.\! a Verdier sequence).\footnote{In $\St^\idem$, the 0-split 1-localization sequences are precisely those sequences that become 1-split 1-localization sequences in $\PrLSt$ upon taking ind-completion.} This is 1-split iff both functors admit right adjoints (i.e.\! it is a \textit{split} Verdier sequence). A 1-split 1-localization sequence
    \[
\begin{tikzcd}[column sep=1.5cm]
        \cA
        \arrow[hook, yshift=0.9ex]{r}{i}
        \arrow[leftarrow, yshift=-0.9ex]{r}[yshift=-0.2ex]{\bot}[swap]{R}
        &
        \cB
        \arrow[yshift=0.9ex]{r}{L}
        \arrow[hookleftarrow, yshift=-0.9ex]{r}[yshift=-0.2ex]{\bot}[swap]{j}
        &
        \cC
        \end{tikzcd}
\]
has a middle 0-split 0-localization sequence $iR \xra{\varepsilon} \id_\cB \xra{\eta} jL$ in $\End_\St(\cB)$.
\end{enumerate}

\paragraph{Primary K-theories.}
\label{flavors.of.primary.K.thys}

Following \cite{BGT}, the discussion of \Cref{pgh.localizn.seqs} (combined with a healthy dose of ``negative thinking'') suggests the consideration of five primary K-theory functors $\iota_1 \St \ra \Spectra$, constructed as follows. For $0 \leq i \leq 4$, we write
\[
\Mot^{[i]} \subseteq \Fun((\iota_1\St^\omega)^\op, \Spectra)
\]
for the full subcategory on those functors that send $i$-split 1-localization sequences in $\St$ to 0-localization sequences in $\Spectra$.\footnote{So by definition, $\Mot^{[1]}$ (resp.\! $\Mot^{[0]}$) is the category of additive (resp.\! localizing) motives of \cite{BGT}.}$^,$\footnote{For diagrammatic reasons, the image of a 1-split (resp.\! 3-split) 1-localization sequence under a functor in $\Mot^{[1]}$ (resp.\! in $\Mot^{[3]}$) will be a 1-split (resp.\! 2-split) 0-localization sequence.} These assemble into a sequence of Bousfield localizations 
\[
\begin{tikzcd}
\Fun ( ( \iota_1 \St^\omega)^\op , \Spectra)
=:
\Mot^{[4]}
\arrow{r}
&
\Mot^{[3]}
\arrow{r}
&
\Mot^{[2]}
\arrow{r}
&
\Mot^{[1]}
\arrow{r}
&
\Mot^{[0]}
\end{tikzcd}
\]
among presentable stable categories. For any small stable category $\cC \in \St$, applying the operation $\hom_{(-)}(\uno_\St,\cC)$ yields the sequence
\[
\Sigma^\infty_+(\iota_0 \cC)
\longra
\Sigma^\infty(\iota_0 \cC)
\longra
\K^\oplus(\cC)
\longra
\K(\cC)
\longra
\KK(\cC)
\]
of K-theory spectra.\footnote{Here, $\iota_0 \cC$ denotes the maximal subgroupoid of $\cC$, and $\K^\oplus$, $\K$, and $\KK$ respectively denote direct-sum, connective, and nonconnective K-theory.}

It is worth remarking that for $1 \leq i \leq 4$, the abelian group $\pi_0 ( \hom_{\Mot^{[i]}}(\uno_\St,\cC) )$ is obtained by splitting the $(i-1)$-split 0-localization sequences in $\cC$. In particular, generalizing Waldhausen's additivity theorem, we find that imposing relations among stable 1-categories corresponds to enforcing relations among the elements of the corresponding $0\th$ K-groups.

\subsubsection{A secondary $S_\bullet$-construction}
\label{subsubsection.S.dot.constrn}

Although primary K-theory can be characterized by a universal property, it was of course first introduced via various explicit constructions -- the most general being Waldhausen's $S_\bullet$-construction \cite{Wald-AKTofspaces}. We would be interested seeing an analogous construction of secondary K-theory. Of course, this may involve additional algebraic structure on a stable 2-category (akin to a Waldhausen structure on a 1-category). In view of \cite[Theorem 7.3.3]{DyckKap-higherSeg}, it is conceivable that this would be a 3-Segal space. We also note the possibility that such a construction would apply only in specific cases; for instance, the subcategory $\iota_1\St^\dzbl \subseteq \iota_1\St^\idem$ of fully dualizable objects carries a Waldhausen structure.\footnote{The cofibrations are the fully faithful functors, which are stable under cobase change by \cite[Lemma 4.18]{HStwo}.} We are optimistic that a secondary $S_\bullet$-construction would relate to an interpretation of secondary K-theory within the context of Goodwillie calculus, following \cite{BarwickAlgKTheoryofHigherCats}.

\subsubsection{The secondary cyclotomic trace}
\label{subsubsection.secondary.cylo.trace}

An essential tool for computing (higher) K-theory is the \textit{cyclotomic trace}, a natural map
\[
\K
\longra
\TC
\]
in $\Fun(\iota_1 \St,\Spectra)$ from K-theory to \emph{topological cyclic homology}. This is a factorization
\[
\begin{tikzcd}[column sep=1.5cm]
\K
\arrow{rr}{\tr_{\text{Dennis}}}
\arrow[dashed]{rd}[swap, sloped, anchor = north]{\tr_{\text{cyclo}}}
&
&
\THH
\\
&
\TC
\arrow{ru}
\end{tikzcd}
\]
of the \textit{Dennis trace} map from K-theory to topological Hochschild homology. By the celebrated Dundas--Goodwillie--McCarthy (DGM) theorem \cite{DGMlocal} (see also \cite{RaskinDGM,elmanto2020nilpotent}), the cyclotomic trace is locally constant along nilpotent extensions. Though complicated, $\TC$ is far more amenable to direct analysis than higher K-groups, and as a result the DGM theorem has led to a plethora of spectacular K-theory computations \cite{HessMad-Witt,HessMad-trunc,HessMad-local,HessMad-dRW,AngGerHess-trunc,AngGerHillLind-trunc,KleinRog-fib,Rog-White,BluMan-AKTS}.

As indicated in \Cref{subsection.intro.background}, a major motivation for the present work is to provide a means of constructing a \bit{secondary cyclotomic trace} map, in order to afford a secondary analog of the DGM theorem for the purpose of computing secondary K-theory.\footnote{We are inspired by the theory of covering homology \cite{BCD-cov,CDD-higherTC}, although note that it only applies to commutative (as opposed to $\EE_2$-)ring spectra.} Here we outline a proposed definition of the target of such a trace map, which we denote by $\TCtwo$ and refer to as \bit{secondary topological cyclic homology}. Analogously to the primary case, this arises from a \textit{secondary Dennis trace} map, whose target we denote by $\THHtwo$ and refer to as \textit{secondary topological Hochschild homology}. We envision \Cref{mainthm:locating.K.2} as providing secondary Dennis and cyclotomic trace maps
\[
\begin{tikzcd}[column sep=1.5cm]
\Ktwo
\arrow{rr}{\trtwo_\text{Dennis}}
\arrow[dashed]{rd}[swap, sloped, anchor = north]{\trtwo_\text{cyclo}}
&
&
\THHtwo
\\
&
\TCtwo
\arrow{ru}
\end{tikzcd}
\]
by Yoneda, analogously to how the corresponding primary trace maps can be deduced from \Cref{thm:bgt-main}.\footnote{The cyclotomic trace was originally constructed by hand, without recourse to a universal property of K-theory. We do not expect that our proposed secondary cyclotomic trace could be similarly constructed by hand, due to the proliferation of symmetries that constitute a secondary cyclotomic structure.} We describe our proposed definition of $\TCtwo$ in \Cref{pgh.TCtwo}, after reviewing the corresponding primary situation in \Cref{pgh.TCone}.

\paragraph{Primary topological cyclic homology.}
\label{pgh.TCone}

The \textit{topological Hochschild homology} of a stable category $\cC \in \St$ is its (spectrally-enriched) factorization homology \cite{AMGRfact} over the circle:
\[
\THH(\cC)
:=
\left( \int_{S^1} \cC \right)
\in
\Spectra
~.
\]
This carries a \textit{cyclotomic structure}: to a first approximation, this consists of an action of the group $\Aut^\fr(S^1) \simeq \TT$ of framed automorphisms of the circle along with, for each degree-$r$ framed self-covering map $S^1 \la S^1$,\footnote{The framings that appear here and below derive from their connection with (higher) category theory: factorization homology is a pairing between $n$-categories and compact framed $n$-manifolds \cite{AFR}.} a corresponding map
\[
\THH(\cC)
\longra
\THH(\cC)^{\uptau \Cyclic_r}
~,
\]
where $(-)^{\uptau \Cyclic_r}$ denotes the \textit{proper $\Cyclic_r$-Tate construction}. Then, the \textit{topological cyclic homology} of $\cC$ is the invariants of this cyclotomic structure:
\[
\TC(\cC)
:=
\THH(\cC)^{\htpy \Cyclo}
~.
\]
While cyclotomic spectra were originally defined in terms of genuine $\TT$-spectra, we note that they can also be defined using the formalism of \textit{stratified categories} \cite{AMGRstrat,AMGRcyclo}.

Given a perfect stack $X$, we have an identification
\[ \THH(X) := \THH(\Perf(X)) \simeq \ms{O}(\ms{L}(X)) \]
of $\THH(X)$ with functions on the free loopspace $\ms{L}X := \hom(S^1,X)$ (a derived mapping stack) \cite{BZFNintegraldrinfeld}. In these terms, the Dennis trace carries a vector bundle to its trace-of-monodromy function on loops in $X$ \cite{TV2009}. In \cite{AMGRtrace}, it is explained that the cyclotomic trace arises from the fact that such functions enjoy subtle relations between their values at a point $(S^1 \xra{\gamma} X) \in \ms{L}X$ and at the corresponding point $(S^1 \xra{r} S^1 \xra{\gamma} X) \in \ms{L}X$, for any degree-$r$ framed self-covering map $S^1 \xra{r} S^1$.\footnote{For instance, when $r=1$ this records the cyclic invariance of traces, and when $r=p$ is prime this records the fact that the $p\th$ power map is a homomorphism mod $p$ (i.e.\! the existence of the Frobenius).}

\paragraph{Secondary topological cyclic homology.}
\label{pgh.TCtwo}

We propose that the \bit{secondary topological Hochschild homology} of a stable 2-category $\cX \in \St_2$ should be its (spectrally-enriched) factorization homology over the torus:
\[
\THHtwo(\cX)
:=
\left( \int_{T^2} \cX \right)
\in 
\Spectra
~.\footnote{At present, enriched factorization homology has only been defined in dimension 1 \cite{AMGRfact}. However, we expect that $\THHtwo$ is equivalent to iterated $\THH$; see \Cref{rmk.The.Big.Triangle}.}
\]
This should correspondingly carry a \bit{secondary cyclotomic structure}: to a first approximation, this should consist of an action of the group $\Aut^\fr(T^2)$ of framed automorphisms of the torus along with, for each framed self-covering map $T^2 \la T^2$, a corresponding map
\[
\THHtwo(\cX)
\longra
\THHtwo(\cX)^{\uptau G}
~,
\]
where $G$ denotes the deck group of the chosen framed self-covering.\footnote{Note that framing-preservation is a homotopical condition, not a point-set one. In particular, a subtlety arises in the 2-dimensional case: whereas it is merely a condition for a self-covering map of $S^1$ to be framing-preserving, it is additional data for a self-covering map of $T^2$ to be framing-preserving. Indeed, there exists a noncontractible space of framing-preservation data on the identity map of $T^2$. (The group of framed diffeomorphisms of $T^2$ is computed in the work \cite{AyaFraHow-Symm}.)} Then, \bit{secondary topological cyclic homology} of $\cX$ should be the invariants of this secondary cyclotomic structure:
\[
\TCtwo(\cX)
:=
\THHtwo(\cX)^{\htpy \Cyclo_2}
~.
\]
We note that $\Aut^\fr(T^2)$ is not (equivalent to) a compact Lie group, and so we do not expect that secondary cyclotomic spectra can be defined in terms of classical genuine equivariant homotopy theory \cite{LMS}. On the other hand, we expect that they can be defined readily via stratified categories.

Given a perfect stack $X$, we expect an equivalence
\[
\THHtwo(X)
:=
\THHtwo(\Perftwo(X))
\simeq
\sco (\scl^2 X)
\]
between our proposed definition of $\THHtwo$ and that of \cite{HStwo} (see also \cite{TV2009, TV-caracteres}), where $\Perftwo(X) \subseteq \QCtwo(X)$ is the sub-2-category of fully dualizable objects. In these terms, we expect that the secondary Dennis trace map carries a perfect sheaf of stable categories to its \textit{secondary}-\textit{trace}-of-monodromy function on double loops in $X$. Namely, given a symmetric monoidal 2-category $\cC$, a fully dualizable object $C \in \cC$, and a pair of commuting endomorphisms $\varphi,\psi \in \End_\cC(C)$, the \bit{secondary trace} is the 2-endomorphism
\[
\trtwo_C(\varphi,\psi)
:=
\left(
\uno_{\End_\cC(\uno_\cC)}
\xra{\tr_{\tr_C(\varphi)}(\tr_\psi(\varphi))}
\uno_{\End_\cC(\uno_\cC)}
\right)
\in
\End_{\End_\cC(\uno_\cC)}(\uno_{\End_\cC(\uno_\cC)})
=:
2\End_\cC(\uno_\cC)
\]
of the unit object of $\cC$ \cite{BZNsecondarytraces}, i.e.\! the trace of the endomorphism 
\[
\tr_C(\varphi)
\xra{\tr_\psi(\varphi)}
\tr_C(\varphi)\]
of the endomorphism
\[
\left( \uno_\cC \xra{\tr_C(\varphi)} \uno_\cC \right)
\in
\End_\cC(\uno_\cC)
~.
\]
Indeed, secondary traces in $\cC = \St$ are scalars (i.e.\! endomorphisms of the sphere spectrum).\footnote{For instance, in the case that $C = (\Sp^\fin)^{\oplus n} \in \St$, the data of a pair of commuting endomorphisms $\varphi,\psi \in \End_\St(C)$ in the special case that $\psi = \id_C$ amounts to an $n$-by-$n$ matrix valued in finite spectra equipped with endomorphisms, in which case the secondary trace is the iterated trace in the evident sense.} Finally, we expect that the secondary cyclotomic trace should arise from analogous relations enjoyed by secondary trace-of-monodromy functions of perfect sheaves of stable categories.

\begin{remark}
\label{rmk.The.Big.Triangle}
We expect that it is straightforward to obtain a secondary Dennis trace map via the commutative diagram \Cref{the.big.triangle},
\begin{figure}[h]
\begin{equation}
\label{the.big.triangle}
\begin{tikzcd}[column sep = 2cm, row sep = 2cm]
\Spectra
\\
\iota_1 \St
\arrow{u}[swap]{\K}[swap, xshift=1cm, yshift=-0.5cm]{\rotatebox{-30}{$\Rightarrow$}}
\arrow[hook, yshift=0.9ex]{r}{i}
\arrow[leftarrow, yshift=-0.9ex]{r}[yshift=-0.2ex]{\top}[swap]{\Env}
&
\Cat(\Spectra)
\arrow{lu}[sloped, anchor = south, swap]{\THH}
\\
\iota_1 \St_2
\arrow{u}[swap,yshift=-0.15cm]{\Kto}[swap, xshift=1.2cm, yshift=-0.15cm]{\Rightarrow}
\arrow[hook]{r}[swap]{j}
\arrow[bend left]{uu}{\Ktwo}
&
\Cat(\iota_1\St)
\arrow[hook]{r}[swap]{k}
\arrow{u}[yshift=-0.15cm]{\Cat(\K)}[swap, xshift=1cm, yshift=-0.6cm]{\rotatebox{-25}{$\Rightarrow$}}
&
\Cat_2(\Spectra)
\arrow{lu}[sloped, anchor = south, swap]{\Cat(\THH)}
\arrow[bend right]{lluu}[swap, sloped, anchor = south]{\THHtwo}
\end{tikzcd}
\end{equation}
\end{figure}
as we now explain.
\begin{itemize}
    \item Given a symmetric monoidal category $\cV$, we write $\Cat_n(\cV)$ for the 1-category of $\cV$-enriched $n$-categories, defined inductively by the formula $\Cat_n(\cV) := \Cat(\Cat_{n-1}(\cV))$. Similarly, given a (laxly) symmetric monoidal functor $\cV \xra{F} \cW$, we write $\Cat(\cV) \xra{\Cat(F)} \Cat(\cW)$ for the induced functor given by applying $F$ to hom-objects.
    
    \item We write $\Cat(\Spectra) \xra{\Env} \iota_1\St$ for the \emph{stable envelope} functor.
    
    \item The identification $\Ktwo \simeq \K \circ \Kto$ is definitional, while the identification $\THHtwo \simeq \THH \circ \Cat(\THH)$ is an instance of an expected pushforward formula for factorization homology of enriched $n$-categories.
    
    \item The natural transformation in the upper-left triangle is the Dennis trace, and the natural transformation in the lower-right triangle is obtained from it by applying $\Cat(-)$.
    
    \item The natural transformation in the square follows from \Cref{mainthm:universal.property.of.K.2.1} and Yoneda, after recognizing the composite
    \[
    \iota_1 \St_2
    \overset{j}{\longhookra}
    \Cat(\iota_1 \St)
    \xra{\Cat(\K)}
    \Cat(\Spectra)
    \xra{\Env}
    \iota_1 \St
    \]
    as a $(2,1)$-additive invariant (see \Cref{define.2.1.additive.invt}).
    
\end{itemize}
Namely, we obtain a secondary Dennis trace $\Ktwo \ra \THHtwo$ as the composite
\[
\Ktwo
\simeq
\K \circ \Kto
\longra
\THH \circ i \circ \Env \circ \Cat(\THH) \circ k \circ j
\xlongla{\sim}
\THH \circ \Cat(\THH) \circ k \circ j
\simeq \THHtwo \circ k \circ j
~,
\]
in which the backwards equivalence follows from the Morita-invariance of $\THH$. We note that this construction relies crucially on the decomposition $T^2 \cong S^1 \times S^1$, and thus we only can deduce the equivariance of the secondary Dennis trace for a small subgroup of $\Aut^\fr(T^2)$ (to say nothing of framed self-covering maps $T^2 \la T^2$).
\end{remark}

\subsection{Foundational results on enriched and higher categories}
\label{subsection.intro.cat.thy}

Our main results (as described in \Cref{subsection.intro.main.results}) are supported by a number of auxiliary results on enriched and higher categories, which we hope to be of independent interest. In \Cref{app.some.enriched.category.theory}, we study some aspects of enriched category theory: co/limits and adjunctions, presentability, compact generation, and semiadditivity. Thereafter, in \Cref{app.some.higher.cat.theory}, we connect with (unenriched) higher category theory: primarily, we perform a consistency check regarding two possible definitions of the 3-category $\St_2$ of small stable 2-categories (see \Cref{rem.only.one.Cat.St.in.section.one.point.one} and \Cref{prop.only.one.Cat.St}), which leads to its closed symmetric monoidal structure.

We briefly discuss our notion of enriched presentability. Given a presentably symmetric monoidal category $\cV$, we define \bit{presentable $\cV$-categories} in terms of $\cV$-enriched category theory (\Cref{def.presentable.V.category}). However, we also prove that these can be described in unenriched terms: namely, we establish an equivalence of categories
\[
\iota_1\PrL_\cV 
\simeq
\Mod_\cV(\iota_1\PrL)
\]
between that of presentable $\cV$-categories and that of $\cV$-modules in $\iota_1\PrL$ (\Cref{thm.presble.V.cats.are.V.mods}). Taking $\cV = \iota_1 \Cat_{n-1}$ to be the category of $(n-1)$-categories, we obtain a category
\[
\iota_1 \PrL_n
:=
\iota_1 \PrL_{\iota_1 \Cat_{n-1}}
\]
of \bit{presentable $n$-categories}. Moreover, we prove that for any small 2-category $\cB \in \Cat_2$, the 2-categories $\Fun(\cB,\Cat)$ and $\Fun(\cB,\St)$ are presentable (Propositions \ref{prop.Fun.B.Cat.is.prbl.and.tensoring.is.ptwise} and \ref{prop.Fun.K.St.is.presentable}).

\begin{remark}
In the recent work \cite{Stef-Presentable} (which appeared as the present paper was nearing completion), Stefanich proposes a markedly different notion of ``presentable $n$-category'', which makes much greater use of set theory: writing $n\PrL$ for the $(n+1)$-category of his presentable $n$-categories, the fundamental example is that $(n-1)\PrL \in n\PrL$.\footnote{We propose referring to this latter notion as ``$n$-presentability''.} By contrast, our presentable $n$-categories live in the same set-theoretic universe for all $n \geq 0$ (and in particular a presentable $(n+1)$-category has an underlying presentable $n$-category). Indeed, our assignment $\cV \mapsto \iota_1 \PrL_\cV$ assembles as a functor
\[
\CAlg(\iota_1 \PrL)
\longra
\iota_1 \what{\what{\Cat}}^{{\sf l.adjt}}
\]
to the massive category whose objects are huge categories and whose morphisms are left adjoint functors, so that the sequence
\[
\Spaces
\simeq
\Cat_0
\longhookra
\iota_1 \Cat
\longhookra
\iota_1 \Cat_2
\longhookra
\cdots
\]
in $\CAlg(\iota_1 \PrL)$ determines a diagram of adjunctions
\[ \begin{tikzcd}[column sep=2cm]
\iota_1 \PrL
\simeq
\iota_1 \PrL_1
\arrow[hook, yshift=0.9ex]{r}
\arrow[leftarrow, yshift=-0.9ex]{r}[yshift=-0.2ex]{\bot}[swap]{\iota_1}
&
\iota_1 \PrL_2
\arrow[hook, yshift=0.9ex]{r}
\arrow[leftarrow, yshift=-0.9ex]{r}[yshift=-0.2ex]{\bot}[swap]{\iota_2}
&
\iota_1 \PrL_3
\arrow[hook, yshift=0.9ex]{r}
\arrow[leftarrow, yshift=-0.9ex]{r}[yshift=-0.2ex]{\bot}[swap]{\iota_3}
&
\cdots
\end{tikzcd} \]
among huge categories.
\end{remark}

\subsection{Outline}
\label{subsection.intro.outline}

This paper is organized as follows.

\begin{itemize}
\item[\Cref{sec.stable.2.cats.and.loc.seqs}:] We discuss stable 2-categories and 1- and 2-localization sequences.
    
\item[\Cref{sec.locating.k.theory.in.noncommutative.motives}:] We study motives and additive invariants, and we prove \Cref{mainthm.homs.in.Mot.of.X.are.K.theory.are.K.theory}.

\item[\Cref{sec.locating.secondary.k.theory.in.motives}:] We study 2- and (2,1)-motives and 2- and (2,1)-additive invariants, and we prove Theorems \ref{mainthm:universal.property.of.K.2.1} and \ref{mainthm:locating.K.2}.

\item[\Cref{app.some.enriched.category.theory}:] We study some aspects of enriched category theory.

\item[\Cref{app.some.higher.cat.theory}:] We perform consistency checks regarding the 2- and 3-categories of stable 1- and 2-categories. 

\end{itemize}

\subsection{Notation and conventions}
\label{subsection.intro.notn}

Our notation and conventions are laid out carefully in \Cref{subsec.foundations.conventions}. Although they are largely standard, here we highlight a few notable points.
\begin{itemize}

\item As indicated in \Cref{subsection.overview}, we take the ``implicit $\infty$'' convention (see \Cref{conv.implicit.infty}).

\item Broadly speaking, notation should be interpreted in the most enriched sense possible (see \Cref{convention.as.enriched.as.possible}). For instance, $\St$ denotes the 2-category of small stable categories; we write $\iota_1 \St$ to denote its underlying 1-category.

\item Although we consider co/limits in enriched and higher categories, we only ever consider those that are indexed over 1-categories (see \Cref{def:enriched.limits.and.colimits} and \Cref{rmk.only.conical}).

\end{itemize}

\subsection{Acknowledgments}
\label{subsection.intro.ack}

It is our pleasure to thank Ben Antieau, David Ayala, Alex Campbell, Andrea Gagna, David Gepner, Yonatan Harpaz, Rune Haugseng, Marc Hoyois, Grigory Kondyrev, Edoardo Lanari, Andrew Macpherson, and Nick Rozenblyum for many illuminating conversations related to the material in this paper. We also thank the anonymous referees for their careful reading and helpful feedback. Additionally, AMG gratefully acknowledges the support of a Zumberge grant from the University of Southern California and RS gratefully acknowledges the support of a Herchel Smith Fellowship from Harvard University.

\section{Stable 2-categories and localization sequences}
\label{sec.stable.2.cats.and.loc.seqs}

In this section, we introduce some of our fundamental objects of study: stable 2-categories (\Cref{subsec.stable.2.cats}), 1-localization sequences (\Cref{subsection.one.loc.seqs}), and 2-localization sequences (\Cref{subsection.two.loc.seqs}). We provide a number of examples, and record some basic results that we use in later sections.

\subsection{Stable 2-categories}
\label{subsec.stable.2.cats}

In this subsection, we introduce stable 2-categories, and we record some basic features of the 3-category $\St_2$ of small stable 2-categories (\Cref{lem.omnibus.in.section.one.point.one}).

\begin{definition}
\label{defn.stably.enriched.two.cat}
A \bit{stably-enriched 2-category} is a 2-category $\cX$ whose hom-categories are stable and whose composition bifunctors are biexact. Given stably-enriched 2-categories $\cX$ and $\cY$, we write
\[
\Fun^\twoex(\cX, \cY) \subseteq \Fun(\cX, \cY)
\]
for the full sub-2-category on those functors that are exact on hom-categories, which we refer to interchangeably as \bit{2-exact} or \bit{stably-enriched}. We write
\[
\Cat(\St) \subseteq \Cat_2
\]
for the 1-full sub-3-category (see \Cref{defn.k.full.sub.n.cat}) on the small stably-enriched 2-categories and the 2-exact functors among them.
\end{definition}

\begin{remark}
\label{rem.only.one.Cat.St.in.section.one.point.one}
By \Cref{obs.stably.enriched.is.unambiguous.on.iota.one}, there is a canonical equivalence
\[
\iota_1\Cat(\St) \simeq \Cat(\iota_1\St)
\]
of 1-categories (where that on the left arises from \Cref{defn.stably.enriched.two.cat} and that on the right is an instance of \Cref{notn.V.enr.cats}). Moreover, by \Cref{prop.only.one.Cat.St} this is compatible with $\iota_1 \Cat_2$-enrichments, i.e.\! it upgrades to an equivalence
\[
\Cat(\St) \simeq \Cat(\iota_1\St)^{\iota_1\Cat_2\enr}
\]
of 3-categories (where that on the right follows from the self-enrichment of $\Cat(\iota_1\St)$ (and makes use of \Cref{notn.change.enrichment})).
\end{remark}

\begin{definition}
\label{def:stable.2.category}
A \bit{stable 2-category} is a semiadditive stably-enriched 2-category.\footnote{We discuss semiadditivity for enriched categories in \Cref{subsec.semiadditive.V.categories}.} In other words, a stable 2-category is a stably-enriched 2-category that admits finite sums (see \Cref{obs.products.and.coproducts.in.stably.enriched.two.cat}). We write
\[
\St_2
\subseteq
\Cat(\St)
\]
for the full sub-3-category on the small stable 2-categories.
\end{definition}

\needspace{2\baselineskip}
\begin{example}
\label{examples.of.st.2}
	\begin{enumerate}
	\item[]
		\item\label{item.St.is.a.stable.two.cat} The motivating example of a stable 2-category is $\St$. 
		\item\label{item.StR.is.a.stable.two.cat} More generally, for any $\EE_2$-ring spectrum $R$, the 2-category $\St_R$ of small stable $R$-linear categories is stable.
		\item More generally, for any stably monoidal category $\cA \in \Alg(\iota_1 \St)$, the 2-category $\LMod_\cA(\St)$ of left $\cA$-modules is stable.
		\item\label{item.ShvCat.is.a.stable.two.cat} For any scheme $X$, the 2-category
	    \[ \QC^{(2)}(X) := \underset{\Spec(R) \ra X}{\lim} \St_R \]
	    of quasicoherent sheaves of stable categories over $X$ is stable.\footnote{In the case where $X$ is 1-affine (in the sense of \cite{gaitsgory2013sheaves}), there is an equivalence $\QC^{(2)}(X) \simeq \LMod_{\QC(X)}(\PrLSt)$.}
		\item\label{item.local.systems.of.stable.cats.is.a.stable.two.cat} 
		For any 2-category $\cC$ and any stable 2-category $\cX$, the functor 2-category $\Fun(\cC,\cX)$ is stable (\Cref{obs.consequences.of.Env.V}). In particular, constructible sheaves of stable categories over a stratified space (and in particular, local systems of stable categories over a space) assemble into a stable 2-category.
\end{enumerate}
\end{example}

\begin{lemma}
\label{lem.omnibus.in.section.one.point.one}
\begin{enumerate}
\item[]

\item\label{omnibus.in.one.point.one.Env.two}

The defining inclusion admits a left adjoint
\[
\begin{tikzcd}[column sep=1.5cm]
\Cat(\St)
\arrow[dashed, yshift=0.9ex]{r}{\Env_2}
\arrow[hookleftarrow,yshift=-0.9ex]{r}[yshift=-0.2ex]{\bot}
&
\St_2
\end{tikzcd}
~.
\]

\item\label{omnibus.in.one.point.one.symm.mon}

The bifunctor $\Fun^\twoex(-,-)$ defines a self-enrichment of both $\Cat(\St)$ and $\St_2$. Moreover, this is adjoint to a symmetric monoidal structure in $\Cat(\St)$, as well as one in $\St_2$ given by the formula
\[
\cX \stackrel{\St_2}{\otimes} \cY
:=
\Env_2 \left( \cX \stackrel{\Cat(\St)}{\otimes} \cY \right)
~.
\]
In particular, the left adjoint $\Env_2$ is symmetric monoidal.

\item\label{omnibus.in.one.point.one.cgsm}

The 3-categories $\Cat(\St)$ and $\St_2$ are both compactly-generatedly symmetric monoidal, and the left adjoint $\Env_2$ preserves compact objects.\footnote{We discuss compactly generated (and compactly-generatedly symmetric monoidal) enriched categories in  \Cref{subsection.cg.V.cats}.}

\item\label{omnibus.in.one.point.one.semiadd}

The 3-category $\St_2$ is semiadditive.

\end{enumerate}
\end{lemma}

\begin{proof}
All of these claims follow from general assertions proved in the appendices, applied in the case that $\cV = \iota_1 \St$: part \Cref{omnibus.in.one.point.one.Env.two} follows from \Cref{prop.Env.V} and \Cref{obs.consequences.of.Env.V} (using \Cref{obs.St.CGSM.and.get.iota.1.St.iota.1.Cat.enr});\footnote{In the notation of \Cref{def.env.v}, we have $\Env_2 := \Env^\oplus_{\iota_1 \St}$.} part \Cref{omnibus.in.one.point.one.symm.mon} follows from \Cref{prop.only.one.Cat.St} and \Cref{obs.consequences.of.Env.V}; part \Cref{omnibus.in.one.point.one.cgsm} follows from Observations \ref{obs.consequences.of.Env.V} and \ref{obs.Cat.V.semiadd.is.compactly.generated} (again using \Cref{obs.St.CGSM.and.get.iota.1.St.iota.1.Cat.enr}); and part \Cref{omnibus.in.one.point.one.semiadd} follows from Observations \ref{obs.St.is.semiadd}, \ref{obs.cat.V.semiadd.is.semiadd}, \ref{obs.iota.one.Cat.St.CGSM.and.get.iota.one.Cat.St.iota.one.cat.two.enr}, and \ref{obs.change.of.enr.takes.sadd.to.sadd}.
\end{proof}

\begin{definition}
\label{defn.Env.two}
The left adjoint $\Env_2$ of \Cref{lem.omnibus.in.section.one.point.one}\Cref{omnibus.in.one.point.one.Env.two} is called the \bit{stable envelope}.\footnote{It follows from parts \Cref{omnibus.in.one.point.one.Env.two}-\Cref{omnibus.in.one.point.one.cgsm} of \Cref{lem.omnibus.in.section.one.point.one} that $\Env_2$ defines a morphism in $\CAlg(\iota_1 \PrL_{3,\omega})$ (in fact in $\CAlg(\iota_1 \PrL_{\iota_1\Cat(\St),\omega})$).}
\end{definition}

\begin{remark}
\label{rmk.sigma.infty.n.plus}
For all $0 \leq n \leq 2$, the forgetful functor $\St_n \ra \Cat_n$ has a left adjoint. We propose that this should be denoted
\[
\Cat_n
\xra{\Sigma^{(\infty,n)}_+}
\St_n~,
\]
so that the base case is
\[
\Sigma^{(\infty,0)}_+
:
\Cat_0
:=
\Spaces
\xra{\Sigma^\infty_+}
\Spectra
=:
\St_0
~.
\]
For $n=1$, this is the functor
\[
\Sigma^{(\infty,1)}_+
:
\Cat_1
\xra{\cP_\Spectra^\fin}
\St_1
\]
of \Cref{obs.St.CGSM.and.get.iota.1.St.iota.1.Cat.enr}, while for $n=2$ this is the composite functor
\[
\Sigma^{(\infty,2)}_+
:
\Cat_2
\xra{\Cat\left( \Sigma^{(\infty,1)}_+ \right)}
\Cat(\St)
\xra{\Env_2}
\St_2
\]
in which the first functors arises from \Cref{obs.iota.one.Cat.St.CGSM.and.get.iota.one.Cat.St.iota.one.cat.two.enr}.\footnote{In analogy with $\Sigma^{(\infty,2)}_+$, the functor $\Sigma^{(\infty,1)}_+$ can also be written as the composite
\[
\Cat_1
\xra{\Cat \left( \Sigma^{(\infty,0)}_+ \right) }
\Cat(\St_0)
\xra{\Env_1}
\St_1~,
\]
where $\Env_1$ denotes the stable envelope functor.}
However, we do not use this notation here: we are proving foundational results that refer directly to the definitions, and so we suspect that it would be distracting.
\end{remark}

\subsection{1-localization sequences}
\label{subsection.one.loc.seqs}

In this subsection, we discuss 1-localization sequences in a stable 2-category. Notably, we prove the existence of a stable 2-category $\oneLoc$ that corepresents 1-localization sequences (\Cref{lem.one.Loc}).

\begin{definition}
\label{def.0.loc.seq}
In the interest of uniformity, we may use the term \bit{0-localization sequence} to refer to an exact sequence in a stable category (i.e.\! a sequence of morphisms equipped with a nullhomotopy of the composite that is simultaneously a fiber sequence and a cofiber sequence).
\end{definition}

\begin{definition}
	\label{def:localization.sequence}
A \bit{1-localization sequence} in a stable 2-category $\cX$ is a diagram
\begin{equation}
\label{prototypical.one.loc.seq}
		\begin{tikzcd}[column sep=1.5cm]
        A
        \arrow[hook, yshift=0.9ex]{r}{i}
        \arrow[leftarrow, yshift=-0.9ex]{r}[yshift=-0.2ex]{\bot}[swap]{R}
        &
        B
        \arrow[yshift=0.9ex]{r}{L}
        \arrow[hookleftarrow, yshift=-0.9ex]{r}[yshift=-0.2ex]{\bot}[swap]{j}
        &
        C
        \end{tikzcd}
\end{equation}
such that the following conditions hold.
\begin{enumerate}

\item\label{condition.ff}

The morphisms $\id_A \xra{\eta} Ri$ and $Lj \xra{\varepsilon} \id_C$ are equivalences.

\item\label{condition.compose.to.zero}

The object $Li \in \hom_\cX(A,C)$ is zero (or equivalently the object $Rj \in \hom_\cX(C,A)$ is zero).

\item\label{condition.middle.zero.loc}

The commutative square
		\[ \begin{tikzcd}
iR
\arrow{r}{\varepsilon}
\arrow{d}
&
\id_B
\arrow{d}{\eta}
\\
0
\arrow{r}
&
jL
\end{tikzcd} \]
resulting from condition \Cref{condition.compose.to.zero} is a 0-localization sequence in $\End_\cX(B)$.
\end{enumerate}
We refer to the 0-localization sequence $iR \ra \id_B \ra jL$ as the \bit{middle 0-localization sequence} associated to the given 1-localization sequence.
\end{definition}

\begin{remark}
\label{rmk.one.loc.seq.just.iL}
It is merely a condition for a pair of composable morphisms $A \xra{i} B \xra{L} C$ in a stable 2-category $\cX$ to extend to a 1-localization sequence.
This is in contrast to the case of 0-localization sequences, where one must provide the datum of a nullhomotopy.
\end{remark}

\begin{remark}
\label{rmk.oneloc.sequence.more.like.exact.seq.of.abgrps.than.zeroloc.seq.of.spectra}
Every morphism in a stable category extends to a 0-localization sequence, simply by taking its cofiber. By contrast, an arbitrary morphism in a stable 2-category $\cX$ need not extend to a 1-localization sequence (even if $\cX$ admits cofibers, e.g.\! in the case that $\cX = \St$). In this sense, 1-localization sequences in a stable 2-category are somewhat akin to exact sequences in an abelian category: given a morphism in an abelian category, taking its cokernel gives an exact sequence if and only if the original morphism was a monomorphism.
\end{remark}

\begin{remark}
A 1-localization sequence is both a fiber sequence and a cofiber sequence. In particular, in the case that $\cX = \St$, one may think of a 1-localization sequence \Cref{prototypical.one.loc.seq} as being specified by the ``extension data'' of the bimodule
\[
\cC^\op
\times
\cA
\xra{\Omega^{\infty-1} \hom_\cB ( j(-), i (=))}
\Spaces
~,
\]
which classifies the bifibration
\[
\cB
\xra{(R,L)}
\cA \times \cC
~.\footnote{Alternatively, we may recover $\cB$ as the compact objects in $\Ind(\cB)$, which participates in a recollement from which it may be recovered as the lax limit of the functor $\Ind(\cA) \xra{L_! R^*} \Ind(\cC)$.}
\]
\end{remark}

\begin{observation}
\label{obs.one.loc.in.St.iff.both.directions.fiber}
The data of a diagram \Cref{prototypical.one.loc.seq} in $\St$ define a 1-localization sequence if and only if both of the sequences $\cA \xhookrightarrow{i} \cB \xra{L} \cC$ and $\cA \xla{R} \cB \xhookleftarrow{j} \cC$ are fiber sequences in $\St$. Indeed, in this case the middle 0-localization sequence follows from the resulting joint conservativity of the functors $\cA \xla{R} \cB$ and $\cB \xra{L} \cC$.
\end{observation}

\begin{remark}
\label{rmk.macrocosm.versions.of.middle.zero.loc}
Suppose that we are given a diagram \Cref{prototypical.one.loc.seq} in $\St$ such that conditions \Cref{condition.ff} and \Cref{condition.compose.to.zero} of \Cref{def:localization.sequence} both hold. Then, condition \Cref{condition.middle.zero.loc} is equivalent to any of the following conditions.
\begin{itemize}

\item[(3a)] The sequence $\cA \xhookrightarrow{i} \cB \xra{L} \cC$ is a cofiber sequence.

\item[(3a$'$)] The subcategory $j(\cC) \subseteq \cB$ is the right-orthogonal of the subcategory $i(\cA) \subseteq \cB$. 

\item[(3b)] The sequence $\cA \xla{R} \cB \xhookleftarrow{j} \cC$ is a cofiber sequence.

\item[(3b$'$)] The subcategory $i(\cA) \subseteq \cB$ is the left-orthogonal of the subcategory $j(\cC) \subseteq \cB$.

\end{itemize}
In particular, a 1-localization sequence \Cref{prototypical.one.loc.seq} determines and is determined by either the reflective localization $L \adj j$ (because we have $\cA \simeq \fib(L)$ and $iR \simeq \fib(\id_\cB \xra{\eta} jL)$) or the coreflective localization $i \adj R$ (because we have $\cC \simeq \fib(R)$ and $jL \simeq \cofib(iR \xra{\varepsilon} \id_\cB)$).
\end{remark}

\begin{example}
\label{examples.of.one.loc.sequences}
	\begin{enumerate}
	\item[]
	\item Let $\cC \in \St$ be a small stable category, and let us write $\zeroLoc(\cC) \subseteq \Fun([1]\times[1],\cC)$ for the stable category of 0-localization sequences in $\cC$. Then, the prototypical example of a 1-localization sequence is the diagram
	\[ \begin{tikzcd}[column sep=3cm]
	\cC
\arrow[hook, yshift=0.9ex]{r}{X \longmapsto (X \ra X \ra 0)}
\arrow[leftarrow, yshift=-0.9ex]{r}[yshift=-0.2ex]{\bot}[swap]{X \longmapsfrom (X \ra Y \ra Z)}
&
\zeroLoc(\cC)
\arrow[yshift=0.9ex]{r}{(X \ra Y \ra Z) \longmapsto Z}
\arrow[hookleftarrow, yshift=-0.9ex]{r}[yshift=-0.2ex]{\bot}[swap]{(0 \ra Z \ra Z) \longmapsfrom Z}
&
\cC
\end{tikzcd} \]
in $\St$ (where we have omitted the requisite nullhomotopies for typographical convenience).
	
		\item\label{ex.one.loc.seq.from.closed.open.decomp.of.qcqs.scheme}
		
		Let $X$ be a qcqs scheme. Given a closed-open decomposition $Z \hookrightarrow X \hookleftarrow U$, we obtain a 1-localization sequence
		\[
		\begin{tikzcd}[column sep=1.5cm]
        \QC_Z(X)
        \arrow[hook, yshift=0.9ex]{r}
        \arrow[leftarrow, yshift=-0.9ex]{r}[yshift=-0.2ex]{\bot}
        &
        \QC(X)
        \arrow[yshift=0.9ex]{r}
        \arrow[hookleftarrow, yshift=-0.9ex]{r}[yshift=-0.2ex]{\bot}
        &
        \QC(U)
        \end{tikzcd}
		\]
		among (large) stable categories, where $\QC_Z(X)\subseteq \QC(X)$ denotes the category of quasicoherent sheaves set-theoretically supported on $Z$.\footnote{On compact objects, this generally only restricts to a 0-split 1-localization sequence $\Perf_Z(X) \hookrightarrow \Perf(X) \ra \Perf(U)$ in $\St$ in the sense of \Cref{subsubsection.flavors.of.Kthy}.}
		\item Let $R$ be a ring. There is a 1-localization sequence
		\[
		\begin{tikzcd}[column sep=1.5cm]
        \mathbf{D}(R)
        \arrow[hook, yshift=0.9ex]{r}{}
        \arrow[leftarrow, yshift=-0.9ex]{r}[yshift=-0.2ex]{\bot}[swap]{}
        &
        \mathbf{K}(R)
        \arrow[yshift=0.9ex]{r}{}
        \arrow[hookleftarrow, yshift=-0.9ex]{r}[yshift=-0.2ex]{\bot}[swap]{}
        &
        \mathbf{A}(R)
        \end{tikzcd}
		\]
		among (large) stable categories: the derived category of $R$, the category of complexes over $R$ (a.k.a.\! the homotopy category of $R$), and the category of acyclic complexes over $R$.
		\item Let $p$ be a prime, and let $\cat{C}_p$ denote the cyclic group of order $p$.
		There is a 1-localization sequence
		\[
		\begin{tikzcd}[column sep=1.5cm]
        \spectra^{\cat{hC}_p}
        \arrow[hook,yshift=0.9ex]{r}
        \arrow[leftarrow, yshift=-0.9ex]{r}[yshift=-0.2ex]{\bot}[swap]{\fgt}
        &
        \spectra^{\cat{gC}_p}
        \arrow[yshift=0.9ex]{r}{\Phi^{\sC_p}}
        \arrow[hookleftarrow, yshift=-0.9ex]{r}[yshift=-0.2ex]{\bot}[swap]{}
        &
        \spectra
        \end{tikzcd}
		\]
		among (large) stable categories, where $\spectra^{\cat{hC}_p}:= \Fun(\sB\sC_p,\Sp)$ denotes the category of homotopy $\cat{C}_p$-spectra and $\spectra^{\cat{gC}_p}$ denotes the category of genuine $\cat{C}_p$-spectra.
		
		\item Given a closed-open decomposition $Z \hookrightarrow X \hookleftarrow U$ among qcqs schemes, there is a 1-localization sequence 
		\[
		\begin{tikzcd}[column sep=1.5cm]
        \SH(U)
        \arrow[hook,yshift=0.9ex]{r}
        \arrow[leftarrow, yshift=-0.9ex]{r}[yshift=-0.2ex]{\bot}
        &
        \SH(X)
        \arrow[yshift=0.9ex]{r}
        \arrow[hookleftarrow, yshift=-0.9ex]{r}[yshift=-0.2ex]{\bot}[swap]{}
        &
        \SH(Z)
        \end{tikzcd}
		\]
		among (large) motivic stable homotopy categories.
	\end{enumerate}
\end{example}

\begin{lemma}
\label{lem.one.Loc}
There is a small stable 2-category $\oneLoc \in \St_2$ that corepresents 1-localization sequences.  Moreover, $\oneLoc \in \St_2^\omega \subseteq \St_2$ is a compact object.
\end{lemma}

\begin{proof}
Consider the sequence
\begin{equation}
\label{sequence.of.adjns.from.Cat.two.to.St.two}
\begin{tikzcd}[column sep=1.5cm]
\Cat_2
\arrow[yshift=0.9ex]{r}{\Cat(\cP_\Spectra^\fin)}
\arrow[hookleftarrow, yshift=-0.9ex]{r}[yshift=-0.2ex]{\bot}[swap]{\fgt}
&
\Cat(\St)
\arrow[yshift=0.9ex]{r}{\Env_2}
\arrow[hookleftarrow, yshift=-0.9ex]{r}[yshift=-0.2ex]{\bot}[swap]{\fgt}
&
\St_2
\end{tikzcd}
\end{equation}
of adjunctions (the former following from \Cref{obs.iota.one.Cat.St.CGSM.and.get.iota.one.Cat.St.iota.one.cat.two.enr} and the latter following from \Cref{lem.omnibus.in.section.one.point.one}\Cref{omnibus.in.one.point.one.Env.two}). We construct the object $\oneLoc \in \St_2$ in steps, which we simply describe in words because the notation would be complicated and unenlightening.
\begin{enumerate}

\item We begin with the object
\[
\Adj^{\vee 2}
\in
\Cat_2
\]
that corepresents pairs of composable adjunctions. For simplicity, we use the notation of \Cref{prototypical.one.loc.seq} when referring to data in $\Adj^{\vee 2}$.

\item We invert the 2-morphisms $\id_\cA \xra{\eta} Ri$ and $Lj \xra{\varepsilon} \id_\cC$ in $\Adj^{\vee 2}$ (which are 2-endomorphisms of $\cA$ and $\cC$, respectively).

\item We apply $\Cat(\cP_\Spectra^\fin)$.

\item We invert the 2-morphism $Li \ra 0$ (a 1-morphism in $\hom(\cA,\cC)$).

\item We invert the 2-morphism $\cofib(iR \ra \id_\cB) \ra jL$ (a 2-endomorphism of $\cB$).

\item We apply $\Env_2$.

\end{enumerate}
The resulting stable 2-category is the desired object $\oneLoc$, which is compact as a consequence of the following observations.
\begin{enumerate}

\item Both left adjoints in diagram \Cref{sequence.of.adjns.from.Cat.two.to.St.two} preserve compact objects.

\item Inverting a 2-morphism in a 2-category amounts to taking a pushout along the morphism $c_2 \ra c_1$ (where we write $c_i \in \Cat_2$ for the walking $i$-cell for all $0 \leq i \leq 2$), while inverting a 2-morphism in a stably-enriched 2-category amounts to taking a pushout along the morphism $\cP_\Spectra^\fin(c_2 \ra c_1)$.

\item The objects $c_0,c_1,c_2,\Adj \in \Cat_2$ are compact. \qedhere

\end{enumerate}
\end{proof}

\begin{notation}
Given a stable 2-category $\cX \in \St_2$, we write
\[
\oneLoc(\cX)
:=
\Fun^\twoex(\oneLoc,\cX)
\in
\St_2
\]
for the stable 2-category of 1-localization sequences in $\cX$.
\end{notation}

\begin{observation}
\label{obs.describe.oneLoc.of.a.stable.two.cat}
It follows from the construction of $\oneLoc \in \St_2$ that for any stable 2-category $\cX \in \St_2$ we have a full inclusion
\[
\oneLoc(\cX)
:=
\Fun^\twoex(\oneLoc,\cX)
\subseteq
\Fun ( \Adj^{\vee 2}, \cX)
\]
among stable 2-categories (using the notation $\Adj^{\vee 2} \in \Cat_2$ introduced in the proof of \Cref{lem.one.Loc}).
\end{observation}

\begin{remark}
As a particular consequence of \Cref{obs.describe.oneLoc.of.a.stable.two.cat}, a 1-morphism in $\oneLoc(\cX)$ is given by a diagram
\[
\begin{tikzcd}[column sep=1.5cm, row sep = 1.5cm]
A
\arrow[hook, yshift=0.9ex]{r}{i}
\arrow[leftarrow, yshift=-0.9ex]{r}[yshift=-0.2ex]{\bot}[swap]{R}
\arrow{d}[swap]{F}
&
B
\arrow[yshift=0.9ex]{r}{L}
\arrow[hookleftarrow, yshift=-0.9ex]{r}[yshift=-0.2ex]{\bot}[swap]{j}
\arrow{d}[swap]{G}
&
C
\arrow{d}{H}
\\
A'
\arrow[hook, yshift=0.9ex]{r}{i'}
\arrow[leftarrow, yshift=-0.9ex]{r}[yshift=-0.2ex]{\bot}[swap]{R'}
&
B'
\arrow[yshift=0.9ex]{r}{L'}
\arrow[hookleftarrow, yshift=-0.9ex]{r}[yshift=-0.2ex]{\bot}[swap]{j'}
&
C'
\end{tikzcd}
\]
in $\cX$ that commutes upon omitting all left adjoints and commutes upon omitting all right adjoints. Note that the 1-morphism $G$ will necessarily interleave the respective middle 0-localization sequences: the functors
\[
\End_\cX(B)
:=
\hom_\cX(B,B)
\xra{G \circ -}
\hom_\cX(B,B')
\xla{- \circ G}
\hom_\cX(B',B')
=:
\End_\cX(B')
\]
respectively carry the two middle 0-localization sequences
\[
iR
\longra
\id_B
\longra
jL
\qquad
\text{and}
\qquad
i'R'
\longra
\id_{B'}
\longra
j'L'
\]
to the single 0-localization sequence
\[
(GiR \longra G \longra GjL)
\simeq
(i' F R \longra G \longra j' H L)
\simeq
(i'R'G \longra G \longra j'L'G)
~.
\]
\end{remark}

\begin{remark}
Using the adjoint functor theorem (and the fact that $\iota_1 \St_2$ is presentable by \Cref{lem.omnibus.in.section.one.point.one}\Cref{omnibus.in.one.point.one.cgsm}), one could prove the existence and compactness of an object $\oneLoc' \in \iota_1 \St_2$ such that
\[
\hom_{\iota_1 \St_2}(\oneLoc',\cX)
\simeq
\iota_0 \Fun^\twoex(\oneLoc',\cX)
\]
is the space of 1-localization sequences in $\cX$. However, this would not enable an explicit description of the 2-categorical structure of $\Fun^\twoex(\oneLoc',\cX)$ such as that given in \Cref{obs.describe.oneLoc.of.a.stable.two.cat}.
\end{remark}

\begin{observation}
\label{one.loc.sequences.in.PrLSt.and.PrRSt}
Via the equivalence $\PrLSt \simeq (\PrRSt)^\onetwoop$, the data of a 1-localization sequence
\[
		\begin{tikzcd}[column sep=1.5cm]
        \cA
        \arrow[hook, yshift=0.9ex]{r}{i}
        \arrow[leftarrow, yshift=-0.9ex]{r}[yshift=-0.2ex]{\bot}[swap]{R}
        &
        \cB
        \arrow[yshift=0.9ex]{r}{L}
        \arrow[hookleftarrow, yshift=-0.9ex]{r}[yshift=-0.2ex]{\bot}[swap]{j}
        &
        \cC
        \end{tikzcd}
\]
in $\PrLSt$ are equivalent to those of the corresponding 1-localization sequence
\[
		\begin{tikzcd}[column sep=1.5cm]
        \cC
        \arrow[hook, yshift=0.9ex]{r}{L^R}
        \arrow[leftarrow, yshift=-0.9ex]{r}[yshift=-0.2ex]{\bot}[swap]{j^R}
        &
        \cB
        \arrow[yshift=0.9ex]{r}{i^R}
        \arrow[hookleftarrow, yshift=-0.9ex]{r}[yshift=-0.2ex]{\bot}[swap]{R^R}
        &
        \cA
        \end{tikzcd}
\]
in $\PrRSt$.
\end{observation}

\begin{observation}
\label{check.one.loc.in.larger.stable.two.cat}
If $\cX \subseteq \cY$ is the inclusion of a 0-full or 1-full sub-stable-2-category, then a sequence of adjunctions in $\cX$ defines a 1-localization sequence if and only if it does so when considered in $\cY$.
\end{observation}

\subsection{2-localization sequences}
\label{subsection.two.loc.seqs}

In this subsection, we discuss 2-localization sequences among stable 2-categories. 

\begin{definition}
	\label{def:2.localization.sequence}
A \bit{2-localization sequence} is a diagram
\begin{equation}
\label{prototypical.two.loc.seq}
		\begin{tikzcd}[column sep=1.5cm]
        \cX
        \arrow[hook, yshift=0.9ex]{r}{i}
        \arrow[leftarrow, yshift=-0.9ex]{r}[yshift=-0.2ex]{\bot}[swap]{R}
        &
        \cY
        \arrow[yshift=0.9ex]{r}{L}
        \arrow[hookleftarrow, yshift=-0.9ex]{r}[yshift=-0.2ex]{\bot}[swap]{j}
        &
        \cZ
        \end{tikzcd}
\end{equation}
in $\St_2$ satisfying the following conditions.
\begin{enumerate}

\item\label{two.loc.seq.ff}

The morphisms $\id_\cX \xra{\eta} Ri$ and $Lj \xra{\varepsilon} \id_\cZ$ are equivalences.

\item\label{two.loc.seq.compose.to.zero}

the object $Li \in \hom_{\St_2}(\cX,\cZ)$ is zero (or equivalently the object $Rj \in \hom_{\St_2}(\cZ,\cX)$ is zero).

\item\label{two.loc.seq.middle.one.loc.seq}

The commutative square
		\[ \begin{tikzcd}
iR
\arrow{r}{\varepsilon}
\arrow{d}
&
\id_\cY
\arrow{d}{\eta}
\\
0
\arrow{r}
&
jL
\end{tikzcd} \]
resulting from condition \Cref{two.loc.seq.compose.to.zero} is a 1-localization sequence in $\End_{\St_2}(\cY)$.
\end{enumerate}
We refer to the 1-localization sequence $iR \to \id_\cY \to jL$ as the \bit{middle 1-localization sequence} associated to the given 2-localization sequence.
\end{definition}

\begin{remark}
	As with 1-localization sequences, it is merely a condition for a pair of composable morphisms $\cX \xra{i} \cY \xra{L} \cZ$ in $\St_2$ to extend to a 2-localization sequence.
\end{remark}

\begin{remark}
\label{rmk.characterize.two.loc.seq}
A 2-localization sequence \Cref{prototypical.two.loc.seq} determines and is determined by the reflective localization $L \adj j$ satisfying the condition that the unit morphism $\id_\cY \xra{\eta} jL$ admits a fully faithful right adjoint. In turn, this is equivalent to requiring that for all $A,B \in \cY$, the functor
\[
\hom_\cY(A,B)
\longra
\hom_\cY(A,jL B)
\simeq
\hom_\cZ ( L A , L B)
\]
admits a fully faithful right adjoint.\footnote{Given a more robust definition of stable 2-category, one should be able to obtain $\cX$ simply as the kernel of $\cY \xra{L} \cZ$, with the right adjoint $R$ given by $R := \fib(\id_\cY \xra{\eta} jL)$. But even in the present context, one can obtain $\cX$ as the univalent-completion of the flagged $\St$-enriched category (i.e.\! categorical $\St$-algebra in the sense of \cite{GH}) with the same space of objects as $\cY$ and with $\hom_\cX(A,B) := \fib(\hom_\cY(A,B) \ra \hom_\cZ(LA,LB))$; now, the right adjoint $R$ arises from the corresponding functors in (what will be) the middle 1-localization sequence.} Of course, dual remarks apply to the coreflective localization $i \adj R$.
\end{remark}

\begin{example}
\label{ex.of.two.locs}
\begin{enumerate}
\item[]
\item\label{ex.two.loc.seq.involving.one.loc}
Let $\cX \in \St_2$ be a small stable 2-category, and let us write $\oneLoc(\cX) \subseteq \Fun([2],\cX)$ for the stable 2-category of 1-localization sequences in $\cX$.  Then, the prototypical example of a 2-localization sequence is the diagram
\begin{equation}
\label{x.to.1Loc.x.to.x.is.2.loc}
\begin{tikzcd}[column sep=3cm]
	\cX
\arrow[hook, yshift=0.9ex]{r}{A \longmapsto (A \adjarr A \adjarr 0)}
\arrow[leftarrow, yshift=-0.9ex]{r}[yshift=-0.2ex]{\bot}[swap]{A \longmapsfrom (A \adjarr B \adjarr C)}
&
\oneLoc(\cX)
\arrow[yshift=0.9ex]{r}{(A \adjarr B \adjarr C) \longmapsto C}
\arrow[hookleftarrow, yshift=-0.9ex]{r}[yshift=-0.2ex]{\bot}[swap]{(0 \adjarr C \adjarr C) \longmapsfrom C}
&
\cX
\end{tikzcd}
\end{equation}
in $\St_2$ (recall \Cref{obs.describe.oneLoc.of.a.stable.two.cat}).

\item Let $X$ be a qcqs scheme. Using \cite[Theorem 2.1.1]{gaitsgory2013sheaves}, we can build on \Cref{examples.of.one.loc.sequences}\Cref{ex.one.loc.seq.from.closed.open.decomp.of.qcqs.scheme} to obtain a 2-localization sequence
\[
		\begin{tikzcd}[column sep=2.5cm]
        \QCtwo_Z(X)
        \arrow[hook, yshift=0.9ex]{r}
        \arrow[leftarrow, yshift=-0.9ex]{r}[yshift=-0.2ex]{\bot}[swap]{-\otimes^L_{\QC(X)} \QC_Z(X)}
        &
        \QCtwo(X)
        \arrow[yshift=0.9ex]{r}{- \otimes^L_{\QC(X)} \QC(U)}
        \arrow[hookleftarrow, yshift=-0.9ex]{r}[yshift=-0.2ex]{\bot}
        &
        \QCtwo(U)
        \end{tikzcd}
		\]
among (large) stable 2-categories, where we write $\otimes^L$ for the tensor product in $\iota_1\PrL$ and we define
\[
\QCtwo_Z(X) := \Mod_{\QC_Z(X)}(\iota_1\PrLSt) \subseteq \Mod_{\QC(X)}(\iota_1\PrLSt) \simeq \QC^{(2)}(X)
\]
to be the 2-category of quasicoherent sheaves of stable categories set-theoretically supported on $Z$.
\end{enumerate}
\end{example}

\begin{remark}
In \Cref{ex.of.two.locs}\Cref{ex.two.loc.seq.involving.one.loc}, it is necessary to restrict to the 1-localization sequences in $\cX$ in order to obtain a 2-localization sequence. By contrast, e.g.\! taking $\cX = \St$, the diagram
\[
\begin{tikzcd}[column sep=3cm]
	\St
\arrow[hook, yshift=0.9ex]{r}{\cA \longmapsto (\cA \ra \cA)}
\arrow[leftarrow, yshift=-0.9ex]{r}[yshift=-0.2ex]{\bot}[swap]{\cA \longmapsfrom (\cA \ra \cB)}
&
\Fun([1],\St)
\arrow[yshift=0.9ex]{r}{(\cA \ra \cB) \longmapsto \cB/\cA}
\arrow[hookleftarrow, yshift=-0.9ex]{r}[yshift=-0.2ex]{\bot}[swap]{(0 \ra \cC) \longmapsfrom \cC}
&
\St
\end{tikzcd} 
\]
is \textit{not} a 2-localization sequence, since it fails to satisfy condition \Cref{two.loc.seq.middle.one.loc.seq} of \Cref{def:2.localization.sequence}: given an object $(\cA \ra \cB) \in \Fun([1],\St)$, the commutative square
\[ \begin{tikzcd}
(\cA \ra \cA)
\arrow{r}
\arrow{d}
&
(\cA \ra \cB)
\arrow{d}
\\
(0 \ra 0)
\arrow{r}
&
(\cB \ra \cB/\cA)
\end{tikzcd} \]
is not generally a 1-localization sequence in $\Fun([1],\St)$, since the sequence $\cA \ra \cB \ra \cB/\cA$ is not generally a 1-localization sequence in $\St$ (recall \Cref{rmk.oneloc.sequence.more.like.exact.seq.of.abgrps.than.zeroloc.seq.of.spectra}).
\end{remark}

\begin{observation}
\label{obs.two.loc.seqs.preserve.cpctness}
A 2-localization sequence \Cref{prototypical.two.loc.seq} determines a retraction
\[
\cX \times \cZ
\xra{i \oplus j}
\cY
\xra{(R,L)}
\cX \times \cZ
~.
\]
In particular, using \Cref{lem.omnibus.in.section.one.point.one}\Cref{omnibus.in.one.point.one.semiadd} and the fact that finite coproducts of compact objects are compact, we see that if $\cX$ and $\cZ$ are compact in $\St_2$ then so is $\cY$ (note \Cref{lem.compact.in.prbl.V.cat.detected.in.underlying}).
\end{observation}

\section{From motives to K-theory}
\label{sec.locating.k.theory.in.noncommutative.motives}

We begin this section by introducing motives over a stable 2-category in \Cref{subsec.motives}. We briefly discuss additive invariants in this context in \Cref{subsec.additive.invariants}, and then proceed to prove \Cref{mainthm.homs.in.Mot.of.X.are.K.theory.are.K.theory} (that hom-spectra in categories of motives are K-theory spectra) in \Cref{subsec.from.motives.to.K.theory}.

\subsection{Motives}
\label{subsec.motives}

In this subsection, we study the construction assigning to each stable 2-category its stable category of motives. In particular, we show that this assembles as a functor to $\PrLSt_\omega$ (Observations \ref{obs.Mot.is.a.two.functor} and \ref{obs.Mot.lands.in.PrLSt.omega}).

\begin{definition}
\label{defn.motives.over.a.stable.two.cat}
Let $\cX \in \St_2$ be a small stable 2-category.
We define the category of (\bit{noncommutative}) \bit{motives over $\cX$} to be the full subcategory
\[
\Mot(\cX) \subseteq \Fun(\iota_1 \cX^\op, \Spectra)
\]
on the functors that preserve zero objects and take 1-localization sequences to 0-localization sequences.  More precisely, a functor
\[
\iota_1 \cX^\op
\xlongra{\cM}
\Spectra
\]
is a motive if $\cM(0_\cX^\circ) \simeq 0_\Spectra$ and for every 1-localization sequence
\[
		\begin{tikzcd}[column sep=1.5cm]
        A
        \arrow[hook, yshift=0.9ex]{r}{i}
        \arrow[leftarrow, yshift=-0.9ex]{r}[yshift=-0.2ex]{\bot}[swap]{R}
        &
        B
        \arrow[yshift=0.9ex]{r}{L}
        \arrow[hookleftarrow, yshift=-0.9ex]{r}[yshift=-0.2ex]{\bot}[swap]{j}
        &
        C
        \end{tikzcd}
\]
in $\cX$, the square
\[
\cM \left(
\begin{tikzcd}[sep = 1.5cm]
A
\arrow{r}{i}
\arrow{d}
&
B
\arrow{d}{L}
\\
0_\cX
\arrow{r}
&
C
\end{tikzcd}
\right)^\circ
=
\left(
\begin{tikzcd}[sep = 1.5cm]
\cM(A^\circ)
&
\cM(B^\circ)
\arrow{l}[swap]{\cM(i^\circ)}
\\
\cM(0_\cX^\circ)
\arrow{u}
&
\cM(C^\circ)
\arrow{l}
\arrow{u}[swap]{\cM(L^\circ)}
\end{tikzcd}
\right)
\]
is a 0-localization sequence in $\Spectra$.\footnote{Equivalently, one may check that the square
\[ \begin{tikzcd}[ampersand replacement=\&, column sep=1.2cm]
\cM(A^\circ)
\arrow{r}{\cM(R^\circ)}
\arrow{d}
\&
\cM(B^\circ)
\arrow{d}{\cM(j^\circ)}
\\
\cM(0_\cX^\circ)
\arrow{r}
\&
\cM(C^\circ)
\end{tikzcd} \]
defines a 0-localization sequence in $\Spectra$.}
\end{definition}

\begin{remark}
Because $\iota_1 \St^\idem$ is compactly generated, we can use \cite[Proposition 5.5(2)]{HStwo} to see that $\Mot((\St^\idem)^\omega)$ is equivalent to the presentable stable category of additive motives of \cite[Definition 6.1]{BGT}.
One could also define \textit{localizing motives} over $\cX$ to be the full subcategory of $\Fun(\iota_1 \cX^\op, \Spectra)$ on those spectral presheaves that carry all bifiber sequences (i.e.\! 0-split 1-localization sequences in the sense of \Cref{subsubsection.flavors.of.Kthy}) to 0-localization sequences.
\end{remark}

\begin{notation}
Let $\cX \in \St_2$ be a small stable 2-category.  We write
\[ \begin{tikzcd}[column sep=1.5cm]
        \Fun(\iota_1 \cX^\op , \Spectra)
        \arrow[dashed, yshift=0.9ex]{r}{L_\cX}
        \arrow[hookleftarrow, yshift=-0.9ex]{r}[yshift=-0.2ex]{\bot}[swap]{R_\cX}
        &
        \Mot(\cX)
        \end{tikzcd} \]
for the adjunction in which the right adjoint is the defining inclusion and the left adjoint exists by the adjoint functor theorem.\footnote{Note that filtered colimits preserve $0_\Spectra$ as well as 0-localization sequences in $\Spectra$, so the inclusion $R_\cX$ is indeed accessible.}
\end{notation}

\begin{notation}
We write
\[
\univ_\cX
:
\iota_1 \cX
\overset{\Yo}{\longhookra}
\Fun( \iota_1 \cX^\op , \Spaces )
\xra{\Sigma^\infty_+}
\Fun( \iota_1 \cX^\op , \Spectra )
\xra{L_\cX}
\Mot(\cX)
\]
for the composite functor.
\end{notation}

\begin{observation}
\label{obs.Mot.is.a.two.functor}
The construction $\cX \mapsto \Mot(\cX)$ assembles as a functor
\[
\iota_2 \St_2
\xra{\Mot}
\PrLSt
~,
\]
as we now explain.  First of all, given a 2-exact functor $\cX \xra{F} \cY$, we obtain an adjunction
\[ \begin{tikzcd}[column sep=2cm]
        \Fun(\iota_1 \cX^\op , \Spectra)
        \arrow[yshift=0.9ex]{r}{(\iota_1 F^\op)_!}
        \arrow[leftarrow, yshift=-0.9ex]{r}[yshift=-0.2ex]{\bot}[swap]{(\iota_1 F^\op)^*}
        &
        \Fun(\iota_1 \cY^\op , \Spectra)
        \end{tikzcd} \]
on spectral presheaves.  By the 2-exactness of $F$, there exists a factorization
\begin{equation}
\label{diagram.of.radjts.between.spectra.pshvs.and.Mots}
\begin{tikzcd}[column sep=2cm, row sep=1.5cm]
\Fun(\iota_1 \cX^\op , \Spectra)
\arrow[leftarrow]{r}{(\iota_1 F^\op)^*}
&
\Fun(\iota_1 \cY^\op , \Spectra)
\\
\Mot(\cX)
\arrow[hook]{u}{R_\cX}
\arrow[leftarrow, dashed]{r}[swap]{\Mot^*(F)}
&
\Mot(\cY)
\arrow[hook]{u}[swap]{R_\cY}
\end{tikzcd}~.
\end{equation}
Writing $\Cat \xra{\cP_\Spectra^*} (\PrRSt)^\onetwoop$ for the functor of spectral presheaves with respect to pullback, these factorizations assemble to give us a diagram
\[ \begin{tikzcd}
\iota_2 \St_2
\arrow{r}{\iota_1}
\arrow[bend right]{rr}[swap]{\Mot^*}[yshift=0.2cm]{\Uparrow R_\bullet}
\arrow[bend right=45]{rrr}[swap]{\Mot}
&
\Cat
\arrow{r}{\cP_\Spectra^*}
\arrow[bend left=35]{rr}{\cP_\Spectra}
&
(\PrRSt)^\onetwoop
\arrow[leftrightarrow]{r}{\sim}
&
\PrLSt
\end{tikzcd}
~,\footnote{That is, $\Mot^*$ is a subfunctor of the composite functor $\iota_2 \St_2 \xra{\iota_1} \Cat \xra{\cP_\Spectra^*} (\PrRSt)^\onetwoop$:
\begin{itemize}
\item
its value on an object $\cX \in \St_2$ is the presentable stable category $\Mot(\cX)$,
\item
its value on a 1-morphism $\cX \xra{F} \cY$ in $\St_2$ is the factorization $\Mot^*(F)$ of diagram \Cref{diagram.of.radjts.between.spectra.pshvs.and.Mots}, and
\item
its value on a 2-morphism
\[ \begin{tikzcd}[ampersand replacement=\&]
\cX
\arrow[bend left]{r}[swap, yshift=-0.1cm]{\Downarrow}
\arrow[bend right]{r}
\&
\cY
\end{tikzcd} \]
in $\St_2$ is simply the restriction of the value of $\cP_\Spectra^* \circ \iota_1$ thereon (which exists because $R_\cX$ is fully faithful).
\end{itemize}
}
\]
in which we define $\Mot$ as the indicated composite. In particular, the functoriality of $\Mot$ on 1-morphisms in $\St_2$ is given by
\[
\Mot(F)
:=
\Mot^*(F)^L
\simeq
L_\cY \circ (\iota_1 F^\op)_! \circ R_\cX
~.
\]
\end{observation}

\begin{observation}
\label{obs.Mot.lands.in.PrLSt.omega}
All four right adjoints in the commutative square \Cref{diagram.of.radjts.between.spectra.pshvs.and.Mots} preserve filtered colimits.  Hence, there exists a factorization
\[
\begin{tikzcd}[column sep=1.5cm]
\iota_2 \St_2
\arrow{r}{\Mot}
\arrow[dashed]{rd}[sloped, anchor=north, swap]{\Mot}
&
\PrLSt
\\
&
\PrLSt_\omega
\arrow[hook]{u}
\end{tikzcd}
~.
\]
\end{observation}

\begin{definition}
\label{def.2.1.ary.k.theory}
The \bit{(2,1)-ary K-theory} of a small stable 2-category $\cX$ is
\[
\Kto(\cX)
:=
\Mot(\cX)^\omega
~.
\]
This assembles as the composite functor
\[
\Kto
:
\iota_2\St_2
\xra{\Mot}
\PrLSt_\omega
\xra[\sim]{(-)^\omega}
\St^\idem
~,
\]
in which the first functor is that of \Cref{obs.Mot.lands.in.PrLSt.omega}.
\end{definition}

\subsection{Additive invariants}
\label{subsec.additive.invariants}

In this brief subsection, we discuss the notion of additive invariants of a stable 2-category.

\begin{definition}
    \label{def:additive.invariant.of.X}
    Let $\cX \in \St_2$ be a small stable 2-category and let $\cT$ be a presentable stable category.
    A functor
    \[
    \iota_1 \cX \xlongra{F} \cT
    \]
    is called a \defn{$\cT$-valued additive invariant of $\cX$} if 
    \begin{itemize}
        \item it preserves zero objects, and
        \item it takes 1-localization sequences in $\cX$ to 0-localization sequences in $\cT$.
    \end{itemize}
    We denote by 
    \[ 
    \Add_\cT(\cX) \subseteq \Fun(\iota_1 \cX, \cT)
    \]
    the full subcategory on the additive invariants.
\end{definition}

\begin{example}
\label{K.theory.is.an.additive.invariant.of.St}
    The K-theory functor 
    \[ \iota_1 \St \xlongra{\K} \Spectra\]
    is a $\Spectra$-valued additive invariant of $\St$.
\end{example}

\begin{remark}
\label{rmk.U.X.is.univ.additive.invt}
The functor
\[
\iota_1 \cX
\xra{\univ_\cX}
\Mot(\cX)
\]
is the universal additive invariant of $\cX$: for any presentable stable category $\cT$, restriction along $\univ_\cX$ defines an equivalence
\[
\Fun^L(\Mot(\cX),\cT)
\xlongra{\sim}
\Add_\cT(\cX)
~.
\]
\end{remark}

\subsection{From motives to K-theory}
\label{subsec.from.motives.to.K.theory}

In this subsection, we prove \Cref{mainthm.homs.in.Mot.of.X.are.K.theory.are.K.theory}.

\begin{theorem}
	\label{thm.homs.in.Mot.of.X.are.K.theory}
	Let $\mathcal{X} \in \St_2$ be a small stable 2-category, and let $A,B \in \cX$.  Then, the functor
	\[ \iota_1\cX \xra{\univ_\cX} \Mot(\cX) \]
	determines a canonical equivalence
	\[
	\K(\hom_{\cX}(A, B)) \simeq \hom_{\Mot(\cX)}(\univ_\cX A, \univ_\cX B)
	\]
	of spectra.
\end{theorem}

\begin{remark}
\label{rem.that.thm.homs.in.Mot.of.X.are.K.theory.can.be.coherent}
We expect that \Cref{thm.homs.in.Mot.of.X.are.K.theory} can be upgraded to be more homotopy coherent, at two levels.  First of all, one can ask for naturality in the objects $A,B \in \cX$: this would amount to an equivalence between flagged spectral categories with fixed space of objects $\iota_0 \cX$.  Thereafter, one can ask for naturality in the variable $\cX$: this would amount to enhancing the previous equivalence to a suitable functor $\iota_1 \St_2 \ra \Fun([1],\fCat(\spectra))$ that takes each $\cX \in \St_2$ to the previous equivalence.
\end{remark}

The following proof is patterned on the argument given in \cite[\S 5.4]{HStwo}, which is itself closely patterned on the proof of \cite[Theorem 7.13]{BGT}.

\begin{proof}[Proof of \Cref{thm.homs.in.Mot.of.X.are.K.theory}]
We begin by extending the usual $S_\bullet$-construction by defining the 2-exact functor
\[
S_\bullet B
:
\cX^\op
\xra{\hom_\cX(-,B)}
\St
\xlongra{S_\bullet}
\Fun(\bDelta^\op,\St)
~. \]
We also introduce the spectral presheaf
\[
\K_B
:
\iota_1 \cX^\op
\xra{\hom_\cX(-,\cB)}
\iota_1 \St
\xlongra{\K}
\Spectra
~.
\]
It is immediate that $\K_B \in \Mot(\cX) \subseteq \Fun((\iota_1 \cX^\omega)^\op,\Spectra)$, as $\hom_\cX(-,B)$ preserves zero objects and 1-localization sequences and $\iota_1\St \xra{\K} \Spectra$ is an additive invariant (\Cref{K.theory.is.an.additive.invariant.of.St}).

We claim that there is a canonical equivalence $\K_B \simeq \univ_\cX B$ in $\Mot(\cX)$.
To see this, consider the evident 1-localization sequence
\[
\const(B)
\longra
{{\sf{Path}}}(S_\bullet B)
\longra
S_\bullet B
\]
in $\Fun(\bDelta^\op,\Fun^{\twoex}(\cX^\op,\St))$, where ${{\sf Path}}(-)$ denotes the simplicial path object of a pointed simplicial object.  This is recorded by a functor $([1] \times [1]) \times \bDelta^\op \ra \Fun^{\twoex}(\cX^\op,\St)$, which admits a factorization
\begin{equation}
\label{factorization.through.iotaone.Funtwoex.Xop.St}
\begin{tikzcd}
{([1] \times [1]) \times \bDelta^\op}
\arrow{r}
\arrow[dashed]{rd}
&
\Fun^{\twoex}(\cX^\op,\St)
\\
&
\iota_1\Fun^{\twoex}(\cX^\op,\St)
\arrow[hook]{u}
\end{tikzcd} ~.
\end{equation}
Thereafter, we obtain the composite functor
\[ \begin{tikzcd}[column sep=2cm]
{([1] \times [1]) \times \bDelta^\op}
\arrow{r}{\Cref{factorization.through.iotaone.Funtwoex.Xop.St}}
&
\iota_1\Fun^{\twoex}(\cX^\op,\St)
\arrow{r}
\ar[draw=none]{d}[name=X, anchor=center]{}
&
\Fun(\iota_1 \cX^\op, \iota_1 \St)
\ar[rounded corners,
            to path={ -- ([xshift=2ex]\tikztostart.east)
                      |- (X.center) \tikztonodes
                      -| ([xshift=-2ex]\tikztotarget.west)
                      -- (\tikztotarget)}]{dll}[at end, swap]{\iota_0}
\\[1.5cm]
\Fun(\iota_1 \cX^\op , \Spaces)
\arrow{r}{\Sigma^\infty_+}
&
\Fun(\iota_1 \cX^\op , \Spectra)
\arrow{r}{L_\cX}
&
\Mot(\cX)
\\[0.5cm]
\iota_1 \cX
\arrow[hook]{u}{\Yo}
\arrow[bend right=5]{rru}[swap, sloped, anchor=north]{\univ_\cX}
\end{tikzcd}
\]
(where we include $\univ_\cX$ in the diagram for emphasis), which yields the composite functor
\[
[1] \times [1]
\longra
\Fun(\bDelta^\op,\Mot(\cX))
\xra{|-|}
\Mot(\cX)
~,
\]
which selects an exact sequence that we denote by
\begin{equation}
\label{exact.seq.on.geom.realizns.in.Mot.X}
\begin{tikzcd}
{|\univ_\cX(\const(B))|}
\arrow{r}
\arrow{d}
&
{|\univ_\cX({{\sf Path}}(S_\bullet B))|}
\arrow{d}
\\
0
\arrow{r}
&
{|\univ_\cX (S_\bullet B) |}
\end{tikzcd}
\end{equation}
in a mild abuse of notation (because the constituents of $S_\bullet B$ (and of ${{\sf Path}}(S_\bullet B)$) need not be representable).  Noting that $|\univ_\cX({\sf Path}(S_\bullet B))| \simeq 0$ (by an extra degeneracy argument) and that $|\univ_\cX(\const(B))| \simeq \univ_\cX(B)$, we find that the exact sequence \Cref{exact.seq.on.geom.realizns.in.Mot.X} records an equivalence
\[
\Sigma (\univ_\cX B) \simeq |\univ_\cX(S_\bullet B)|
\]
in $\Mot(\cX)$. From here, the equivalence $\K_B \simeq \univ_\cX B$ in $\Mot(\cX)$ follows from an identical argument to the one given in the proof of \cite[Lemma 5.22]{HStwo}.  Finally, we obtain the sequence of equivalences
\begin{align*}
\hom_{\Mot(\cX)}(\univ_\cX A , \univ_\cX B)
&\simeq
\hom_{\Fun(\iota_1\cX^\op,\Spectra)}(\Sigma^\infty_+ \Yo(A) , \univ_\cX B)
\\
&\simeq
\hom_{\Fun(\iota_1\cX^\op,\Spectra)}(\Sigma^\infty_+ \Yo(A) , \K_B)
\\
&\simeq
\K_B(A)
\\
&:=
\K(\hom_\cX(A,B))~,
\end{align*}
as desired.
\end{proof}

\section{From 2-motives to secondary K-theory}
\label{sec.locating.secondary.k.theory.in.motives}

In this section, we prove our two main results: \Cref{mainthm:universal.property.of.K.2.1} (that the hom-categories among (2,1)-motives are (2,1)-ary K-theory categories) in \Cref{subsec.from.2.1.motives.to.2.1.ary.K.theory} and \Cref{mainthm:locating.K.2} (that hom-spectra among 2-motives are secondary K-theory spectra) in \Cref{subsec.from.2.motives.to.secondary.K.theory}. Leading up to these, we study (2,1)-motives in \Cref{subsec.2.1.motives} and (2,1)-additive invariants in \Cref{subsec.2.1.additive.invariants}.

\subsection{(2,1)-motives}
\label{subsec.2.1.motives}

In this subsection, we introduce the stable 2-category of (2,1)-motives and establish a localization functor thereonto (\Cref{lem.get.L.two.one}).

\begin{definition}
\label{def.2.1.motives}
The (large) stable 2-category of \bit{(2,1)-motives} is the full sub-2-category
\begin{equation}
\label{inclusion.of.two.one.motives.into.functors}
\Mot_{2,1} \subseteq \Fun ( ( \iota_2 \St_2^\omega)^\op , \St )
\end{equation}
on those functors that preserve zero objects, carry 2-localization sequences to 1-localization sequences, and take values in the full sub-2-category $\St^\idem \subseteq \St$ on the idempotent-complete stable categories.
\end{definition}

\begin{remark}
The requirement that (2,1)-motives take values in \emph{idempotent-complete} stable categories is used crucially in the proof of \Cref{thm:universal.property.of.K.2.1}. (In the notation there, it affords the extension of $\alpha_\cX$ along $v_\cX$.)    
\end{remark}

\begin{lemma}
\label{lem.get.L.two.one}
The fully faithful inclusion \Cref{inclusion.of.two.one.motives.into.functors} admits a left adjoint
\[ \begin{tikzcd}[column sep=1.5cm]
        \Fun ( ( \iota_2 \St_2^\omega)^\op , \St )
        \arrow[dashed, yshift=0.9ex]{r}{L_{2,1}}
        \arrow[hookleftarrow, yshift=-0.9ex]{r}[yshift=-0.2ex]{\bot}[swap]{i}
        &
        \Mot_{2,1}
        \end{tikzcd}~. \]
Moreover, the left adjoint $L_{2,1}$ preserves compact objects.
\end{lemma}

\begin{proof}
We first show that the functor $i$ is a right adjoint using \Cref{lem.G.a.right.adjoint.between.V.categories.iff}.

For this, we first show that $\iota_1 i$ is a right adjoint. First of all, it follows from \Cref{prop.Fun.K.St.is.presentable} and the fact that $\iota_2\St_2^\omega$ is small (by \Cref{lem.omnibus.in.section.one.point.one}\Cref{omnibus.in.one.point.one.cgsm}) that $\iota_1 \Fun( (\iota_2 \St_2^\omega)^\op , \St)$ is presentable (and in fact that $\Fun( (\iota_2 \St_2^\omega)^\op , \St)$ is a presentable 2-category). Let us write $\Mot_{2,1}' \subseteq \Fun ( ( \iota_2 \St_2^\omega)^\op , \St)$ for the full sub-2-category on those $\St$-valued presheaves that preserve zero objects and carry 2-localization sequences to 1-localization sequences.  Because zero objects and 1-localization sequences are preserved by both functors in the adjunction
\[ \begin{tikzcd}[column sep=1.5cm]
\St
\arrow[yshift=0.9ex]{r}{(-)^\idem}
\arrow[hookleftarrow, yshift=-0.9ex]{r}[yshift=-0.2ex]{\bot}[swap]{\fgt}
&
\St^\idem
\end{tikzcd}
~,
\]
the adjunction obtained by applying $\Fun((\iota_2\St_2^\omega)^\op,-)$ restricts to the adjunction on the right in the diagram
\[ \begin{tikzcd}[column sep=1.5cm]
\iota_1 \Fun( (\iota_2 \St_2^\omega)^\op , \St)
\arrow[dashed, yshift=0.9ex]{r}
\arrow[hookleftarrow, yshift=-0.9ex]{r}[yshift=-0.2ex]{\bot}[swap]{\fgt}
&
\iota_1\Mot_{2,1}'
\arrow[yshift=0.9ex]{r}
\arrow[hookleftarrow, yshift=-0.9ex]{r}[yshift=-0.2ex]{\bot}[swap]{\fgt}
&
\iota_1\Mot_{2,1}
\end{tikzcd}~,
\]
in which it therefore remains to construct the dashed left adjoint. We claim that it is in fact a Bousfield localization. More precisely, using the notation of the diagram
\[ \begin{tikzcd}[column sep=1.5cm]
\iota_2 \St_2^\omega
\arrow[hook]{r}{\Yo}
&
\Fun ( ( \iota_2 \St_2^\omega )^\op , \Cat )
\arrow[yshift=0.9ex]{r}{\cP_\Spectra^\fin}
\arrow[hookleftarrow, yshift=-0.9ex]{r}[yshift=-0.2ex]{\bot}[swap]{\fgt}
&
\Fun ( ( \iota_2 \St_2^\omega )^\op , \St )
\end{tikzcd} ~, \]
we claim that the objects of $\Mot_{2,1}'$ are precisely those that are local with respect to the following set of morphisms.
\begin{itemize}

\item First of all, we have the morphism
\[
\cP_\Spectra^{\fin}( \Yo ( 0_{\St_2} ) )
\longra
\const_{0_\St}
~.
\]

\item Second of all, since $\iota_2 \St_2^\omega$ is small, so is $\Fun ( [2],\iota_2\St_2^\omega)$, and so the 2-localization sequences in $\iota_2 \St_2^\omega$ form a (small) set. For each element of this set, using the notation from \Cref{prototypical.two.loc.seq} as usual, we form the commutative squares
\[
\begin{tikzcd}[row sep = 1.5cm, column sep = 1.5cm]
\cP_\Spectra^\fin(\Yo(\cX))
\arrow{r}{\cP_\Spectra^\fin(\Yo(i))}
\arrow{d}
&
\cP_\Spectra^\fin(\Yo(\cY))
\arrow{d}{\cP_\Spectra^\fin(\Yo(L))}
\\
\cP_\Spectra^\fin(\Yo(0_{\St_2}))
\arrow{r}
&
\cP_\Spectra^\fin(\Yo(\cZ))
\end{tikzcd}
\qquad
\textup{and}
\qquad
\begin{tikzcd}[row sep = 1.5cm, column sep = 1.5cm]
\cP_\Spectra^\fin(\Yo(\cX))
&
\cP_\Spectra^\fin(\Yo(\cY))
\arrow{l}[swap]{\cP_\Spectra^\fin(\Yo(R))}
\\
\cP_\Spectra^\fin(\Yo(0_{\St_2}))
\arrow{u}
&
\cP_\Spectra^\fin(\Yo(\cZ))
\arrow{l}
\arrow{u}[swap]{\cP_\Spectra^\fin(\Yo(j))}
\end{tikzcd}
\]
and take the canonical morphisms to their terminal terms from the pushouts of their remaining terms.

\end{itemize}
It is clear that being local with respect to the first morphism means that $0_{\St_2}$ is taken to $0_\St$. Thereafter, by \Cref{obs.one.loc.in.St.iff.both.directions.fiber}, being local with respect to the other morphisms is equivalent to carrying 2-localization sequences to 1-localization sequences.

In order to show that $i$ is a right adjoint, it remains to verify that $\Mot_{2,1}$ admits cotensors over all small categories and that $i$ commutes with these. First of all, by combining Propositions \ref{prop.Fun.K.St.is.presentable} and \ref{prop.Fun.B.Cat.is.prbl.and.tensoring.is.ptwise} with Lemmas \ref{lem.G.a.right.adjoint.between.V.categories.iff}\Cref{part.radjt.lem.G.a.right.adjoint.between.V.categories.iff} and \ref{lem.presble.V.cats.admit.cotensors}, we find that $\Fun((\iota_2 \St_2^\omega)^\op,\St)$ admits cotensors over any $\cK \in \Cat$, computed by applying $\Fun(\cK,-)$ pointwise. Moreover, this preserves the zero object, idempotent-completeness, and 1-localization sequences (the last by \Cref{obs.one.loc.in.St.iff.both.directions.fiber}). The claim now follows from the fact that $\Mot_{2,1} \subseteq \Fun((\iota_2\St_2^\omega)^\op,\St)$ is a full sub-2-category.

Now, observe that the functor $\iota_1 i$ preserves filtered colimits. It follows that the adjunction $\iota_1 L_{2,1} \adj \iota_1 i$ is an $\omega$-accessible localization of a presentable category. Thereafter, it follows from \Cref{lem.compact.in.prbl.V.cat.detected.in.underlying} that the left adjoint $L_{2,1}$ preserves compact objects.
\end{proof}

\subsection{(2,1)-additive invariants}
\label{subsec.2.1.additive.invariants}

In this subsection, we define (2,1)-additive invariants and prove that the functors $\Mot(-)$ and $\Kto(-)$ are such (\Cref{lem.Mot.and.K.2.1.add.invar}).

\begin{definition}
\label{define.2.1.additive.invt}
Let $\cT$ be a presentable stable 2-category.  A functor
\[
\iota_2 \St_2
\xlongra{F}
\cT
\]
is called a \bit{$\cT$-valued (2,1)-additive invariant} if
\begin{itemize}
\item it preserves zero objects,
\item it takes 2-localization sequences in $\St_2$ to 1-localization sequences in $\cT$, and
\item it preserves filtered colimits.
\end{itemize}
\end{definition}

\begin{lemma}
\label{lem.Mot.and.K.2.1.add.invar}
The noncommutative motives functor 
\[
\iota_2\St_2 
\xra{\Mot}
\PrLSt_\omega
\]
is a $\PrLSt_\omega$-valued (2,1)-additive invariant, and hence the (2,1)-ary K-theory functor
\[
\iota_2 \St_2
\xra{\Kto}
\St^\idem
\]
is a $\St^\idem$-valued (2,1)-additive invariant.
\end{lemma}

\begin{proof}
It is clear that $\Mot$ preserves zero objects.  To see that it also preserves filtered colimits, note that all functors in the composite
\[
\iota_2 \St_2
\xlongra{\iota_1}
\Cat
\xra{\cP_\Spectra^\fin}
\St
\xra{\Ind}
\PrLSt_\omega
\]
preserve filtered colimits (using \Cref{lem.equiv.conds.for.presentable.n.cat}\Cref{item.colimits.in.UC.compute.colimits.in.C}), and that the extraction of zero objects and 1-localization sequences commutes with filtered colimits in $\St_2$ because the objects $0_{\St_2}, \oneLoc \in \St_2^\omega$ are compact (the latter by \Cref{lem.one.Loc}).

To complete the proof, we show that $\Mot$ carries 2-localization sequences to 1-localization sequences.  For this, choose any 2-localization sequence \Cref{prototypical.two.loc.seq} among small stable 2-categories. On applying $\Mot$, we obtain a sequence of adjunctions
\begin{equation}
\label{check.its.a.one.loc.seq.in.PrLSt.omega}
\begin{tikzcd}[column sep = 1.5cm]
\Mot(\cX)
\arrow[hook, yshift=0.9ex]{r}{\Mot(i)}
\arrow[leftarrow, yshift=-0.9ex]{r}[yshift=-0.2ex]{\bot}[swap]{\Mot(R)}
&
\Mot(\cY)
\arrow[yshift=0.9ex]{r}{\Mot(L)}
\arrow[hookleftarrow, yshift=-0.9ex]{r}[yshift=-0.2ex]{\bot}[swap]{\Mot(j)}
&
\Mot(\cZ)
\end{tikzcd}
\end{equation}
in $\PrLSt_\omega$, which we must show is a 1-localization sequence.  As $\PrLSt_\omega \subseteq \PrLSt$ is a 1-full sub-2-category, by \Cref{check.one.loc.in.larger.stable.two.cat} it suffices to check that the sequence of adjunctions \Cref{check.its.a.one.loc.seq.in.PrLSt.omega} is a 1-localization sequence in $\PrLSt$.  Thereafter, by \Cref{one.loc.sequences.in.PrLSt.and.PrRSt} (and using the notation of \Cref{obs.Mot.is.a.two.functor}) it suffices to check instead that the sequence of adjunctions
\[
\begin{tikzcd}[column sep = 2cm]
\Mot(\cZ)
\arrow[hook, yshift=0.9ex]{r}{\Mot^*(L)}
\arrow[leftarrow, yshift=-0.9ex]{r}[yshift=-0.2ex]{\bot}[swap]{\Mot^*(j)}
&
\Mot(\cY)
\arrow[yshift=0.9ex]{r}{\Mot^*(i)}
\arrow[hookleftarrow, yshift=-0.9ex]{r}[yshift=-0.2ex]{\bot}[swap]{\Mot^*(R)}
&
\Mot(\cX)
\end{tikzcd}
\]
in $\PrRSt$ is a 1-localization sequence.  Clearly $\Mot^*(i) \Mot^*(L) \simeq \Mot^*(Li) \simeq \Mot^*(0) \simeq 0$, so it remains to check that the middle sequence
\begin{equation}
\label{middle.sequence.of.mot.star.in.end.Y}
\Mot^*(L)\Mot^*(j)
\simeq
\Mot^*(jL)
\longra
\id_{\Mot(\cY)}
\longra
\Mot^*(iR)
\simeq 
\Mot^*(R) \Mot^*(i)
\end{equation}
is a 0-localization sequence in $\End(\Mot(\cY))$.  For this it suffices to check that given an arbitrary motive $\cM \in \Mot(\cY)$, the evaluation
\begin{equation}
\label{evaluation.at.M.of.middle.sequence.of.mot.star.in.end.Y}
\Mot^*(jL)(\cM) 
\longra 
\cM 
\longra
\Mot^*(iR)(\cM)
\end{equation}
of the sequence \Cref{middle.sequence.of.mot.star.in.end.Y} at $\cM$ is a 0-localization sequence in $\Mot(\cY)$.  In turn, for this it suffices to observe that the sequence \Cref{evaluation.at.M.of.middle.sequence.of.mot.star.in.end.Y} evaluates at an arbitrary object $Y \in \cY$ as the sequence
\[
\cM ( jL(Y) \longla Y \longla iR(Y) )
~,
\]
which is a 0-localization sequence by the assumptions that $\cM$ is a motive and \Cref{prototypical.two.loc.seq} is a 2-localization sequence.
\end{proof}

\begin{notation}
We write $\iota_2 \St_2 \xra{\univ_{2,1}} \Mot_{2,1}$ for the composite
\[
\univ_{2,1}
:
\iota_2 \St_2
\xra{\Yo}
\Fun( (\iota_2 \St_2^\omega)^\op , \Cat )
\xra{\cP_\Sp^\fin}
\Fun( (\iota_2 \St_2^\omega)^\op, \St)
\xra{L_{2,1}}
\Mot_{2,1}~.
\]
\end{notation}

\begin{remark}
\label{rmk.U.2.1.is.univ.2.1.additive.invt}
Evidently, the functor $\iota_2 \St_2 \xra{\univ_{2,1}} \Mot_{2,1}$ is a (2,1)-additive invariant.  In parallel with \Cref{rmk.U.X.is.univ.additive.invt}, we expect that it is universal. However, this would rely on the assertion (which we do not prove) that for a small 2-category $\cC \in \Cat_2$, the Yoneda embedding $\cC \ra \Fun(\cC^\op,\Cat_2)$ is the free cocompletion in an appropriate 2-categorical sense.
\end{remark}

\begin{observation}
\label{obs.Uto.pres.cpcts}
The functor $\univ_{2,1}$ preserves compact objects, being a composite of functors that preserve compact objects (the last by \Cref{lem.get.L.two.one}).
\end{observation}

\subsection{From (2,1)-motives to (2,1)-ary K-theory}
\label{subsec.from.2.1.motives.to.2.1.ary.K.theory}

In this subsection, we prove \Cref{mainthm:universal.property.of.K.2.1}.

\begin{theorem}
	\label{thm:universal.property.of.K.2.1}
	Let $\cX,\cY \in \St_2$ be small stable 2-categories, and assume that $\cX$ is compact.  Then, the functor
	\[
	\iota_2 \St_2 \xra{\univ_{2,1}} \Mot_{2,1}
	\]
	determines a canonical equivalence
	\[ \hom_{\Mot_{2,1}}(\univ_{2,1}(\cX) , \univ_{2,1}(\cY)) \simeq \Kto(\Fun^\twoex(\cX,\cY)) \]
	of stable categories.
\end{theorem}

\begin{remark}
In direct analogy with \Cref{rem.that.thm.homs.in.Mot.of.X.are.K.theory.can.be.coherent}, we expect that \Cref{thm:universal.property.of.K.2.1} can be upgraded to be more homotopy coherent.
\end{remark}

\begin{proof}[Proof of \Cref{thm:universal.property.of.K.2.1}]
	Consider the objects
	\[
	\cP_\cY
	:=
	\cP_\Spectra^\fin ( \iota_1 \Fun^\twoex ( - , \cY ))
	\simeq
	\cP_\Spectra^\fin ( \hom_{\iota_2 \St_2} ( - , \cY ) )
	\]
	and
	\[
	\Kto_\cY
	:=
	\Kto(\Fun^\twoex(-,\cY))
	:=
	\Mot(\Fun^\twoex( - , \cY ) )^\omega
	\]
	of $\Fun((\iota_2\St_2^\omega)^\op,\St)$.  It is immediate that $\Kto_\cY$ is in fact a (2,1)-motive, as $\Fun^\twoex(-,\cY)$ evidently preserves zero objects and 2-localization sequences and $\iota_2 \St_2 \xra{\Kto} \St^\idem$ is a (2,1)-additive invariant by \Cref{lem.Mot.and.K.2.1.add.invar}.  We note for future reference that the construction $\cY \mapsto \cP_\cY$ evidently assembles as a functor $\iota_2 \St_2 \xra{\cP_{(-)}} \Fun((\iota_2\St_2^\omega)^\op,\St)$.

Now, consider the morphism $\cP_\cY \xlongra{u} \Kto_\cY$ in $\Fun((\iota_2\St_2^\omega)^\op,\St)$ given by the composite
\[
\begin{tikzcd}[row sep = 0.3cm]
\cP_\cY 
\arrow[phantom]{d}[sloped]{:=}
\arrow[dashed]{rr}{u}
&
&[2cm]
\Kto_\cY 
\arrow[phantom]{d}[sloped]{:=}
\\
\Fun ( \iota_1 \Fun^\twoex(-,\cY)^\op , \Spectra )^\fin 
\arrow[hook]{r}[swap]{v}
&
\Fun ( \iota_1 \Fun^\twoex(-,\cY)^\op , \Spectra )^\omega
\arrow{r}[swap]{\left(L_{\Fun^\twoex(-,\cY)}\right)^\omega}
&
\Mot(\Fun^\twoex(-,\cY))^\omega
\end{tikzcd}
~,
\]
where the second functor results from \Cref{obs.Mot.lands.in.PrLSt.omega}. By the adjunction
\begin{equation}
\label{adjn.from.St.valued.pshvs.on.itwoSttwo.to.Mottwoone}
\begin{tikzcd}[column sep=1.5cm]
\Fun((\iota_2 \St_2^\omega)^\op , \St)
\arrow[yshift=0.9ex]{r}{L_{2,1}}
\arrow[hookleftarrow, yshift=-0.9ex]{r}[yshift=-0.2ex]{\bot}
&
\Mot_{2,1}
\end{tikzcd}
\end{equation}
of \Cref{lem.get.L.two.one} and the assumption that $\cX$ is compact, we have a natural equivalence
	\[
	\hom_{\Mot_{2,1}}(\univ_{2,1} \cX , \univ_{2,1} \cY )
	\simeq
	(L_{2,1} \cP_\cY)(\cX)
	~. \]
	To prove the theorem, it therefore suffices to show that the morphism $u$ is the component at $\cP_\cY$ of the unit of the adjunction \Cref{adjn.from.St.valued.pshvs.on.itwoSttwo.to.Mottwoone}.  For this, we will show that for any $(2,1)$-motive $\cM \in \Mot_{2,1} \subseteq \Fun( ( \iota_2 \St_2^\omega)^\op , \St)$ and any morphism $\cP_\cY \xra{\alpha} \cM$, there exists a unique factorization
	\[ \begin{tikzcd}
	\cP_\cY
	\arrow{r}{\alpha}
	\arrow{d}[swap]{u}
	&
	\cM
	\\
	\Kto_\cY
	\arrow[dashed]{ru}
	\end{tikzcd}~.\footnote{Note that the adjunction \Cref{adjn.from.St.valued.pshvs.on.itwoSttwo.to.Mottwoone} is $\iota_1 \St$-enriched, so that unenriched initiality implies enriched initiality (because $0_\St \in \iota_1 \St$ is the only stable category with a single equivalence class of objects).} \]
We observe that the morphism $\cP_\cY \xra{u} \Kto_\cY$ is an epimorphism (because it is a componentwise epimorphism),\footnote{Note that a morphism in $\PrLSt_\omega$ whose right adjoint in $\what{\Cat}$ is fully faithful defines an epimorphism in $\what{\Cat}$ and so is itself an epimorphism in the subcategory $\PrLSt_\omega \subset \what{\Cat}$.} and so it suffices to prove that there exists a factorization
\[ \begin{tikzcd}[row sep=1.5cm]
\cP_\cY(\cX)
:=
&[-1.1cm]
\Fun(\iota_1\Fun^\twoex(\cX,\cY)^\op , \Sp)^\fin
\arrow{r}{\alpha_\cX}
\arrow{d}[swap]{u_\cX}
&
\cM(\cX)
\\
\Kto_\cY(\cX)
:=
\hspace{-0.8cm}
&
\Mot(\Fun^\twoex(\cX,\cY))^\omega
\arrow[dashed]{ru}
\end{tikzcd} \]
for every $\cX \in \St_2^\omega$.  As $\cM(\cX)$ is idempotent-complete, $\alpha_\cX$ admits an extension along $v_\cX$. So it remains to show that this extension admits a further extension along $(L_{\Fun^\twoex(\cX,\cY)})^\omega$. For this, it suffices to prove the following two claims, which make reference to the factorization
\[ \begin{tikzcd}[row sep=1.5cm]
\iota_1\Fun^\twoex(\cX,\cY)
\arrow{r}{\Yo}
\arrow[dashed]{rrd}[sloped, anchor=north, swap]{\Sigma^\infty_+ \Yo}
&
\Fun(\iota_1\Fun^\twoex(\cX,\cY)^\op,\Spaces)
\arrow{r}{\Sigma^\infty_+}
&
\Fun(\iota_1\Fun^\twoex(\cX,\cY)^\op,\Sp)
\\
&
&
\Fun(\iota_1\Fun^\twoex(\cX,\cY)^\op , \Sp)^\fin
\arrow[hook]{u}
\end{tikzcd} ~. \]
\begin{enumerate}
\item The functor $\alpha_\cX$ carries the object
\[
\Sigma^\infty_+ \Yo( 0_{\Fun^\twoex(\cX,\cY)} )
\in
\Fun(\iota_1\Fun^\twoex(\cX,\cY)^\op , \Sp)^\fin
\]
to the zero object of $\cM(\cX)$.
\end{enumerate}
In order to state the second claim, choose any 1-localization sequence
\begin{equation}
\label{arbitrary.one.loc.seq.in.Fun.twoex.X.Y}
\begin{tikzcd}[column sep=1.5cm]
F
\arrow[hook, yshift=0.9ex]{r}{i}
\arrow[leftarrow, yshift=-0.9ex]{r}[yshift=-0.2ex]{\bot}[swap]{R}
&
G
\arrow[yshift=0.9ex]{r}{L}
\arrow[hookleftarrow, yshift=-0.9ex]{r}[yshift=-0.2ex]{\bot}[swap]{j}
&
H
\end{tikzcd}
\end{equation}
in $\Fun^\twoex(\cX,\cY)$.  As a result of the first claim, the composite
\begin{equation}
\label{spectralized.Yo.of.a.one.loc.seq.in.Fun.twoex.X.Y}
\Sigma^\infty_+ \Yo(F)
\xra{\Sigma^\infty_+ \Yo(i)}
\Sigma^\infty_+ \Yo(G)
\xra{\Sigma^\infty_+ \Yo(L)}
\Sigma^\infty_+ \Yo(H)
\end{equation}
in $\Fun(\iota_1\Fun^\twoex(\cX,\cY)^\op , \Sp)^\fin$ comes equipped with a canonical nullhomotopy.
\begin{enumerate}
\setcounter{enumi}{1}
\item The functor $\alpha_\cX$ carries the composite \Cref{spectralized.Yo.of.a.one.loc.seq.in.Fun.twoex.X.Y} to a cofiber sequence in $\cM(\cX)$.
\end{enumerate}

We begin with the first claim.  Observe first the equivalence
\[
\Sigma^\infty_+ \Yo\left( 0_{\Fun^\twoex(\cX,\cY)} \right)
\simeq
\const_\SS
\]
in $\Fun(\iota_1\Fun^\twoex(\cX,\cY)^\op , \Sp)^\fin$.  Observe too that the unique morphism $0_{\St_2} \ra \cY$ in $\St_2$ determines a morphism
\[
\const_{\Sp^\fin}
\simeq
\cP_{0_{\St_2}} \longra \cP_\cY
\]
in $\Fun((\iota_2\St_2^\omega)^\op,\St)$, whose component at $\cX \in \iota_2\St_2^\omega$ is the morphism
\begin{equation}
\label{morphism.on.Ps.from.0Sttwo.to.Y}
\Sp^\fin
\longra
\cP_\cY(\cX)
:=
\Fun(\iota_1\Fun^\twoex(\cX,\cY)^\op , \Sp)^\fin
\end{equation}
in $\St$ that classifies the object $\Sigma_+^\infty\Yo\left(0_{\Fun^\twoex(\cX,\cY)}\right)$.  We thus obtain a composite morphism
\begin{equation}
\label{composite.morphism.from.unoSt.to.M}
\const_{\Sp^\fin}
\simeq
\cP_{0_{\St_2}}
\xra{\Cref{morphism.on.Ps.from.0Sttwo.to.Y}}
\cP_\cY
\xra{\alpha}
\cM
\end{equation}
in $\Fun((\iota_2\St_2^\omega)^\op,\St)$ whose component at $\cX \in \iota_2\St_2^\omega$ is the morphism
\begin{equation}
\label{component.at.X.of.composite.morphism.from.unoSt.to.M}
\Sp^\fin
\longra
\cM(\cX)
\end{equation}
in $\St$ that classifies the object 
\begin{equation}
\label{object.classified.by.component.at.X.of.composite.morphism.from.unoSt.to.M}
\alpha_\cX\left(\Sigma^\infty_+ \Yo\left( 0_{\Fun^\twoex(\cX,\cY)} \right)\right) \in \cM(\cX)~.
\end{equation}
In light of the equivalence
\[
\hom_{\Fun(\iota_1\Fun^\twoex(\cX, \cY)^\op, \Sp)}(\cP_{0_{\St_2}}, \cM)
\simeq
\hom_{\Mot_{2,1}} ( L_{2,1} ( \cP_{0_{\St_2}} ) , \cM )
\simeq
\cM(0_{\St_2})
\simeq
0_\St
~,
\]
we find that the morphism \Cref{composite.morphism.from.unoSt.to.M} must be the zero morphism, so that the morphism \Cref{component.at.X.of.composite.morphism.from.unoSt.to.M} must be the zero morphism, so that the object \Cref{object.classified.by.component.at.X.of.composite.morphism.from.unoSt.to.M} must be the zero object.

We now turn to the second claim. A 1-localization sequence \Cref{arbitrary.one.loc.seq.in.Fun.twoex.X.Y} in $\Fun^\twoex(\cX,\cY)$ is classified by a functor 
\begin{equation}
\label{functor.from.oneLoc.to.Funtwoex.X.Y}
\oneLoc 
\longra 
\Fun^\twoex(\cX, \cY)~.
\end{equation}
This is recorded by a functor
\begin{equation}
\label{functor.from.square.classifying.FGH}
[1]\times[1] \longra \iota_1 \Fun^\twoex(\cX, \cY)~,
\end{equation}
which composes to a functor
\begin{equation}
\label{functor.classifying.cofiber.sequence.in.m.x.with.F.G.H}
[1]\times[1] 
\xra{\Cref{functor.from.square.classifying.FGH}} 
\iota_1\Fun^\twoex(\cX, \cY) 
\longra
\cP_\cY(\cX) 
\xra{\alpha_\cX} 
\cM(\cX)
\end{equation}
selecting a commutative square
\begin{equation}
\label{commutative.square.in.m.x.from.F.G.H}
\begin{tikzcd}[column sep=1.5cm, row sep=1.5cm]
&[-1.7cm]
\alpha_\cX \left( \Sigma^\infty_+ \Yo(F) \right)
\arrow{r}{\alpha_\cX \left( \Sigma^\infty_+ \Yo(i) \right)}
\arrow{d}
&
\alpha_\cX \left( \Sigma^\infty_+ \Yo(G) \right)
\arrow{d}{\alpha_\cX \left( \Sigma^\infty_+ \Yo(L) \right)}
\\
0_{\cM(\cX)}
\simeq
&
\alpha_\cX \left( \Sigma^\infty_+ \Yo(0_{\Fun^\twoex(\cX,\cY)}) \right)
\arrow{r}
&
\alpha_\cX \left( \Sigma^\infty_+ \Yo(H) \right)
\end{tikzcd}~.
\end{equation}
(Here, the bottom left equivalence follows from the first claim.) It remains to show that \Cref{commutative.square.in.m.x.from.F.G.H} is a cofiber sequence.

Let us write
\[
\cX
\xra{FGH}
\Fun^\twoex(\oneLoc,\cY)
=:
\oneLoc(\cY)
\]
for the functor corresponding to the functor \Cref{functor.from.oneLoc.to.Funtwoex.X.Y} (recall \Cref{lem.omnibus.in.section.one.point.one}\Cref{omnibus.in.one.point.one.symm.mon}). By \Cref{lem.omnibus.in.section.one.point.one}\Cref{omnibus.in.one.point.one.cgsm}, we may write $\cY \simeq \colim_{i \in \cI}(\cY_i)$ for some filtered diagram $\cI \xra{\cY_\bullet} \St_2^\omega$ of compact stable 2-categories.  By \Cref{lem.one.Loc}, we obtain an equivalence $\oneLoc(\cY) \simeq \colim_{i \in \cI}(\oneLoc(\cY_i))$. Since $\cX \in \St_2$ is compact, the functor $FGH$ factors through some $\oneLoc(\cY_i)$.  Applying \Cref{obs.two.loc.seqs.preserve.cpctness} to the 2-localization sequence
\begin{equation}
\label{two.loc.sequences.with.one.Loc.Yi}
\begin{tikzcd}[column sep=1.5cm]
\cY_i
\arrow[hook, yshift=0.9ex]{r}{\id\da0}
\arrow[leftarrow, yshift=-0.9ex]{r}[yshift=-0.2ex]{\bot}[swap]{\ev_0}
&
\oneLoc(\cY_i)
\arrow[yshift=0.9ex]{r}{\ev_2}
\arrow[hookleftarrow, yshift=-0.9ex]{r}[yshift=-0.2ex]{\bot}[swap]{0\da\id}
&
\cY_i
\end{tikzcd}
\end{equation}
of \Cref{ex.of.two.locs}\Cref{ex.two.loc.seq.involving.one.loc}, we find that $\oneLoc(\cY_i) \in \St_2$ is compact, so that we may apply $\cM$ to it.\footnote{This argument would be slightly simplified by a theory of ind-extension for 2-categories.} Now, in order to show that \Cref{commutative.square.in.m.x.from.F.G.H} is a cofiber sequence, we procure a lift
\begin{equation}
\label{lift.of.commutative.square.along.M.FGH}
\begin{tikzcd}[column sep = 1.5cm, row sep = 1.5cm]
&
\cM(\oneLoc(\cY_i))
\arrow{d}{\cM(FGH)}
\\
{[1] \times [1]}
\arrow[dashed]{ru}
\arrow{r}[swap]{\Cref{functor.classifying.cofiber.sequence.in.m.x.with.F.G.H}}
&
\cM(\cX)
\end{tikzcd}
\end{equation}
that selects a cofiber sequence in $\cM(\oneLoc(\cY_i))$. Indeed, applying $\cM$ to the 2-localization sequence \Cref{two.loc.sequences.with.one.Loc.Yi} yields a 1-localization sequence
\begin{equation}
\label{one.localization.sequence.m.of.y.to.oneloc.y.to.y}
\begin{tikzcd}[column sep=1.5cm]
\cM(\cY_i)
\arrow[hook, yshift=0.9ex]{r}{\cM(\ev_0)}
\arrow[leftarrow, yshift=-0.9ex]{r}[yshift=-0.2ex]{\bot}[swap]{\cM(\id\da 0)}
&
\cM(\oneLoc(\cY_i))
\arrow[yshift=0.9ex]{r}{\cM(0 \da \id)}
\arrow[hookleftarrow, yshift=-0.9ex]{r}[yshift=-0.2ex]{\bot}[swap]{\cM(\ev_2)}
&
\cM(\cY_i)
\end{tikzcd} 
\end{equation}
in $\St$ (recall that $\cM$ is 1-contravariant but 2-covariant). The middle 0-localization sequence in the 1-localization sequence \Cref{one.localization.sequence.m.of.y.to.oneloc.y.to.y} is the square 
\begin{equation}
\label{middle.0.loc.sequence.in.m.of.y.to.oneloc.y.to.y}
\begin{tikzcd}[column sep = 1.5cm, row sep = 1.5cm]
\cM((\id\da0) \circ \ev_0) 
\arrow{r}
\arrow{d}
&
\id_{\cM(\oneLoc(\cY_i))} 
\arrow{d}
\\
0_{\End_\St(\cM(\oneLoc(\cY_i)))}
\arrow{r}
&
\cM((0\da\id) \circ \ev_2) 
\end{tikzcd}
\end{equation}
in $\End_\St(\cM(\oneLoc(\cY_i)))$.  Let us consider the composite morphism $\alpha_i : \cP_{\cY_i} \ra \cP_\cY \xra{\alpha} \cM$ as an object $\alpha_i \in \cM(\cY_i)$.  Then, evaluating the square \Cref{middle.0.loc.sequence.in.m.of.y.to.oneloc.y.to.y} at the object $\cM(\ev_1)(\alpha_i) \in \cM(\oneLoc(\cY_i))$, we obtain a 0-localization sequence \Cref{zero.loc.seq.in.M.of.1loc.of.Y.i}
\begin{figure}[h]
\begin{equation}
\label{zero.loc.seq.in.M.of.1loc.of.Y.i}
\begin{tikzcd}[row sep=1.5cm]
\cM(\ev_0)(\alpha_i)
\\[-1.5cm]
\rotatebox{90}{$\simeq$}
\\[-1.5cm]
\cM(\ev_1 \circ (\id \da 0) \circ \ev_0)
\\[-1.5cm]
\rotatebox{90}{$\simeq$}
\\[-1.5cm]
\cM((\id\da0) \circ \ev_0) ( \cM(\ev_1)(\alpha_i))
\arrow{r}
\arrow{d}
&
\cM(\ev_1)(\alpha_i)
\arrow{d}
\\
0_{\cM(\oneLoc(\cY))}
\arrow{r}
&
\cM((0\da\id) \circ \ev_2) ( \cM(\ev_1)(\alpha_i))
\\[-1.5cm]
&
\rotatebox{90}{$\simeq$}
\\[-1.5cm]
&
\cM(\ev_1 \circ ( 0 \da \id) \circ \ev_2)(\alpha_i)
\\[-1.5cm]
&
\rotatebox{90}{$\simeq$}
\\[-1.5cm]
&
\cM(\ev_2)(\alpha_i)
\end{tikzcd}
\end{equation}
\end{figure}
in $\cM(\oneLoc(\cY_i))$ that defines a lift \Cref{lift.of.commutative.square.along.M.FGH}, as desired.
\end{proof}

\subsection{From 2-motives to secondary K-theory}
\label{subsec.from.2.motives.to.secondary.K.theory}

In this subsection, we (re)introduce secondary K-theory and 2-motives, and then we prove \Cref{mainthm:locating.K.2}.

\begin{definition}
\label{def.secondary.K.theory}
Let $\cX \in \St_2$ be a small stable 2-category. The \bit{secondary K-theory} of $\cX$ is the spectrum 
\[
\Ktwo := \K(\Kto(\cX))~.
\]
\end{definition}

\begin{definition}
\label{defn.two.motives}
We define the stable category of \bit{2-motives} to be
\[
\Mot_2
:=
\Mot(\Mot_{2,1}^\omega)
~,
\]
the category of motives over the small stable 2-category of compact $(2,1)$-motives.
\end{definition}

\begin{notation}
\label{notn.U.2}
We write $\univ_2$ for the composite
\[
\univ_2
:
\iota_1 \St_2^\omega
\xra{\iota_1 \univ_{2,1}^\omega}
\iota_1 \Mot_{2,1}^\omega
\xra{\univ_{\Mot_{2,1}^\omega} }
\Mot(\Mot_{2,1}^\omega)
=:
\Mot_2
~,
\]
where the first functor results from \Cref{obs.Uto.pres.cpcts}.
\end{notation}

\begin{theorem}
	\label{thm:locating.K.2}
	Let $\cX,\cY \in \St_2^\omega$ be compact stable 2-categories.  
	Then, the functor
	\[
	\iota_1 \St_2^\omega
	\xra{\univ_2}
	\Mot_2
	\]
	determines a canonical equivalence
	\[
	\hom_{\Mot_2}(\univ_2(\cX),\univ_2(\cY))
	\simeq
	\Ktwo(\Fun^\twoex(\cX,\cY))
	\]
	of spectra.
\end{theorem}

\begin{proof}
We compute that
\begin{align}
\nonumber
\hom_{\Mot_2}(\univ_2(\cX),\univ_2(\cY))
& :=
\hom_{\Mot(\Mot_{2,1}^\omega)} \left( \univ_{\Mot_{2,1}^\omega} ( (\iota_1 \univ_{2,1}^\omega) ( \cX ) ) , \univ_{\Mot_{2,1}^\omega} ( (\iota_1 \univ_{2,1}^\omega) ( \cY ) ) \right)
\\
\label{use.that.homs.in.Mot.of.X.are.K.theory}
& \simeq
\K \left( \hom_{\Mot_{2,1}^\omega} ( \iota_1 \univ_{2,1}^\omega ( \cX )  ,  \iota_1 \univ_{2,1}^\omega ( \cY ) ) \right)
\\
\nonumber
& =
\K \left( \hom_{\Mot_{2,1}} ( \univ_{2,1} ( \cX )  , \univ_{2,1} ( \cY ) ) \right)
\\
\label{use.universal.property.of.K.2.1}
& \simeq
\K \left( \Kto ( \Fun^\twoex(\cX,\cY) ) \right)
\\
\nonumber
& =:
\Ktwo ( \Fun^\twoex(\cX,\cY) )
~,
\end{align}
where equivalence \Cref{use.that.homs.in.Mot.of.X.are.K.theory} follows from \Cref{thm.homs.in.Mot.of.X.are.K.theory} and equivalence \Cref{use.universal.property.of.K.2.1} follows from \Cref{thm:universal.property.of.K.2.1}.
\end{proof}

\begin{remark}
\label{rmk.functor.Utwo.prime}
In parallel with \Cref{define.2.1.additive.invt}, let us say that a functor $\iota_1\St_2 \xra{F} \cT$ to a presentable stable category is a \textit{$\cT$-valued 2-additive invariant} if it preserves zero objects and filtered colimits and takes 2-localization sequences to 0-localization sequences. In contrast with \Cref{rmk.U.2.1.is.univ.2.1.additive.invt}, we do not expect that the functor $\univ_2$ of \Cref{notn.U.2} is the universal 2-additive invariant. On the other hand, it is easy to construct a universal 2-additive invariant $\iota_1\St_2 \xra{\univ_2'} \Mot_2'$, and thereafter to obtain a commutative triangle
\[ \begin{tikzcd}
\iota_1 \St_2^\omega
\arrow{r}{\univ_2'}
\arrow{rd}[sloped, anchor=north, swap]{\univ_2}
&
\Mot_2'
\arrow[dashed]{d}
\\
&
\Mot_2
\end{tikzcd}~. \]
We do not know whether the hom-spectra in $\Mot_2'$ are also secondary K-theory spectra, i.e.\! (in view of \Cref{thm:locating.K.2}) whether the factorization is fully faithful when restricted to the image of $\iota_1 \St_2^\omega$.
\end{remark}

\appendix

\section{Some enriched category theory}
\label{app.some.enriched.category.theory}

We begin this section in \Cref{subsec.foundations.conventions} by laying out the conventions that we use throughout this paper. Thereafter, we study a number of basic aspects of enriched category theory: co/limits and adjunctions (\Cref{subsection.enriched.limits.colimits.and.adjns}), presentability (\Cref{subsection.prbl.V.cats}), compact generation (\Cref{subsection.cg.V.cats}), and semiadditivity (\Cref{subsec.semiadditive.V.categories}).

\subsection{Conventions}
\label{subsec.foundations.conventions}

In this subsection, we lay out the conventions regarding enriched and higher categories that we use throughout this paper.

\begin{convention}
\label{conv.implicit.infty}
Throughout, we use the ``implicit $\infty$'' convention. For instance, by ``category'' we mean ``$\infty$-category'' (meaning ``$(\infty,1)$-category''), by ``limit'' we mean ``$\infty$-categorical limit'', by ``$n$-category'' we mean ``$(\infty, n$)-category'', and so on.
\end{convention}

\begin{convention}
We use the foundations of category theory established in \cite{HTT,HA}, with which we assume a basic familiarity. We also use many of the notations and conventions introduced there, albeit with a few exceptions -- notably \Cref{convention.as.enriched.as.possible}.
\end{convention}

\begin{convention}
We use the device of Grothendieck universes, using the terms ``small'', ``large'', ``huge'', and ``massive'' for convenience (rather than referring to implicitly chosen strongly inaccessible cardinals). However, to avoid unnecessary repetition, we often merely state results for ``small'' objects; those that make no reference to size issues will obviously apply to ``large'' objects as well.
\end{convention}

\begin{remark}
We systematically omit routine higher-algebraic arguments that would clutter exposition (e.g.\! the associativity asserted in \Cref{obs.tensoring.gives.action}).
\end{remark}

\begin{notation}
We use the notation $\boxtimes$ to refer to any otherwise-unspecified monoidal or symmetric monoidal structure. Given an arbitrary monoidal category $\cW$, we may write $\boxtimes_\cW$ for its monoidal structure and $\uno_\cW$ for its unit object.
\end{notation}

\begin{notation}
\label{notn.V.enr.cats}
Given a monoidal category $(\cW,\boxtimes)$, we write
\[
\Cat(\cW^\boxtimes)
\]
for the large category of small $\cW$-enriched categories \cite{GH}.\footnote{We generally suppress the distinction between $\cW$-enriched categories and ``flagged'' $\cW$-enriched categories (a.k.a.\! ``categorical algebras''), as it will not be relevant for us. For example, \Cref{obs.facts.about.enr.cats}\Cref{symm.mon.str.of.Cat.V} actually describes the symmetric monoidal structure on \textit{flagged} $\cV$-enriched categories. The unique exception is the proof of \Cref{lem.if.V.is.cgsm.then.Cat.V.is.cgsm}.} When the monoidal structure $\boxtimes$ is clear, we may simply write
\[
\Cat(\cW) := \Cat(\cW^\boxtimes)
~.
\]
We also write
\[
\what{\Cat}(\cW)
:=
\what{\Cat}(\cW^\boxtimes)
\]
for the huge category of large $\cW$-enriched categories. We often use the term ``$\cW$-category'' to mean ``$\cW$-enriched category''.
\end{notation}

\begin{notation}
\label{convention.as.enriched.as.possible}
All constructions and concepts (hom-objects, co/limits, adjunctions, etc.) should be understood to be ``as enriched as possible''; we use additional notation when we wish to refer to unenriched notions. In particular, we write $\Cat$ for the large 2-category of small 1-categories, and we automatically consider stable categories as $\Spectra$-enriched. In general, given a $\cW$-category $\cC \in \Cat(\cW)$ and objects $X,Y \in \cC$, we write
\[
\hom_\cC(X,Y)
\in
\cW
\]
for the $\cW$-object of morphisms from $X$ to $Y$. Furthermore, if we wish to specifically emphasize $\cW$ (e.g.\! when we are discussing multiple enrichments), we will write
\[
\hom_\cC^\cW(X,Y)
:=
\hom_\cC(X,Y)
~.
\]
There is a single exception, described in parts \Cref{self.enrichment.of.V} and \Cref{symm.mon.str.of.Cat.V} of \Cref{obs.facts.about.enr.cats}.\footnote{Namely, we find it more convenient and more descriptive to write $\cW$ and $\ul{\cW}$ for a monoidal category and its self-enrichment, rather than writing $U(\cW)$ and $\cW$ respectively.}
\end{notation}

\begin{notation}
\label{notn.iota.k}
We use the notation $\iota_k$ to denote the ``maximal sub-$k$-category'' of an $n$ category (for any $0 \leq k \leq n$), in a way that will be given its full meaning by \Cref{define.n.cats} (see also \Cref{obs.iota.k.best.functoriality}). In particular, we write $\iota_1 \Cat$ for the large 1-category of small 1-categories.
\end{notation}

\begin{notation}
\label{convention.V.is.presentably.s.m.}
We will be interested in categories enriched in a presentably symmetric monoidal category. However, many of the basic definitions we give in this subsection apply to categories enriched in an arbitrary monoidal category. We therefore write
\[
\cV
:=
(\cV, \boxtimes)
\in
\CAlg(\iota_1 \PrL)
\]
for an arbitrary but fixed presentably symmetric monoidal category and
\[
\cW
\in
\Alg(\iota_1 \Cat)
\]
for an arbitrary but fixed monoidal category. We use these same notations but include subscripts when we wish to refer to multiple such objects.
\end{notation}

\begin{remark}
We comment on the two simplifying assumptions of \Cref{convention.V.is.presentably.s.m.}: presentability and symmetry.
\begin{enumerate}

\item Given a monoidal category $\cW \in \Alg(\Cat)$, Day convolution makes its presheaf category $\cP(\cW)$ presentably monoidal, in such a way that the Yoneda embedding
\[
\cW
\longhookra
\cP(\cW)
\]
is monoidal -- in addition to being fully faithful and limit-preserving as usual. It follows that we obtain a fully faithful inclusion
\[
\Cat(\cW)
\longhookra
\Cat(\cP(\cW))
\]
from $\cW$-categories into $\cP(\cW)$-categories. Inasmuch as enriched category theory generally makes reference to limits (but not colimits) in the enriching category, this implies that the assumption that $\cV$ be presentably monoidal entails no real loss of generality. Moreover, a presentably monoidal category is canonically self-enriched, which provides notational simplification (in contrast with the resulting enrichment of $\cW$ over $\cP(\cW)$).

\item In general, if $\cW$ is $(\cO \otimes \EE_1)$-monoidal then $\Cat(\cW)$ is $\cO$-monoidal. In particular, if $\cW$ is symmetric monoidal then so is $\Cat(\cW)$. We will ultimately only be interested in categories enriched in a symmetric monoidal category. As the general case requires some additional care (e.g.\! regarding the handedness of co/tensors), we find it convenient to restrict ourselves to the symmetric monoidal case.

\end{enumerate}
\end{remark}

\begin{notation}
\label{notn.change.enrichment}
We will at times endow a category with multiple enrichments. These will always be compatible: given a $\cW_0$-category $\cC \in \Cat(\cW_0)$ and a laxly monoidal functor $\cW_1 \ra \cW_0$, we will write
\[
\cC^{\cW_1\enr}
\in
\Cat(\cW_1)
\]
for a chosen lift through the resulting functor
\[
\Cat(\cW_1)
\longra
\Cat(\cW_0)
~.
\]
In this situation, we may also write
\[
\cC^{\cW_0\enr}
:=
\cC
\in
\Cat(\cW_0)
\]
to emphasize the original enrichment.
\end{notation}

\begin{observation}
\label{obs.facts.about.enr.cats}
We collect a number of facts about enriched categories here from \cite{GH,Haugsengbimods,macphersonoperad,Hinichyonedainfty} that we use without further comment, establishing a number of further conventions along the way.
\begin{enumerate}

\item

A laxly monoidal functor
\[
\cW_0
\xlongla{G}
\cW_1
\]
determines a functor
\[
\Cat(\cW_0)
\xla{\Cat(G)}
\Cat(\cW_1)
\]
between categories of enriched categories (given by applying $G$ hom-wise). Moreover, $\Cat(G)$ is a monomorphism if $G$ is.

\item

Given an adjunction
\[ \begin{tikzcd}[column sep=2cm]
\cW_0
\arrow[yshift=0.9ex]{r}{F}
\arrow[leftarrow, yshift=-0.9ex]{r}[yshift=-0.2ex]{\bot}[swap]{G}
&
\cW_1
\end{tikzcd}~, \]
if $F$ is monoidal then $G$ is laxly monoidal in a canonical way, and thereafter we obtain an adjunction
\[ \begin{tikzcd}[column sep=2cm]
\Cat(\cW_0)
\arrow[yshift=0.9ex]{r}{\Cat(F)}
\arrow[leftarrow, yshift=-0.9ex]{r}[yshift=-0.2ex]{\bot}[swap]{\Cat(G)}
&
\Cat(\cW_1)
\end{tikzcd}~. \]

\item\label{Spaces.enr.cats.are.cats}

Categories enriched in spaces are simply ordinary categories:
\[
\Cat(\Spaces)
\simeq
\iota_1 \Cat
~.
\]

\item\label{und.cat.of.an.enr.cat}

The functor
\[
\cW
\xra{\hom_\cW(\uno_\cW,-)}
\Spaces
\]
is laxly monoidal. We simply write
\[
\Cat(\cW)
\xlongra{U}
\Cat(\Spaces)
\simeq
\iota_1 \Cat
\]
for the functor $\Cat(\hom_\cW(\uno_\cW,-))$, which we refer to as the \bit{underlying} (\bit{unenriched}) \bit{category} functor. We note that $U(\cC)$ has the same space of objects as $\cC$; as a result, we implicitly consider objects of $\cC$ as objects of $U(\cC)$. Note in particular that objects of $\cC$ are equivalent just when they are equivalent as objects of $U(\cC)$.

\item Given a category $\cI \in \Cat$ and a $\cW$-category $\cC \in \Cat(\cW)$, we may simply write $\cI \ra \cC$ for a morphism $\cI \ra U(\cC)$ to the underlying category of $\cC$. In particular, this gives meaning to the notion of an ordinary commutative diagram in an enriched category.

\item\label{self.enrichment.of.V}

As $\cV$ is presentably monoidal, it admits an internal hom, and hence we can consider it as self-enriched. We write
\[
\ul{\cV}
:=
\cV^{\cV\enr}
\in
\what{\Cat}(\cV)
\]
for this self-enrichment. Of course, we have $\cV \simeq U(\ul{\cV})$. 
\item\label{symm.mon.str.of.Cat.V}

The fact that $\cV$ is presentably symmetric monoidal implies that $\Cat(\cV)$ is also presentably symmetric monoidal. Its symmetric monoidal structure (which, by abuse of notation, we also denote by $\boxtimes$) is described by the formulas
\[
\iota_0(\cC \boxtimes \cD) := \iota_0 \cC \times \iota_0 \cD
\qquad
\textup{and}
\qquad
\hom_{\cC \boxtimes \cD}((C_0,D_0),(C_1,D_1)) := \hom_\cC(C_0,C_1) \boxtimes \hom_\cD(D_0,D_1)
~.
\]
It follows that $\Cat(\cV)$ also admits an internal hom. As a special case of \Cref{self.enrichment.of.V}, we write $\ul{\Cat(\cV)} \in \what{\Cat}(\Cat(\cV))$ for this self-enrichment. We simply write
\[
\Fun(\cC,\cD)
:=
\hom_{\ul{\Cat(\cV)}}(\cC,\cD)
\in
\Cat(\cV)
\]
for its internal hom.

\item A $\cV$-category $\cC \in \Cat(\cV)$ admits a natural \textit{enriched Yoneda embedding}: a(n enrichedly) fully faithful embedding
\[
\cC
\xhookrightarrow{C \mapsto \hom_\cC(-,C)}
\Fun(\cC^\op , \ul{\cV})
~.
\]

\end{enumerate}
\end{observation}

\begin{definition}
\label{define.n.cats}
We define a \bit{0-category} to be a space:
\[
\iota_1 \Cat_0 := \Spaces
~.
\]
We consider $\iota_1 \Cat_0$ as a symmetric monoidal category via cartesian product. Thereafter, for any $n \geq 1$, we define a (small or large) \bit{$n$-category} to be a (resp.\! small or large) category enriched in small $(n-1)$-categories:
\[
\iota_1 \Cat_n
:=
\Cat(\iota_1 \Cat_{n-1})
:=
\Cat(\iota_1 \Cat_{n-1}^\times)
\qquad
\text{and}
\qquad
\iota_1 \what{\Cat}_n
:=
\what{\Cat}(\iota_1 \Cat_{n-1})
:=
\what{\Cat}(\iota_1 \Cat_{n-1}^\times)
~.
\]
We consider $\iota_1 \Cat_n$ and $\iota_1 \what{\Cat}_n$ as symmetric monoidal categories via cartesian product.
\end{definition}

\begin{observation}
We use the following facts regarding \Cref{define.n.cats} without further comment.
\begin{enumerate}

\item Our definition of the 1-category of $n$-categories agrees with all definitions in the literature \cite{BarSPunicity}. In particular, $\iota_1 \Cat_1 \simeq \iota_1 \Cat$.

\item For every $n \geq 0$, $\iota_1 \Cat_n$ is presentably symmetric monoidal (and in particular cartesian closed) \cite{Rezk-ncats}.

\item We have a canonical equivalence
\[
\Cat
\simeq
\ul{\iota_1 \Cat_1}
\in
\iota_1 \what{\Cat}(\iota_1 \Cat_1)
=:
\iota_1 \what{\Cat}_2
~.
\]

\item The functor
\[
\iota_1 \Cat_n
\xlongra{U}
\iota_1 \Cat
\]
is simply $\iota_1$, because the functor
\[
\iota_1 \Cat_{n-1}
\xra{\hom_{\iota_1 \Cat_{n-1}} ( \uno_{\iota_1 \Cat_{n-1}} , - )}
\Spaces
\]
is simply $\iota_0$.

\end{enumerate}
\end{observation}

\begin{notation}
We write
\[
\Cat_n
:=
\ul{\iota_1 \Cat_n}
\in
\iota_1 \what{\Cat}(\iota_1 \Cat_n)
=:
\iota_1 \what{\Cat}_{n+1}
\]
for the self-enrichment of $\iota_1 \Cat_n \in \Cat$.
\end{notation}

\begin{observation}
\label{obs.iota.k.best.functoriality}
For any $0 \leq k \leq n$, passage to maximal sub-$k$-categories defines a morphism
\[
\iota_{k+1} \Cat_n
\xlongra{\iota_k}
\Cat_k
\]
of $(k+1)$-categories. We use this fact without further comment.
\end{observation}

\begin{notation}
Given an $n$-category $\cC \in \Cat_n$ and any $0 < k \leq n$, we write $\cC^{\kop}$ for the $n$-category obtained by reversing the direction of its $k$-morphisms. In particular, we simply write $\cC^\op := \cC^\oneop$. Moreover, assuming that $n \geq 2$ we write $\cC^\onetwoop := (\cC^{\oneop})^{\twoop} \simeq (\cC^{\twoop})^{\oneop}$.
\end{notation}

\begin{definition}
\label{defn.k.full.sub.n.cat}
Fix any $0 \leq k \leq n$. Given an $n$-category $\cC \in \Cat_n$, we say that a sub-$n$-category $\cC' \subseteq \cC$ is \bit{$k$-full} if for every morphism $\cD \xra{F} \cC$ in $\Cat_n$, the spaces of factorizations
\[
\begin{tikzcd}
\iota_k \cD
\arrow{r}{\iota_k F}
\arrow[dashed]{rd}
&
\iota_k \cC
\\
&
\iota_k \cC'
\arrow[hook]{u}
\end{tikzcd}
\qquad
\text{and}
\qquad
\begin{tikzcd}
\cD
\arrow{r}{F}
\arrow[dashed]{rd}
&
\cC
\\
&
\cC'
\arrow[hook]{u}
\end{tikzcd}
\]
are equivalent (via the canonical map from the space of factorizations in $\iota_1 \Cat_n$ to the space of factorizations in $\iota_1 \Cat_k$). We refer to a 0-full sub-$n$-category simply as a \bit{full sub-$n$-category}.
\end{definition}

\begin{notation}
We define the various (huge) 2-categories appearing in the diagram
\begin{equation}
\label{diagram.define.Pr}
\begin{tikzcd}
\PrL_\omega
\arrow[hook]{r}
&
\PrL
\arrow[hook]{r}
&
\what{\Cat}
\arrow[hookleftarrow]{r}
&
\PrR
\\
\PrLSt_\omega
\arrow[hook]{u}
\arrow[hook]{r}
&
\PrLSt
\arrow[hook]{u}
&
&
\PrRSt
\arrow[hook]{u}
\end{tikzcd}
\end{equation}
as follows: ${\sf Pr}$ stands for presentable, $(-)^L$ and $(-)^R$ stand for left and right adjoint functors as 1-morphisms, $(-)^{\sf st}$ requires stability, and $(-)_\omega$ requires compact-generation and functors that preserve (not only colimits but also) compact objects. (So, in diagram \Cref{diagram.define.Pr}, all horizontal morphisms are 1-full and all vertical morphisms are 0-full.)
\end{notation}

\subsection{Limits, colimits, and adjunctions}
\label{subsection.enriched.limits.colimits.and.adjns}

In this subsection, we establish a few basic results regarding co/limits and adjunctions in the enriched context. Notably, we give criteria guaranteeing the compatibility of enriched and unenriched limits (\Cref{lem.limits.in.C.versus.UC}), and we characterize the unenriched adjunctions that arise from enriched adjunctions (\Cref{lem.G.a.right.adjoint.between.V.categories.iff}).

\begin{localass}
In this subsection, we fix a category $\cI \in \Cat$ and $\cV$-enriched categories $\cC,\cD \in \what{\Cat}(\cV)$.
\end{localass}

\begin{definition}
\label{def:enriched.limits.and.colimits}
For any functor $\cI \xra{C_\bullet} U(\cC)$ in $\what{\Cat}(\Spaces) \simeq \what{\Cat}$ to the underlying unenriched category of $\cC$, its \bit{colimit} and \bit{limit} (if they exist) are objects
\[
\colim_\cI^\cC(C_\bullet)
\qquad
\textup{and}
\qquad
\lim_\cI^\cC(C_\bullet)
\]
of $\cC$ equipped with $\cV$-natural equivalences
\[
\hom_\cC ( \colim^\cC_\cI(C_\bullet) , T )
\simeq
\lim^{\cV}_{\cI^\op} \hom_\cC ( C_\bullet , T )
\qquad
\textup{and}
\qquad
\hom_\cC ( T , \lim^\cC_\cI(C_\bullet) )
\simeq
\lim^{\cV}_\cI \hom_\cC ( T , C_\bullet )
\]
in $\cV$ for every $T \in \cC$.\footnote{Note that this is not circular, as $\lim^{\cV}_\cI$ is merely a limit in $\cV \in \what{\Cat}$.}
\end{definition}

\begin{remark}
\label{rmk.only.conical}
In the enriched category theory literature, the notions given in \Cref{def:enriched.limits.and.colimits} are referred to as \textit{conical} co/limits: the special case of a weighted co/limit where the indexing $\cV$-category is free on an unenriched category and the weight is constant at the unit object of $\cV$. We omit the term ``conical'' because these are the only sorts of co/limits that we consider. Moreover, in the case that $\cC$ is a \emph{presentable} $\cV$-category, there is no ambiguity (see \Cref{lem.equiv.conds.for.presentable.n.cat,lem.limits.in.C.versus.UC}).\footnote{For any $\cV$-category $\cC \in \what{\Cat}(\cV)$, co/limits in $\cC$ are also co/limits in $U(\cC)$. However, the converse need not hold in general.}
\end{remark}

\begin{observation}
\label{obs.conical.limits.are.limits.in.underlying}
Let $\cI \xra{C_\bullet} U(\cC)$ be a diagram. If $\lim_\cI^\cC(C_\bullet)$ (resp. $\colim_\cI^\cC(C_\bullet)$) exists, then so does $\lim_\cI^{U(\cC)}(C_\bullet)$ (resp. $\colim_\cI^{U(\cC)}(C_\bullet)$), and moreover these limits (resp. colimits) agree.
\end{observation}

\begin{definition}
\begin{enumerate}
\item[]

\item

\begin{enumerate}

\item The \bit{weak tensor} of an object $C \in \cC$ by an object $V \in \cV$ is the copresheaf
\[
V \tensor^\flat C
:=
\hom_{\cV}(V , \hom_\cC(C,-))
\in
\hom_{\what{\Cat}} ( U(\cC) , \Spaces )
~.
\]
If these copresheaves are corepresentable (in $U(\cC)$) for all $V \in \cV$ and all $C \in \cC$, we say that $\cC$ \bit{admits weak tensors}.

\item The \bit{weak cotensor} of an object $C \in \cC$ by an object $V \in \cV$ is the presheaf
\[
V \cotensor^\flat C
:=
\hom_{\cV}(V , \hom_\cC(-,C))
\in
\hom_{\what{\Cat}} ( U(\cC)^\op , \Spaces )
~.
\]
If these presheaves are representable (in $U(\cC)$) for all $V \in \cV$ and all $C \in \cC$, we say that $\cC$ \bit{admits weak cotensors}.

\end{enumerate}

\item

\begin{enumerate}

\item The \bit{tensor} of an object $C \in \cC$ by an object $V \in \cV$ is the enriched copresheaf
\[
V \tensor C
:=
\hom_{\ul{\cV}} ( V , \hom_\cC(C,-))
\in
\hom_{\what{\Cat}(\cV)} ( \cC , \ul{\cV} )
~.
\]
If these copresheaves are corepresentable (in $\cC$) for all $V \in \cV$ and all $C \in \cC$, we say that $\cC$ \bit{admits tensors}.

\item The \bit{cotensor} of an object $C \in \cC$ by an object $V \in \cV$ is the enriched presheaf
\[
V \cotensor C
:=
\hom_{\ul{\cV}}(V , \hom_\cC(-,C))
\in
\hom_{\what{\Cat}(\cV)} ( \cC^\op , \ul{\cV} )
~.
\]
If these presheaves are representable (in $\cC$) for all $V \in \cV$ and all $C \in \cC$, we say that $\cC$ \bit{admits cotensors}.

\end{enumerate}

\end{enumerate}
\end{definition}

\begin{observation}
\label{obs.co.tensors.are.weak.co.tensors}
Choose any $C \in \cC$ and $V \in \cV$. If the tensor $V \tensor C$ is corepresentable in $\cC$, then this object also corepresents the weak tensor $V \tensor^\flat C$ in $U(\cC)$.\footnote{More generally, the weak tensor $V \tensor^\flat C$ is the underlying unenriched copresheaf of the tensor $V \tensor C$, i.e.\! the former is the image of the latter under the composite
\[
\hom_{\what{\Cat}(\cV)} ( \cC , \ul{\cV} )
\xlongra{U}
\hom_{\what{\Cat}} ( U(\cC) , \cV )
\xra{\hom_\cV ( \uno_\cV , - ) \circ (-)}
\hom_{\what{\Cat}}( U(\cC) , \Spaces )
~.
\]
.} Dually, if the cotensor $V \cotensor C$ is representable in $\cC$, then this object also represents the weak cotensor $V \cotensor^\flat C$ in $U(\cC)$. In particular, if $\cC$ admits co/tensors, then it admits weak co/tensors, and these objects coincide. We may use this fact without further comment.
\end{observation}

\begin{notation}
In view of \Cref{obs.co.tensors.are.weak.co.tensors}, if $\cC$ admits all co/tensors, we use the notation corresponding to which property we are using: co/tensors or weak co/tensors.
\end{notation}

\begin{remark}
As a partial converse to \Cref{obs.co.tensors.are.weak.co.tensors}, if $\cV \xra{\hom_{\cV}(\uno_\cV,-)} \Spaces$ is conservative, then weak tensors are automatically cotensors.
\end{remark}

\begin{observation}
\label{obs.tensoring.gives.action}
Suppose that $\cC$ admits tensors. By \cite[Theorems 7.3 and 7.15]{Heine-enriched}, we obtain an enriched bifunctor $\ul{\cV} \boxtimes \cC \xra{- \tensor -} \cC$, which extends to an action of $\ul{\cV} \in \Alg(\what{\Cat}(\cV))$ on $\cC$.
\end{observation}

\begin{remark}
It is not possible to weaken the assumptions of \Cref{obs.tensoring.gives.action} to require only weak tensors. Indeed, the bifunctor $\cV \times U(\cC) \xra{-\tensor^\flat-} U(\cC)$ fails to be associative in general. 
\end{remark}

\begin{lemma}
\label{lem.limits.in.C.versus.UC}
\begin{enumerate}
\item[] 

\item\label{want.limits.in.UC.to.be.limits.in.C} Suppose that one of the following conditions holds.
\begin{enumerate}
\item\label{admits.weak.tensors} $\cC$ admits weak tensors.
\item\label{admits.weak.cotensors.and.these.commute.with.limits} $\cC$ admits weak cotensors and, for each $V \in \cV$, the functor $U(\cC) \xra{V \cotensor^\flat (-)} U(\cC)$ commutes with limits.
\end{enumerate}
Then limits in $U(\cC)$ compute limits in $\cC$, in the sense that whenever $\lim_\cI^{U(\cC)}(C_\bullet)$ exists, so does $\lim_\cI^\cC(C_\bullet)$, and these limits agree.

\item\label{want.colimits.in.UC.to.be.colimits.in.C} Suppose that one of the following conditions holds.
\begin{enumerate}
\item $\cC$ admits weak cotensors.
\item $\cC$ admits weak tensors and, for each $V \in \cV$, the functor $U(\cC) \xra{V \tensor^\flat (-)} U(\cC)$ commutes with colimits.
\end{enumerate}
Then colimits in $U(\cC)$ compute colimits in $\cC$, in the sense that whenever $\colim_\cI^{U(\cC)}(C_\bullet)$ exists, so does $\colim_\cI^\cC(C_\bullet)$, and these colimits agree.
\end{enumerate}
\end{lemma}

\begin{proof}
We prove only the first part, as the second is formally dual. Given a functor $\cI \xra{C_\bullet} U(\cC)$ such that the limit $\lim^{U(\cC)}_\cI(C_\bullet) \in U(\cC)$ exists, we must show that this limit in $U(\cC)$ is also a limit in $\cC$. For this, it suffices to show that for any test objects $T \in \cC$ and $V \in \cV$ we have a sequence of equivalences
\begin{align}
\label{pull.limit.in.UC.out.to.limit.in.Spaces}
\hom_{\cV}(V , \hom_\cC ( T , \lim^{U(\cC)}_\cI ( C_\bullet ) ) )
&
\simeq
\lim^\Spaces_\cI \hom_{\cV} ( V , \hom_\cC ( T , C_\bullet ) )
\\
\label{push.limit.in.Spaces.to.limit.in.UV}
&
\simeq
\hom_{\cV} ( V , \lim^{\cV}_\cI ( \hom_\cC ( T , C_\bullet ) ) )
~.
\end{align}
Equivalence \Cref{push.limit.in.Spaces.to.limit.in.UV} follows from the universal property of the limit in $\cV$. Condition \Cref{admits.weak.tensors} implies equivalence \Cref{pull.limit.in.UC.out.to.limit.in.Spaces} via the sequence of equivalences
\begin{align*}
\hom_{\cV}(V , \hom_\cC ( T , \lim^{U(\cC)}_\cI ( C_\bullet ) ) )
&
\simeq
\hom_{U(\cC)} ( V \tensor^\flat T , \lim^{U(\cC)}_\cI ( C_\bullet ) )
\simeq
\lim^\Spaces_\cI \hom_{U(\cC)} ( V \tensor^\flat T , C_\bullet)
\\
&
\simeq
\lim^\Spaces_\cI \hom_{\cV} ( V , \hom_\cC ( T , C_\bullet ))
\end{align*}
(using the universal property of the weak tensor). Meanwhile, condition \Cref{admits.weak.cotensors.and.these.commute.with.limits} implies equivalence \Cref{pull.limit.in.UC.out.to.limit.in.Spaces} via the sequence of equivalences
\begin{align}
\nonumber
\hom_{\cV}(V , \hom_\cC ( T , \lim^{U(\cC)}_\cI ( C_\bullet ) ) )
&
\simeq
\hom_{U(\cC)} ( T , V \cotensor^\flat \lim^{U(\cC)}_\cI ( C_\bullet ) )
\\
\label{commute.limit.with.cotensor}
&
\simeq
\hom_{U(\cC)} ( T , \lim^{U(\cC)}_\cI ( V \cotensor^\flat C_\bullet ) )
\\
\nonumber
& \simeq
\lim^\Spaces_\cI \hom_{U(\cC)} ( T , V \cotensor^\flat C_\bullet )
\simeq
\lim^\Spaces_\cI \hom_{\cV} ( V , \hom_\cC ( T , C_\bullet ) )
~,
\end{align}
(where equivalence \Cref{commute.limit.with.cotensor} follows from the assumption that weak cotensors commute with limits).
\end{proof}

\begin{definition}
Given functors $\cC \xra{F} \cD$ and $\cC \xla{G} \cD$, an \bit{adjunction} between $F$ and $G$ is a $\cV$-natural equivalence
\[
\hom_\cD(F(C), D) 
\simeq 
\hom_\cC(C, G(D))
\]
in $\cV$. We write
\[ \begin{tikzcd}[column sep=1.5cm]
\cC
\arrow[yshift=0.9ex]{r}{F}
\arrow[yshift=-0.9ex, leftarrow]{r}[yshift=-0.2ex]{\bot}[swap]{G}
&
\cD
\end{tikzcd} \]
for an adjunction.
\end{definition}

\begin{observation}
\label{obs.adjts.as.factorizns.through.yoneda}
If it exists, the right adjoint to a functor $\cC \xra{F} \cD$ is the factorization
\[ \begin{tikzcd}[column sep = 2.5cm, row sep=1.5cm]
\Fun ( \cC^\op , \ul{\cV} )
&
\cD
\arrow{l}[swap]{\hom_\cD ( F(=),-)}
\arrow[dashed]{ld}
\\
\cC
\arrow[hook]{u}
\end{tikzcd}
\]
through the Yoneda embedding. Dually, if it exists, the left adjoint to a functor $\cC \xla{G} \cD$ is the factorization
\[ \begin{tikzcd}[column sep = 2.5cm, row sep=1.5cm]
\cC
\arrow{r}{\hom_\cC ( - , G(=))}
\arrow[dashed]{rd}
&
\Fun ( \cD^\op , \ul{\cV} )
\\
&
\cD
\arrow[hook]{u}
\end{tikzcd} \]
through the Yoneda embedding.
\end{observation}

The following lemma will allow us to promote unenriched adjunctions to enriched ones.

\begin{lemma}
\label{lem.G.a.right.adjoint.between.V.categories.iff}
\begin{enumerate}
\item[] 

\item\label{part.ladjt.lem.G.a.right.adjoint.between.V.categories.iff}

Suppose that $\cC$ and $\cD$ admit weak tensors. Then a functor $\cC \xra{F} \cD$ is a left adjoint if and only if the following two conditions are satisfied:
\begin{enumerate}
    \item the functor $U(\cC) \xra{U(F)} U(\cD)$ is a left adjoint, and
    \item\label{cond.F.commutes.with.weak.tensors} the functor $F$ commutes with weak tensors.
\end{enumerate}

\item\label{part.radjt.lem.G.a.right.adjoint.between.V.categories.iff}

Suppose that $\cC$ and $\cD$ admit weak cotensors. Then a functor $\cC \xla{G} \cD$ is a right adjoint if and only if the following two conditions are satisfied:
\begin{enumerate}
    \item the functor $U(\cC) \xla{U(G)} U(\cD)$ is a right adjoint, and
    \item the functor $G$ commutes with weak cotensors.
\end{enumerate}

\end{enumerate}
\end{lemma}

\begin{proof}
We prove the second part to shake things up a bit; the first part is formally dual. It is clear that the conditions are necessary, so let us show that they are also sufficient. Let us write $g := U(G)$, and $f \adj g$ for its left adjoint. Then, we have natural equivalences
\begin{align*}
\hom_{\cV}(V , \hom_\cC ( C , G(D)))
&
\simeq 
\hom_{U(\cC)}(C,V \cotensor^\flat G(D))
\simeq 
\hom_{U(\cC)} (C,G(V \cotensor^\flat D))
\\
& =:
\hom_{U(\cC)} (C, g(V \cotensor^\flat D))
\simeq
\hom_{U(\cD)} ( f(C) , V \cotensor^\flat D )
\\
&
\simeq
\hom_{\cV} ( V , \hom_\cD( f(C) , D) )
\end{align*}
for any $C \in \cC$, $D \in \cD$, and $V \in \cV$. Hence, the claim follows from \Cref{obs.adjts.as.factorizns.through.yoneda}.
\end{proof}

\begin{observation}
\label{obs.radjt.in.presbly.V.enriched.adjn.commutes.with.cotensors}
In part \Cref{part.ladjt.lem.G.a.right.adjoint.between.V.categories.iff} (resp.\! part \Cref{part.radjt.lem.G.a.right.adjoint.between.V.categories.iff}) of \Cref{lem.G.a.right.adjoint.between.V.categories.iff}, if $\cC$ and $\cD$ admit tensors (resp.\! cotensors), then the second condition is equivalent to the condition that $F$ commutes with tensors (resp.\! that $G$ commutes with cotensors).
\end{observation}

\subsection{Presentable $\cV$-categories}
\label{subsection.prbl.V.cats}

In this subsection, we introduce presentable $\cV$-categories and study their basic features. Notably, we provide a recognition result (\Cref{lem.equiv.conds.for.presentable.n.cat}), which among other applications leads to an adjoint functor theorem (\Cref{cor.enriched.adjt.functor.thm}) and implies that $\Cat$-valued presheaves form a presentable 2-category (\Cref{prop.Fun.B.Cat.is.prbl.and.tensoring.is.ptwise}). Among other miscellany, we also prove that presentable $\cV$-categories can be studied in unenriched terms (\Cref{thm.presble.V.cats.are.V.mods}).

\begin{notation}
We write $\otimes$ for the symmetric monoidal structure on $\PrL$: this is characterized by the property that given $\cC,\cD, \cE \in \PrL$, the datum of a colimit-preserving functor
\[\cC \otimes \cD \longra \cE\] is equivalent to that of a functor \[\cC \times \cD \longra \cE\] that preserves colimits separately in each variable.\footnote{So, we obtain a laxly symmetric monoidal monomorphism $(\iota_1 \PrL,\otimes) \hookrightarrow (\iota_1 \what{\Cat},\times)$.}
\end{notation}

\begin{definition}
\label{def.presentable.V.category}
A $\cV$-category $\cC \in \what{\Cat}(\cV)$ is \bit{presentable} if it admits tensors and moreover the action of $\cV$ on $U(\cC)$ resulting from \Cref{obs.tensoring.gives.action} lies in $(\iota_1 \PrL,\otimes) \subseteq (\iota_1 \what{\Cat},\times)$.\footnote{To elaborate, this means that $U(\cC) \in \iota_1 \PrL$ and furthermore that the functor $\cV \times U(\cC) \xra{-\tensor-} U(\cC)$ commutes with colimits separately in each variable (so that it defines a morphism $\cV \otimes U(\cC) \ra U(\cC)$ in $\PrL$).} We write
\[
\iota_1 \PrL_\cV
\subseteq
\what{\Cat}(\cV)
\]
for the subcategory whose objects are presentable $\cV$-categories and whose morphisms are ($\cV$-enriched) left adjoint functors. As a particular case, we write
\[
\iota_1 \PrL_n
:=
\iota_1 \PrL_{\iota_1 \Cat_{n-1}}
\subseteq
\what{\Cat}(\iota_1 \Cat_{n-1})
=:
\iota_1 \what{\Cat}_n
\]
for the category of presentable $n$-categories.
\end{definition}

\begin{remark}
\Cref{def.presentable.V.category} may be rephrased as follows: writing $\what{\Cat}(\cV)^\tensor \subseteq \what{\Cat}(\cV)$ for the full subcategory on those (possibly large) $\cV$-categories that admit tensors, we have a commutative diagram
\[
\begin{tikzcd}
\iota_1 \PrL_\cV
\arrow[hook]{r}
\arrow{d}
&
\what{\Cat}(\cV)^\tensor
\arrow[hook]{r}
\arrow{d}
&
\what{\Cat}(\cV)
\arrow{d}{U}
\\
\Mod_\cV ( \iota_1 \PrL )
\arrow[hook]{r}
&
\Mod_\cV(\iota_1 \what{\Cat})
\arrow{r}[swap]{\fgt}
&
\what{\Cat}
\end{tikzcd}
\]
among huge categories in which the left square is a pullback.
\end{remark}

\begin{example}
\label{example.self.enrichment.of.V.is.a.prbl.V.cat}
The self-enrichment $\ul{\cV} \in \what{\Cat}(\cV)$ defines a presentable $\cV$-category.
\end{example}

\begin{lemma}
\label{lem.equiv.conds.for.presentable.n.cat}
Let $\cC \in \what{\Cat}(\cV)$ be a (possibly large) $\cV$-enriched category that admits tensors, and suppose that $U(\cC) \in \what{\Cat}$ is presentable. Then, the following are equivalent.
\begin{enumerate}
\item\label{item.C.is.prble}

The $\cV$-category $\cC \in \what{\Cat}(\cV)$ is presentable.

\item\label{item.tensor.with.fixed.V}

For every $V \in \cV$, the functor
\[
U(\cC)
\xra{V \tensor^\flat (-)}
U(\cC)
\]
preserves colimits.

\item\label{item.colimits.in.UC.compute.colimits.in.C}

Colimits in $U(\cC)$ compute colimits in $\cC$.

\end{enumerate}
\end{lemma}

\begin{proof}
That condition $\Cref{item.C.is.prble}$ implies condition $\Cref{item.tensor.with.fixed.V}$ is immediate from the fact that tensors are weak tensors. Conversely, condition $\Cref{item.tensor.with.fixed.V}$ implies condition $\Cref{item.C.is.prble}$ because, for any $C \in U(\cC)$, the functor
\[
\cV
\xra{- \tensor^\flat C}
U(\cC)
\]
preserves colimits. 
That condition $\Cref{item.tensor.with.fixed.V}$ implies condition $\Cref{item.colimits.in.UC.compute.colimits.in.C}$ follows from  \Cref{lem.limits.in.C.versus.UC}\Cref{want.colimits.in.UC.to.be.colimits.in.C}. To conclude, it therefore suffices to show that condition $\Cref{item.colimits.in.UC.compute.colimits.in.C}$ implies condition $\Cref{item.tensor.with.fixed.V}$. So, choose any $\cI \in \Cat$, any diagram $\cI \xra{C_\bullet} U(\cC)$, and any object $V \in \cV$. Assuming condition $\Cref{item.colimits.in.UC.compute.colimits.in.C}$, we must show that the canonical morphism
\[
\colim^{U(\cC)}_\cI ( V \tensor^\weak C_\bullet)
\longra
V \tensor^\weak \colim^{U(\cC)}_\cI ( C_\bullet)
\]
is an equivalence. For this, for any $T \in U(\cC)$ we compute that
\begin{align*}
\hom_{U(\cC)} ( V \tensor^\flat \colim^{U(\cC)}_\cI (C_\bullet) , T )
& \simeq
\hom_{\cV} ( V , \hom_\cC ( \colim^{U(\cC)}_\cI(C_\bullet) , T ) )
\overset{\Cref{item.colimits.in.UC.compute.colimits.in.C}}{\simeq}
\hom_{\cV} ( V , \hom_\cC ( \colim^{\cC}_\cI(C_\bullet) , T ) )
\\
& \simeq
\hom_{\cV} ( V , \lim^{\cV}_{\cI^\op} \hom_\cC ( C_\bullet , T ) )
\simeq
\lim^\Spaces_{\cI^\op} \hom_{\cV} ( V , \hom_\cC ( C_\bullet , T ) )
\\
& \simeq
\lim^\Spaces_{\cI^\op} \hom_{U(\cC)} ( V \tensor^\weak C_\bullet  , T )
\simeq
\hom_{U(\cC)} ( \colim^{U(\cC)}_\cI ( V \tensor^\weak C_\bullet ) , T )
~. \qedhere
\end{align*}
\end{proof}

Here are two important consequences of \Cref{lem.equiv.conds.for.presentable.n.cat}.

\begin{corollary}
\label{cor.enriched.adjt.functor.thm}
Let $\cC,\cD \in \iota_1 \PrL_\cV$ be presentable $\cV$-categories.
\begin{enumerate}

\item A functor $\cC \xra{F} \cD$ is a left adjoint if and only if it preserves colimits and tensors.

\item A functor $\cD \xleftarrow{G} \cD$ is a right adjoint if and only if it preserves limits, cotensors, and $\kappa$-filtered colimits for some regular cardinal $\kappa$.
\end{enumerate}
\end{corollary}

\begin{proof}
This follows from combining Lemmas \ref{lem.equiv.conds.for.presentable.n.cat} and \ref{lem.G.a.right.adjoint.between.V.categories.iff} and the unenriched adjoint functor theorem.
\end{proof}

\begin{proposition}
\label{prop.Fun.B.Cat.is.prbl.and.tensoring.is.ptwise}
Let $\cB \in \Cat_2$ be a small 2-category. Then, the functor 2-category $\Fun(\cB,\Cat) \in \what{\Cat}_2$ is presentable. Moreover, the action in $\iota_1\PrL$ of $\iota_1 \Cat$ on $\iota_1 \Fun(\cB,\Cat)$ is given by pointwise product:
\[ \begin{tikzcd}[row sep=0cm]
\iota_1 \Cat
\times
\iota_1 \Fun ( \cB,\Cat)
\arrow{r}
&
\iota_1 \Fun ( \cB,\Cat)
\\
\rotatebox{90}{$\in$}
&
\rotatebox{90}{$\in$}
\\
\left( \cK , \cB \xlongra{F} \Cat \right)
\arrow[maps to]{r}
&
\left( \cB \xlongra{F} \Cat \xra{\cK \times (-)} \Cat \right)
\end{tikzcd}
~.
\]
\end{proposition}

\begin{proof}
The following facts result from \cite[Proposition 3.3.1]{GHL}.
\begin{itemize}

\item Applying \cite[Proposition A.3.7.6]{HTT}, we find that the category $\iota_1 \Fun ( \cB , \Cat) \in \what{\Cat}$ is presentable.

\item Passing through the Quillen equivalence between the projective and injective model structures and using \cite[Theorem 2.1]{MG-qadjns}, we find that colimits in $\iota_1 \Fun ( \cB, \Cat)$ are computed pointwise.

\item The 2-category $\Fun ( \cB , \Cat) \in \what{\Cat}(\iota_1 \Cat)$ admits tensors (and hence weak tensors), and these are given by pointwise product.

\end{itemize}
The claim now follows from \Cref{lem.equiv.conds.for.presentable.n.cat} (using its condition \Cref{item.tensor.with.fixed.V}).
\end{proof}

We now show that presentable $\cV$-categories can be studied in unenriched terms.

\begin{theorem}
\label{thm.presble.V.cats.are.V.mods}
The functor
\[
\iota_1 \PrL_\cV
\longra
\Mod_\cV ( \iota_1 \PrL)
\]
resulting from \Cref{obs.tensoring.gives.action} is an equivalence.
\end{theorem}

\begin{proof}
The functor is fully faithful by \Cref{lem.G.a.right.adjoint.between.V.categories.iff}, so it remains to show that it is surjective. For this, fix any $\cC \in \Mod_\cV ( \iota_1 \PrL)$, and let us denote its action by $\cV \otimes \cC \xra{(-)\cdot (-)} \cC$. Now, by \cite[Theorem 7.3]{Heine-enriched}, $\cC$ admits a canonical enhancement to a $\cP(\cV)$-enriched category $\cC^{\cP(\cV)\enr}$ via the formula
\[
\hom_{\cC^{\cP(\cV)\enr}}(X,Y)(V)
:=
\hom_\cC(V \cdot X , Y )
~.
\]
For any $X,Y \in \cC$ this presheaf is representable by the adjoint functor theorem, and so in fact $\cC^{\cP(\cV)\enr} $ lies in the subcategory 
\[\what{\Cat}(\cV) \subseteq \what{\Cat}(\cP(\cV))~.
\]
Directly from the definition, $\cC^{\cP(\cV)\enr}$ admits tensors, which on the underlying category $\cC \simeq U(\cC^{\cP(\cV)\enr})$ recovers the given action of $\cV$ on $\cC$. So indeed, the object $\cC^{\cP(\cV)\enr} \in \iota_1 \PrL_\cV$ lifts the object $\cC \in \Mod_\cV(\iota_1 \PrL)$.
\end{proof}

\begin{notation}
We implicitly use the equivalence of \Cref{thm.presble.V.cats.are.V.mods}, without adding extra notation. Specifically, for $\cC \in \Mod_\cV(\iota_1\PrL)$, we write $\cC$ for the corresponding presentable $\cV$-category, $U(\cC)$ for the underlying category thereof, and  $\cV \otimes U(\cC) \xra{-\tensor -} U(\cC)$ for the action of $\cV$ in $\iota_1 \PrL$ thereon.
\end{notation}

We discuss change of enrichment for presentable enriched categories.

\begin{observation}
\label{obs.consequences.of.morphism.in.CAlg.PrL}
Let $\cV_0 \xra{F} \cV_1$ be a morphism in $\CAlg(\iota_1 \PrL)$, and write $F \adj G$ for the underlying adjunction in $\what{\Cat}$.
\begin{enumerate}

\item\label{obs.for.Cathat.of.radjt.to.presly.s.m.fctr.both.directions.agree}

The diagram
\[ \begin{tikzcd}[column sep = 1.5cm]
\what{\Cat}(\cV_0)
&
\what{\Cat}(\cV_1)
\arrow{l}[swap]{\what{\Cat}(G)}
\\
\iota_1 \PrL_{\cV_0}
\arrow{d}[sloped, anchor=north]{\sim}
\arrow[hook]{u}
&
\iota_1 \PrL_{\cV_1}
\arrow{d}[sloped, anchor=south]{\sim}
\arrow[hook]{u}
\\
\Mod_{\cV_0}(\iota_1 \PrL)
&
\Mod_{\cV_1}(\iota_1 \PrL)
\arrow{l}{F^*}
\end{tikzcd} 
\]
commutes (in which the equivalences are those of \Cref{thm.presble.V.cats.are.V.mods}).

\item 

For any $\cC \in \Cat(\cV_1)$, the tensor and cotensor of $X \in \cC^{\cV_0\enr}$ by $V \in \cV_0$ are respectively given by the formulas
\[
V \tensor X \simeq F(V) \tensor X
\qquad
\text{and}
\qquad
V \cotensor X \simeq F(V) \cotensor X
\]
(as enriched co/presheaves).

\item\label{item.presly.s.m.loczn}

If $G$ is fully faithful, then the tensor and cotensor of $X \in (\cV_1)^{\cV_0\enr}$ by $V \in \cV_0$ are also respectively given by the formulas
\[
V
\tensor
X
\simeq
F(V \boxtimes G(X))
\qquad
\text{and}
\qquad
V
\cotensor
X
\simeq
\hom_{\ul{\cV_0}} ( V , G(X) )
\]
(i.e.\! the object $\hom_{\ul{\cV_0}} ( V , G(X) ) \in \cV_0$ is in the image of $G$).\footnote{More generally, even if $G$ is not fully faithful we have that $G(V \cotensor X) \simeq V \cotensor G(X) \simeq \hom_{\ul{\cV_0}}(V,G(X))$ by \Cref{lem.G.a.right.adjoint.between.V.categories.iff}\Cref{part.radjt.lem.G.a.right.adjoint.between.V.categories.iff}.} In particular, the self-enrichment of $\cV_0$ restricts to the self-enrichment of $\cV_1$.\footnote{More precisely, the fully faithful inclusion
$
\what{\Cat}(\cV_1)
\xhookrightarrow{\what{\Cat}(G)}
\what{\Cat}(\cV_0)
$
carries $\ul{\cV_1}$ to the full sub-$\cV_0$-category $(\cV_1)^{\cV_0\enr} \subseteq \ul{\cV_0}$.}
\end{enumerate}
\end{observation}

We have the following enriched analog of the fact that presentable categories admit limits.

\begin{lemma}
\label{lem.presble.V.cats.admit.cotensors}
Let $\cC \in \iota_1\PrL_\cV$ be a presentable $\cV$-category. For any $C \in \cC$ and any $V \in \cV$, the composite
\begin{equation}
\label{composite.giving.cotensor.in.a.prbl.V.cat}
U(\cC)^\op
\xra{V \tensor -}
U(\cC)^\op
\xra{\hom_{U(\cC)}(-,C)}
\Spaces
\end{equation}
is representable, and moreover it is the cotensor of $C$ by $V$. In particular, $\cC$ admits cotensors.
\end{lemma}

\begin{proof}
It is immediate that the composite \Cref{composite.giving.cotensor.in.a.prbl.V.cat} is the weak cotensor. Moreover, it is representable by an object $(V \cotensor^\flat C) \in U(\cC)$ by the adjoint functor theorem. So, it remains to check that this weak cotensor is in fact a cotensor: that is, we must show that for any $T \in \cC$ we have
\[
\hom_{\cC}(T,V \cotensor^\flat C)
\simeq
\hom_{\ul{\cV}}(V , \hom_{\cC}(T,C) )
~.
\]
Applying $\hom_\cV(W,-)$ for an arbitrary $W \in \cV$, we find that
\begin{align*}
\hom_\cV ( W , \hom_{\cC}(T,V \cotensor^\flat C) )
& \simeq
\hom_{U(\cC)} ( W \tensor^\flat T , V \cotensor^\flat C )
\simeq
\hom_\cV ( V , \hom_{\cC} ( W \tensor^\flat T , C ) )
\\
& \simeq
\hom_\cV ( V \tensor^\flat ( W \tensor^\flat T ) , C )
\simeq
\hom_\cV ( V \tensor ( W \tensor T ) , C )
\\
& \simeq
\hom_{U(\cC)} ( ( V \boxtimes W) \tensor T , C )
\simeq 
\hom_{U(\cC)} ( (V \boxtimes W) \tensor^\flat T, C)
\\
& \simeq
\hom_\cV ( V \boxtimes W, \hom_{\cC}(T , C ) )
\simeq
\hom_\cV ( W , \hom_{\ul{\cV}} ( V , \hom_{\cC}(T,C)))
~,
\end{align*}
as desired.
\end{proof}

\subsection{Compactly generated $\cV$-categories}
\label{subsection.cg.V.cats}

In this subsection, we study compact generation in the enriched context. Notably, we prove that $\Cat(\cV)$ is compactly-generatedly symmetric monoidal if $\cV$ is (\Cref{lem.if.V.is.cgsm.then.Cat.V.is.cgsm}), and we establish a relationship between enriched and unenriched compactness in presentable $\cV$-categories (\Cref{lem.compact.in.prbl.V.cat.detected.in.underlying}).

\begin{definition}
\label{def.compact.object.in.V.category}
Given a $\cV$-category $\cC$, we say that an object $C \in \cC$ is \bit{compact} if the functor 
\[
\cC \xra{\hom_\cC(C, -)} \ul{\cV}
\]
preserves filtered colimits. We write
\[
\cC^\omega
\subseteq
\cC
\]
for the full sub-$\cV$-category on the compact objects.
\end{definition}

\begin{remark}
The term ``filtered'' here refers only to filtered 1-categories and colimits thereover; we do not contemplate any enriched notions of filteredness. 
\end{remark}

\begin{observation}
\label{obs.if.enr.right.adj.pres.filt.colims.then.left.adj.pres.cmpcts}
Given an adjunction $F \adj G$ in $\Cat(\cV)$, if $G$ preserves filtered colimits then $F$ preserves compact objects.
\end{observation}

\begin{definition}
\label{def.compactly.generatedly.monoidal}
We define the category of \bit{compactly-generatedly symmetric monoidal} (or simply \bit{cgsm}) categories to be the subcategory
\[
\CAlg(\iota_1\PrL_\omega)
\subseteq
\CAlg(\iota_1\PrL)
~.
\]
In other words, a presentably symmetric monoidal category is cgsm if it is compactly generated and moreover its compact objects form a symmetric monoidal subcategory.
\end{definition}

\begin{example}
\label{obs.Cat.n.is.compactly.generated.iota.1.Cat.n.category}
It follows from inductively applying \Cref{lem.if.V.is.cgsm.then.Cat.V.is.cgsm} below that the category $\iota_1\Cat_n$ is cgsm.
\end{example}

\begin{localass}
\label{localass.V.cmpctly.gnrtdly.symm.mon}
In this subsection, we assume that our presentably symmetric monoidal $\cV$ is in fact cgsm.
\end{localass}

\begin{lemma}
\label{lem.if.V.is.cgsm.then.Cat.V.is.cgsm}
$\Cat(\cV)$ is cgsm.
\end{lemma}

\begin{proof}
We use the notation of \cite{GH}. We first show that $\Cat(\cV)$ is compactly generated. First of all, because $\cV$ is compactly generated, we have that $\Cat(\cV)$ is generated by the objects $[n](V_1,\ldots,V_n) \in \Cat(\cV)$ for $[n] \in \bDelta$ and $V_i \in \cV^\omega$. To show that these objects are compact, we observe that
\[
[n](V_1,\ldots,V_n)
\simeq
\left(
[1](V_1)
\coprod_{[0]_\cV}
\cdots
\coprod_{[0]_\cV}
[1](V_n)
\right)
~,
\]
so it suffices to show that the objects $[0]_\cV,[1](V) \in \Cat(\cV)$ are compact for $V \in \cV^\omega$.
\begin{itemize}
\item
We first show that the object $[0]_\cV \in \Cat(\cV)$ is compact. It suffices to show that the object $E^1_\cV \in \Alg_{\sf cat}(\cV)$ is compact. For this, we observe that the functor
\begin{equation}
\label{und.segal.spaces.functor}
\Alg_{\sf cat}(\cV) \xra{ \Alg_{\sf cat}(\hom_\cV(\uno_\cV , - ) ) }
\Alg_{\sf cat}(\Spaces)
\end{equation}
preserves filtered colimits because by assumption the object $\uno_\cV \in \cV^\omega$ is compact. So, it suffices to show that the object $E^1:= E^1_\Spaces \in \Alg_{\sf cat}(\Spaces)$ is compact. This follows from the facts that it is the Segalification of a finite colimit of compact simplicial spaces and that the Segalification functor preserves compact objects because its right adjoint preserves filtered colimits. So indeed, the object $[0]_\cV \in \Cat(\cV)$ is compact.
\item
We now show that the object $[1](V) \in \Cat(\cV)$ is compact, where $V \in \cV^\omega$. Observe that for any $\cC \in \Cat(\cV)$, applying the functor
\[
\Cat(\cV)^\op
\xra{\hom_{\Cat(\cV)}(-,\cC)}
\Spaces
\]
to the morphism
\[
[1](V)
\longla
[0]_\cV \amalg [0]_\cV
\]
yields a morphism
\[
\hom_{\Cat(\cV)}([1](V),\cC)
\longra
\hom_{\Cat(\cV)}([0]_\cV \amalg [0]_\cV , \cC)
\simeq
(\iota_0 \cC)^{\times 2}
\]
in $\Spaces$, whose fiber over a point $(X,Y) \in (\iota_0 \cC)^{\times 2}$ is the space
\[
\hom_\cV(V , \hom_\cC(X,Y))
\in
\Spaces
~.
\]
Hence, the compactness of $[1](V) \in \Cat(\cV)$ follows from the fact that the functor \Cref{und.segal.spaces.functor} preserves filtered colimits.
\end{itemize}
So indeed, $\Cat(\cV)$ is compactly generated. Now, because $[0]_\cV \in \Cat(\cV)$ is the unit object, to show that $\Cat(\cV)$ is cgsm it remains to show that for any $V,W \in \cV^\omega$ the object $[1](V) \boxtimes [1](W) \in \Cat(\cV)$ is compact. This follows from the identification
\[
\left(
[1](V) \boxtimes [1](W)
\right)
\simeq
\left(
[2](V,W)
\coprod_{[1](V \boxtimes W)}
[2](W,V)
\right)
\]
in $\Cat(\cV)$ along with the observation that $V \boxtimes W \in \cV^\omega$ is compact by the assumption that $\cV$ is cgsm.
\end{proof}

\begin{lemma}
\label{lem.compact.in.prbl.V.cat.detected.in.underlying}
Let $\cC$ be a presentable $\cV$-category.
\begin{enumerate}

\item\label{item.enrichedly.cpct.implies.unenrichedly.cpct} Every compact object of $\cC$ is compact in $U(\cC)$.

\item\label{item.unenrichedly.cpct.implies.enrichedly.cpct.if.action.preserves.cpcts} Suppose that for every pair of compact objects $V \in \cV^\omega$ and $C \in U(\cC)^\omega$, the object $(V \tensor^\flat C) \in U(\cC)$ is compact. Then every compact object of $U(\cC)$ is compact in $\cC$.

\end{enumerate}
\end{lemma}

\begin{proof}
We begin with part \Cref{item.enrichedly.cpct.implies.unenrichedly.cpct}. Suppose that $C \in \cC^\omega$; we would like to show that $C \in U(\cC)^\omega$. To do this, let $\cI \xlongra{D_\bullet} U(\cC)$ be any filtered diagram; we then compute that 
\begin{align}
\nonumber
\hom_{U(\cC)}(C, \colim_\cI^{U(\cC)}(D_\bullet))
&
:=
\hom_\cV(\uno_\cV, \hom_\cC(C, \colim_\cI^{U(\cC)}(D_\bullet)))
\\
\label{cmpct.obs.in.prbl.use.colims.computed.in.underlying}
&
\simeq 
\hom_\cV(\uno_\cV, \hom_\cC(C, \colim_\cI^\cC(D_\bullet)))
\\
\label{cmpct.obs.in.prbl.use.c.is.compact}
&
\simeq
\hom_\cV(\uno_\cV, \colim_\cI^{\ul{\cV}} \hom_\cC(C, D_\bullet))
\\
\label{cmpct.obs.in.prbl.use.colims.in.V.underline.computed.in.V}
&
\simeq
\hom_\cV(\uno_\cV, \colim_\cI^\cV \hom_\cC(C, D_\bullet))
\\
\label{cmpct.obs.in.prbl.use.uno.V.is.compact}
&
\simeq 
\colim^\Spaces_\cI \hom_\cV(\uno_\cV, \hom_\cC(C, D_\bullet))
\\
\nonumber
&
=:
\colim^\Spaces_\cI \hom_{U(\cC)}(C, D_\bullet)~,
\end{align}
where equivalences \Cref{cmpct.obs.in.prbl.use.colims.computed.in.underlying} and \Cref{cmpct.obs.in.prbl.use.colims.in.V.underline.computed.in.V} follow from Lemmas \ref{lem.limits.in.C.versus.UC}\ref{want.colimits.in.UC.to.be.colimits.in.C} and \ref{lem.presble.V.cats.admit.cotensors}, equivalence \Cref{cmpct.obs.in.prbl.use.c.is.compact} follows from the assumption that $C \in \cC^\omega$, and equivalence \Cref{cmpct.obs.in.prbl.use.uno.V.is.compact} follows from the fact that the unit $\uno_\cV \in \cV$ is compact.

We now turn to part \Cref{item.unenrichedly.cpct.implies.enrichedly.cpct.if.action.preserves.cpcts}. Suppose that $C \in U(\cC)^\omega$; we will show that under the stated assumption, for any filtered diagram $\cI \xlongra{D_\bullet} U(\cC)$ we have a canonical equivalence
\[
\hom_\cC(C, \colim_\cI^\cC(D_\bullet)) \simeq \colim_\cI^{\ul{\cV}} \hom_\cC(C, D_\bullet)
\]
in $\cV$, i.e.\! that $C \in \cC^\omega$. To verify this equivalence, by \Cref{localass.V.cmpctly.gnrtdly.symm.mon} it suffices to verify the equivalence after applying $\hom_\cV(V,-)$ for arbitrary $V \in \cV^\omega$. Indeed, we compute that
\begin{align}
\nonumber
\hom_\cV(V, \hom_\cC(C, \colim_\cI^\cC(D_\bullet)))
&
\simeq
\hom_{U(\cC)}(V \tensor^\flat C, \colim_\cI^\cC(D_\bullet))
\\
\label{proving.unenrichedly.cpct.implies.enrichedly.use.that.colims.in.C.are.colims.in.UC}
&
\simeq
\hom_{U(\cC)}(V \tensor^\flat C, \colim_\cI^{U(\cC)}(D_\bullet))
\\
\label{cmpct.obs.in.prbl.use.T.tensor.c.is.cmpct}
&
\simeq 
\colim_\cI^\Spaces \hom_{U(\cC)}(V \tensor^\flat C, D_\bullet)
\\
\nonumber
&
\simeq
\colim_\cI^\Spaces \hom_\cV(V, \hom_\cC(C, D_\bullet))
\\
\nonumber
&
\simeq
\hom_\cV(V, \colim_\cI^\cV \hom_\cC(C, D_\bullet))
\\
\label{proving.unenrichedly.cpct.implies.enrichedly.use.that.colims.in.ulV.are.colims.in.V}
&
\simeq
\hom_\cV(V, \colim_\cI^{\ul{\cV}} \hom_\cC(C, D_\bullet))
~,
\end{align}
where equivalences \Cref{proving.unenrichedly.cpct.implies.enrichedly.use.that.colims.in.C.are.colims.in.UC} and \Cref{proving.unenrichedly.cpct.implies.enrichedly.use.that.colims.in.ulV.are.colims.in.V} follow from Lemmas \ref{lem.limits.in.C.versus.UC}\ref{want.colimits.in.UC.to.be.colimits.in.C} and \ref{lem.presble.V.cats.admit.cotensors}, and \Cref{cmpct.obs.in.prbl.use.T.tensor.c.is.cmpct} follows from the assumption that $V \tensor^\flat C$ is compact in $U(\cC)$.
\end{proof}

\begin{definition}
\label{def.compactly.generated.V.category}
We say that a presentable $\cV$-category $\cC$ is \bit{compactly generated} if the action 
\[
\cV \otimes U(\cC) \xlongra{\tensor} U(\cC)
\]
lies in $\iota_1\PrL_\omega \subset \iota_1 \PrL$. In other words, this is the requirement that $U(\cC)$ is compactly generated and that there exists a factorization
\[ \begin{tikzcd}
\cV^\omega \times (U(\cC))^\omega
\arrow[hook]{r}
\arrow[dashed]{rrrd}
&
\cV \times U(\cC)
\arrow{r}
&
\cV \otimes U(\cC)
\arrow{r}{\tensor}
&
U(\cC)
\\
&
&
&
(U(\cC))^\omega
\arrow[hook]{u}
\end{tikzcd}~. \]
We write
\[
\iota_1\PrL_{\cV,\omega}
\subseteq
\iota_1\PrL_\cV
\]
for the subcategory on the compactly generated $\cV$-categories, whose morphisms are those morphisms in $\iota_1 \PrL_\cV$ that preserve compact objects.
\end{definition}

\begin{example}
\label{ex.V.is.compactly.generated.V.category}
By \Cref{localass.V.cmpctly.gnrtdly.symm.mon}, the self-enrichment $\ul{\cV} \in \Cat(\cV)$ is compactly generated. 
\end{example}

\begin{remark}
It is immediate that \Cref{thm.presble.V.cats.are.V.mods} refines to an equivalence
\[ \begin{tikzcd}
\iota_1\PrL_\cV
\arrow{r}{\sim}
&
\Mod_\cV(\iota_1 \PrL)
\\
\iota_1 \PrL_{\cV,\omega}
\arrow[dashed]{r}[swap]{\sim}
\arrow[hook]{u}
&
\Mod_\cV(\iota_1 \PrL_\omega)
\arrow[hook]{u}
\end{tikzcd} \]
between subcategories.
\end{remark}

\subsection{Semiadditive $\cV$-categories}
\label{subsec.semiadditive.V.categories}

In this subsection, we study semiadditivity in the enriched context. This notion has some particularly nice features when the enriching category $\cV$ is itself semiadditive. Namely, in this context we prove that semiadditive $\cV$-categories are equivalent to commutative monoids (\Cref{lem.semiadds.are.cmons}) (leading to a semiadditive envelope functor (\Cref{prop.Env.V})) and that semiadditive $\cV$-categories form a semiadditive $\Cat(\cV)$-category (\Cref{obs.cat.V.semiadd.is.semiadd}).

\begin{definition}
\label{def.unenr.semiadd}
We say that a category $\cC \in \Cat$ is \bit{semiadditive} if it admits finite products and coproducts and these agree (the comparison map existing once the empty coproduct and empty product agree). In this case, we refer to these finite products and coproducts as (\bit{direct}) \bit{sums}. Given objects $C,D \in \cC$, we write $C \oplus D \in \cC$ for their sum; we write $0_\cC \in \cC$ for the zero object (i.e.\! the empty sum).
\end{definition}

\begin{definition}
\label{def.enr.semiadd}
We say that a $\cV$-category $\cC \in \Cat(\cV)$ is \bit{semiadditive} if $U(\cC) \in \Cat$ is semiadditive and moreover $\cC$ admits finite products and coproducts (which therefore agree by \Cref{obs.conical.limits.are.limits.in.underlying}).\footnote{Clearly, taking $\cV = \Spaces$ recovers \Cref{def.unenr.semiadd}.} We write
\[
\Cat(\cV)^\sadd \subseteq \Cat(\cV)
\]
for the subcategory on the semiadditive $\cV$-categories whose morphisms are those functors that preserve finite sums. We also write
\[
\ul{\Cat(\cV)}^\sadd := (\Cat(\cV)^\sadd)^{\Cat(\cV)\enr} \subseteq \Cat(\cV)^{\Cat(\cV)\enr} =: \ul{\Cat(\cV)}
\]
for the 1-full sub-$\Cat(\cV)$-category (in the evident sense) on the semiadditive $\cV$-categories and finite-sum-preserving functors.
\end{definition}

\begin{observation}
\label{obs.change.of.enr.takes.sadd.to.sadd}
Let $\cV_0 \xra{F} \cV_1$ be a morphism in $\CAlg(\iota_1\PrL)$ with right adjoint $\cV_0 \xla{G} \cV_1$. Then, there exists a factorization
\[ \begin{tikzcd}
\Cat(\cV_0)
&
\Cat(\cV_1)
\arrow{l}[swap]{\Cat(G)}
\\
\Cat(\cV_0)^\sadd
\arrow[hook]{u}
&
\Cat(\cV_1)^\sadd
\arrow[hook]{u}
\arrow[dashed]{l}
\end{tikzcd}
~.
\]
\end{observation}

\begin{localass}
For the remainder of this subsection, we assume that $\cV \in \CAlg(\iota_1 \PrL)$ is semiadditive.
\end{localass}

\begin{remark}
All of the material in this subsection directly generalizes material that is standard in the case that $\cV = \Ab$.
\end{remark}

\begin{observation}
As the symmetric monoidal structure of $\cV$ commutes with colimits separately in each variable, in particular it commutes with direct sums separately in each variable. We use this fact without further comment.
\end{observation}

\begin{observation}
\label{obs.semiadd.enr.gives.semiadd}
In any $\cV$-category $\cC \in \Cat(\cV)$, finite products and coproducts coincide. That is, if a finite coproduct exists then it is also a finite product, and conversely.
\end{observation}

\begin{definition}
\label{defn.direct.sums}
In view of \Cref{obs.semiadd.enr.gives.semiadd}, we refer to finite co/products in a $\cV$-category $\cC \in \Cat(\cV)$ as (\bit{direct}) \bit{sums}. Given objects $C,D \in \cC$, we write $C \oplus D \in \cC$ for their sum, and we write $0_\cC \in \cC$ for a zero object if it exists.
\end{definition}

\begin{example}
\label{ex.self.enr.of.sadd.is.sadd}
The self-enrichment $\ul{\cV} \in \Cat(\cV)$ is semiadditive.
\end{example}

\begin{observation}
\label{obs.semiadd.enriched.implies.automatic.finite.sum.preservation}
As sums can be detected diagramatically (in a way that can be upgraded to $\cV$-enrichments via Yoneda), any morphism $\cC \ra \cD$ in $\Cat(\cV)$ automatically preserves any finite sums that exist in $\cC$. In other words, the subcategory 
\[
\Cat(\cV)^\sadd \subseteq \Cat(\cV)
\]
is full (and not merely 1-full), and likewise the sub-$\Cat(\cV)$-category
\[
\ul{\Cat(\cV)}^\sadd \subseteq \ul{\Cat(\cV)}
\]
is full.
\end{observation}

\begin{observation}
\label{obs.Cat.of.V.has.products}
Because $\Cat(\cV)$ is presentable, it admits finite products. These commute with the passage to underlying groupoids (because the functor $\Cat(\cV) \xra{U} \Cat$ commutes with limits). Moreover, for any $\cC,\cD \in \Cat(\cV)$ and $C,C' \in \cC$ and $D,D' \in \cD$, we have an equivalence
\[
\hom_{\cC \times \cD} ((C,D) , (C',D'))
\simeq
\hom_\cC ( C,C') \times \hom_\cD(D,D')
=:
\hom_\cC ( C,C') \oplus \hom_\cD(D,D')
\]
in $\cV$. It follows that pointwise sums in $\cC \times \cD$ are sums.
\end{observation}

\begin{observation}
We use the following facts about the zero object $0_\cV \in \cV$ without further comment.
\begin{enumerate}

\item It admits a canonical structure of a commutative algebra object (with respect to $\boxtimes_\cV$).

\item Considered as an object $0_\cV \in \Alg(\cV)$, it is coempty: there are no morphisms out of it besides its identity morphism.

\item It is contagious under $\boxtimes$, because it is an empty colimit and $\boxtimes$ commutes with colimits separately in each variable.

\item It is its only module.

\end{enumerate}
\end{observation}

\begin{observation}
\label{obs.end.zero.means.zero}
An object of $\cC \in \Cat(\cV)$ is a zero object and only if its endomorphism algebra is $0_\cV \in \Alg(\cV)$.
\end{observation}

\begin{lemma}
\label{lem.semiadds.are.cmons}
The forgetful functor 
\[
\CAlg(\Cat(\cV), \times) =: \CMon(\Cat(\cV))
\xlongra{\fgt} 
\Cat(\cV)
\]
factors as an equivalence 
\[
\begin{tikzcd}[column sep = 1.5cm]
\CMon(\Cat(\cV))
\arrow{r}{\fgt}
\arrow[dashed]{rd}[sloped, anchor=north, swap]{\sim}
&
\Cat(\cV) 
\\
&
\Cat(\cV)^\sadd 
\arrow[hook]{u}
\end{tikzcd}
~.
\]

\end{lemma}

\begin{proof}
The forgetful functor is fully faithful by \Cref{obs.semiadd.enriched.implies.automatic.finite.sum.preservation}. 
So, suppose that $(\cC ,\boxplus) \in \CMon(\Cat(\cV))$. First of all, the unit morphism
\[
\pt_{\Cat(\cV)}
\simeq
\fB 0_\cV
\longra
\cC
\]
selects an object $I \in \cC$ equipped with a morphism $0_\cV \ra \End_\cC(I)$ of algebra objects of $\cV$, which implies that $I \in \cC$ is a zero object by \Cref{obs.end.zero.means.zero}. Thereafter, the unitality of $\boxplus$ implies that for any $C \in \cC$ we have
\[
C \boxplus 0_\cC \simeq C \simeq 0_\cC \boxplus C
~.
\]
Now, for any $C,D \in \cC$, by \Cref{obs.Cat.of.V.has.products} we have a composite equivalence
\[
(C,D)
\simeq
(C \oplus 0_\cC , 0_\cC \oplus D)
\simeq
(C,0_\cC) \oplus (0_\cC,D)
\]
in $\cC \times \cC$. On the other hand, the morphism
\[
\cC \times \cC
\xlongra{\boxplus}
\cC
\]
in $\Cat(\cV)$ preserves sums by \Cref{obs.semiadd.enriched.implies.automatic.finite.sum.preservation}. It follows that we have a canonical composite equivalence
\[
C \boxplus D
\simeq
(C \boxplus 0_\cC) \oplus (0_\cC \boxplus D)
\simeq
C \oplus D
\]
in $\cC$. So the commutative monoid structure $\boxplus$ is canonically equivalent to $\oplus$ (and in particular, $\cC$ is semiadditive).
\end{proof}

\begin{observation}
\label{obs.cat.V.semiadd.is.semiadd}
It follows directly from \Cref{lem.semiadds.are.cmons} that the $\Cat(\cV)$-category $\ul{\Cat(\cV)}^\sadd \in \what{\Cat}(\Cat(\cV))$ is semiadditive.
\end{observation}

\begin{proposition}
\label{prop.Env.V}
There exists a left adjoint localization functor
\begin{equation}
\label{CatV.enriched.Env.V.adjn}
\begin{tikzcd}[column sep=2cm]
\Cat(\cV)
\arrow[dashed, yshift=0.9ex]{r}{\Env_\cV^\oplus}
\arrow[hookleftarrow, yshift=-0.9ex]{r}[yshift=-0.2ex]{\bot}[swap]{\fgt}
&
\Cat(\cV)^\sadd
\end{tikzcd}
~.
\end{equation}
Moreover, this left adjoint is compatible with the symmetric monoidal structure of $\Cat(\cV)$, i.e.\! for any $\cE,\cF \in \Cat(\cV)$ the morphism
\[
\Env^\oplus_\cV \left( \cE \boxtimes \left( \cF \xlongra{\eta} \fgt\left(\Env^\oplus_\cV (\cF)\right) \right) \right)
\]
in $\Cat(\cV)^\sadd$ is an equivalence.
\end{proposition}

\begin{definition}
\label{def.env.v}
The left adjoint $\Env^\oplus_\cV$ of adjunction \Cref{CatV.enriched.Env.V.adjn} is called the (\bit{$\cV$-enriched}) \bit{semiadditive envelope}.
\end{definition}

\begin{proof}[Proof of \Cref{prop.Env.V}]
It follows immediately from \Cref{lem.semiadds.are.cmons} that $\Cat(\cV)^\sadd$ is presentable, and thereafter that the left adjoint $\Env^\oplus_\cV$ indeed exists.\footnote{Note that for any presentable category $\cC$, $\CMon(\cC)$ is an accessible localization of $\Fun ( \Fin_* , \cC)$.}

In order to proceed, we first show that the weak cotensor of $\cC \in \Cat(\cV)^\sadd$ by $\cE \in \Cat(\cV)$ is given by the formula
\begin{equation}
\label{cotensor.formula.for.CatVsadd.over.CatV}
\cE \cotensor^\flat \cC
\simeq
\hom_{\Cat(\cV)}(\cE , \fgt(\cC))
\in
\Cat(\cV)
~.
\end{equation}
First of all, this hom $\cV$-category is semiadditive by \Cref{lem.semiadds.are.cmons} (along with the fact that products in $\Cat(\cV)$ compute products in $\ul{\Cat(\cV)}$ by \Cref{example.self.enrichment.of.V.is.a.prbl.V.cat} and \Cref{lem.limits.in.C.versus.UC}\Cref{want.limits.in.UC.to.be.limits.in.C}). Thereafter, for any $\cD \in \Cat(\cV)^\sadd$ we have equivalences
\begin{align}
\label{use.ff.of.CatVsadd.in.CatV.first.time.for.cotensor}
\hom_{{\Cat(\cV)}^\sadd}
\left(
\cD
,
\hom_{\ul{\Cat(\cV)}} ( \cE , \fgt(\cC) )
\right)
& 
\simeq
\hom_{{\Cat(\cV)}}
\left(
\fgt(\cD)
,
\hom_{\ul{\Cat(\cV)}} ( \cE , \fgt(\cC) )
\right)
\\
\nonumber
& \simeq
\hom_{{\Cat(\cV)}}
\left(
\cE
,
\hom_{\ul{\Cat(\cV)}} ( \fgt(\cD) , \fgt(\cC) )
\right)
\\
\label{use.ff.of.CatVsadd.in.CatV.second.time.for.cotensor}
& \simeq
\hom_{{\Cat(\cV)} }
\left(
\cE
,
\hom_{\ul{\Cat(\cV)}^\sadd} ( \cD , \cC )
\right)
~,
\end{align}
where equivalences \Cref{use.ff.of.CatVsadd.in.CatV.first.time.for.cotensor} and \Cref{use.ff.of.CatVsadd.in.CatV.second.time.for.cotensor} follow from \Cref{obs.semiadd.enriched.implies.automatic.finite.sum.preservation}. 

We now verify the compatibility of $\Env^\oplus_\cV$ with the symmetric monoidal structure of $\Cat(\cV)$. For this, we simplify our notation by writing $L := \Env^\oplus_\cV$ and $R := \fgt$. Then, for any $\cC \in \Cat(\cV)^\sadd$ we compute that
\begin{align}
\nonumber
\hom_{\Cat(\cV)^\sadd}
(
L ( \cE \boxtimes RL(\cF) )
,
\cC
)
& \simeq
\hom_{\Cat(\cV) } ( \cE \boxtimes RL(\cF) , R(\cC) )
\\
\nonumber
& \simeq
\hom_{\Cat(\cV)} \left( RL(\cF) , \hom_{\ul{\Cat(\cV)}} ( \cE , R(\cC) ) \right)
\\
\label{use.cotensor.of.CatVsadd.over.CatV.first.time}
& \simeq
\hom_{\Cat(\cV)} \left( RL(\cF) , R \left(\cE \cotensor^\flat \cC \right) \right)
\\
\label{use.ff.for.checking.compatibility.of.sm.str}
& \simeq
\hom_{\Cat(\cV)^\sadd} \left( L(\cF) , \cE \cotensor^\flat \cC \right)
\\
\nonumber
& \simeq
\hom_{\Cat(\cV)} \left( \cF , R \left( \cE \cotensor^\flat \cC \right) \right)
\\
\label{use.cotensor.of.CatVsadd.over.CatV.second.time}
& \simeq
\hom_{\Cat(\cV)} \left( \cF , \hom_{\ul{\Cat(\cV)}} ( \cE , R(\cC) ) \right)
\\
\nonumber
& \simeq
\hom_{\Cat(\cV)} ( \cE \boxtimes \cF , R(\cC) )
\\
\nonumber
& \simeq
\hom_{\Cat(\cV)^\sadd}( L(\cE \boxtimes \cF) , \cC )
~,
\end{align}
where equivalences \Cref{use.cotensor.of.CatVsadd.over.CatV.first.time} and \Cref{use.cotensor.of.CatVsadd.over.CatV.second.time} follow from equivalence \Cref{cotensor.formula.for.CatVsadd.over.CatV} and equivalence \Cref{use.ff.for.checking.compatibility.of.sm.str} follows from \Cref{obs.semiadd.enriched.implies.automatic.finite.sum.preservation}.
\end{proof}

\begin{observation}
\label{obs.consequences.of.Env.V}
\Cref{prop.Env.V} immediately upgrades $\Env^\oplus_\cV$ to a morphism in $\CAlg(\iota_1 \PrL)$, where we endow $\Cat(\cV)^\sadd$ with the symmetric monoidal structure given by the formula
\[
\cC \boxtimes \cD
:=
\Env^\oplus_\cV \left( \fgt(\cC) \boxtimes \fgt(\cD) \right)
~.
\]
Moreover, by \Cref{obs.consequences.of.morphism.in.CAlg.PrL}\Cref{item.presly.s.m.loczn}, the tensor and cotensor of $\cC \in \ul{\Cat(\cV)}^\sadd$ by $\cE \in \ul{\Cat(\cV)}$ are given by the formulas
\[
\cE \tensor \cC
\simeq
\Env^\oplus_\cV ( \cE \boxtimes \fgt(\cC) )
\qquad
\text{and}
\qquad
\cE \cotensor \cC
\simeq
\hom_{\ul{\Cat(\cV)}}(\cE , \fgt(\cC) )
~,
\]
and the self-enrichment of $\Cat(\cV)$ restricts to the self-enrichment of $\Cat(\cV)^\sadd$.
\end{observation}

\begin{observation}
\label{obs.Cat.V.semiadd.is.compactly.generated}
It follows from \Cref{prop.Env.V} and \Cref{obs.consequences.of.Env.V} that if $\cV$ is cgsm then so is $\Cat(\cV)^\sadd$.
\end{observation}

\section{Some higher category theory}
\label{app.some.higher.cat.theory}

In this section, we show that stable 2-categories can be studied both in enriched terms (as certain $\iota_1\St$-enriched categories) as well as in unenriched terms (as certain 2-categories); this basic consistency check is established in \Cref{subsec.stably.enriched.2.cats}. We begin with an analogous discussion of stable 1-categories in \Cref{subsec.stable.cats}.

\begin{remark}
Throughout this section (particularly in \Cref{subsec.stably.enriched.2.cats}), we make an effort to use precise notation even though it is somewhat cumbersome, because we find it to be clarifying. In particular, we often include forgetful functors in our notation here that we may omit elsewhere in the paper.
\end{remark}

\subsection{Stable categories}
\label{subsec.stable.cats}

In this subsection we review some basic features of stable categories, as both a warm-up and preparatory material for our study of stably-enriched 2-categories in \Cref{subsec.stably.enriched.2.cats}. Notably, we establish the consistency check \Cref{lem.St.is.self.enr.of.iota.one.St} (of which the consistency check \Cref{prop.only.one.Cat.St} is a categorification). We also prove that $\St$-valued presheaves form a presentable 2-category (\Cref{prop.Fun.K.St.is.presentable}).

\begin{notation}
\label{notn.stablecat}
We write $\St \subseteq \Cat$ for the 1-full sub-2-category on the stable categories and exact functors. We write $\Fun^\ex(\cC,\cD) \subseteq \Fun(\cC,\cD)$ for the full subcategory on the exact functors (i.e.\! the hom-category in $\St$).
\end{notation}

\begin{notation}
We write $\otimes$ for symmetric monoidal structure on $\iota_1 \St$, whose unit object is $\Spectra^\fin$. This monoidal structure corepresents biexact functors: for stable categories $\cC, \cD,\cE \in \St$, an exact functor $\cC \otimes \cD \ra \cE$ is equivalent data to a functor $\cC \times \cD \ra \cE$ that is exact separately in each variable.
\end{notation}

\begin{observation}
\label{obs.condition.for.map.of.stable.cats.to.factor.thru.univ.biexact}
For $\cC,\cD,\cE \in \St$, it is merely a condition for a functor $\cC \times \cD \ra \cE$ to factor through the universal biexact functor $\cC \times \cD \ra \cC \otimes \cD$, because it is merely a condition for the composite functors
\[
\cC \xra{(\id_\cC,0)} \cC \times \cD \longra \cE
\qquad
\text{and}
\qquad
\cD \xra{(0,\id_\cD)} \cC \times \cD \longra \cE
\]
between stable categories to be exact.
\end{observation}

\begin{observation}
\label{obs.St.CGSM.and.get.iota.1.St.iota.1.Cat.enr}
The categories $\iota_1\Cat$ and $\iota_1\St$ participate in an adjunction
\begin{equation}
\label{unenriched.PSpfin.adjn}
\begin{tikzcd}[column sep = 2cm]
\iota_1 \Cat 
\arrow[yshift=0.9ex]{r}{\cP_\Spectra^\fin}
\arrow[hookleftarrow, yshift=-0.9ex]{r}[yshift=-0.2ex]{\bot}[swap]{\fgt}
&
\iota_1 \St
\end{tikzcd}
\end{equation}
taking a category $\cC$ to its category $\cP_\Spectra^\fin(\cC)$ of finite spectral presheaves (i.e.\! the smallest stable subcategory of $\Fun ( \cC^\op , \Spectra)$ containing the image of the composite $\cC \hookrightarrow \Fun ( \cC^\op , \Spaces) \xra{\Sigma^\infty_+} \Fun ( \cC^\op , \Spectra)$). Moreover, the left adjoint $\cP_\Spectra^\fin$ defines a morphism in $\CAlg(\iota_1 \PrL_\omega)$ (in particular, $\iota_1 \St$ is cgsm). By \Cref{thm.presble.V.cats.are.V.mods} and \Cref{obs.consequences.of.morphism.in.CAlg.PrL}\Cref{obs.for.Cathat.of.radjt.to.presly.s.m.fctr.both.directions.agree}, we obtain objects
\[
\ul{\iota_1 \St}
\in
\iota_1 \PrL_{\iota_1 \St}
\qquad
\text{and}
\qquad
(\iota_1 \St)^{\iota_1 \Cat\enr}
:=
\what{\Cat}(\fgt)(\ul{\iota_1 \St})
\in
\iota_1 \PrL_{\iota_1 \Cat}
=:
\iota_1 \PrL_2
\subset
\iota_1 \what{\Cat}_2
\]
as well as an adjunction
\begin{equation}
\label{enriched.PSpfin.adjn}
\begin{tikzcd}[column sep = 2cm]
\Cat
:=
\ul{\iota_1 \Cat}
\arrow[yshift=0.9ex]{r}{\cP_\Spectra^\fin}
\arrow[hookleftarrow, yshift=-0.9ex]{r}[yshift=-0.2ex]{\bot}[swap]{\fgt}
&
(\iota_1 \St)^{\iota_1 \Cat\enr}
\end{tikzcd}
\end{equation}
in $\iota_1 \what{\Cat}_2$ (which recovers the adjunction \Cref{unenriched.PSpfin.adjn} upon applying $\iota_1 \what{\Cat}_2 \xra{\iota_1} \iota_1 \what{\Cat}$).
\end{observation}

We have the following consistency check, that the two 2-categorical enhancements of $\iota_1\St$ --- that given in \Cref{notn.stablecat} and that of \Cref{obs.St.CGSM.and.get.iota.1.St.iota.1.Cat.enr} --- agree.

\begin{lemma}
\label{lem.St.is.self.enr.of.iota.one.St}
The right adjoint of adjunction \Cref{enriched.PSpfin.adjn} factors as an equivalence
\[ \begin{tikzcd}
\Cat
\arrow[hookleftarrow]{r}{\fgt}
&
(\iota_1 \St)^{\iota_1\Cat\enr}
\arrow[dashed]{ld}[sloped, anchor = north]{\sim}
\\
\St
\arrow[hook]{u}
\end{tikzcd} \]
in $\iota_1 \what{\Cat}_2$. In particular, $\Fun^\ex(-,-)$ is adjoint to $(-) \otimes (-)$.
\end{lemma}

\begin{observation}
\label{obs.cotensor.of.stable.cat.by.cat}
By \Cref{lem.presble.V.cats.admit.cotensors}, $(\iota_1 \St)^{\iota_1 \Cat\enr} \in \iota_1 \PrL_2$ admits cotensors, and moreover by \Cref{obs.radjt.in.presbly.V.enriched.adjn.commutes.with.cotensors} these commute with the right adjoint $\Cat \xla{\fgt} (\iota_1 \St)^{\iota_1 \Cat\enr}$. In other words, for any $\cK \in \iota_1\Cat$ and any $\cC \in (\iota_1\St)^{\iota_1 \Cat\enr}$, we have a stable category $\cK \cotensor \cC \in (\iota_1\St)^{\iota_1 \Cat\enr}$ and an equivalence
\[
\fgt(\cK \cotensor \cC)
\simeq
\cK \cotensor \fgt(\cC)
=:
\Fun ( \cK , \fgt(\cC) )
\in
\Cat
~.\footnote{In particular, the category $\Fun(\cK, \fgt(\cC)) \in \Cat$ is stable. (Of course, this can also be deduced directly.)}
\]
\end{observation}

\begin{proof}[Proof of \Cref{lem.St.is.self.enr.of.iota.one.St}]
It is clear that the asserted factorization exists and is an equivalence on underlying 1-categories. So, it remains to check that it induces an equivalence on hom 1-categories. For this, let $\cC,\cD \in \iota_1 \St$ be stable categories. For any $\cK \in \iota_1 \Cat$, we have the two inclusions
\begin{align}
\nonumber
\hom_{\iota_1 \Cat}( \cK , \hom_{ (\iota_1 \St)^{\iota_1 \Cat\enr} } ( \cC , \cD ) )
& \simeq
\hom_{\iota_1 \St} \left( \cC , \cK \cotensor^\flat \cD \right)
\\
\label{use.that.cotensor.from.a.cat.into.a.stable.cat.is.functors}
& \simeq
\hom_{\iota_1 \St} ( \cC , \Fun ( \cK , \cD ) )
\\
\label{include.homs.in.iota.one.St.to.homs.in.iota.one.Cat}
& \subseteq
\hom_{\iota_1 \Cat} ( \cC , \Fun ( \cK , \cD ) )
\\
\nonumber
& \simeq
\hom_{\iota_1 \Cat} ( \cK , \Fun ( \cC , \cD ) )
\\
\label{include.functors.into.Fun.ex}
& \supseteq
\hom_{\iota_1 \Cat} ( \cK , \Fun^\ex ( \cC , \cD ) )
~,
\end{align}
where equivalence \Cref{use.that.cotensor.from.a.cat.into.a.stable.cat.is.functors} (as well as the existence of the weak cotensor $\cK \cotensor^\flat \cD \in (\iota_1 \St)^{\iota_1 \Cat\enr}$) follows from \Cref{obs.cotensor.of.stable.cat.by.cat}. Now, the inclusion \Cref{include.functors.into.Fun.ex} is of the subspace on those functors $\cK \xra{F} \Fun ( \cC , \cD )$ such that for all $K \in \cK$ the functor $\cC \xra{F(K)} \cD$ is exact. On the other hand, the inclusion \Cref{include.homs.in.iota.one.St.to.homs.in.iota.one.Cat} is of the subspace on those functors $\cC \xra{G} \Fun ( \cK , \cD)$ that are exact, and these are characterized by the requirement that for all $K \in \cK$ the composite $\cC \xra{G} \Fun ( \cK , \cD ) \xra{\ev_K} \cD$ is exact. Thus these inclusions are of the same subspace.
\end{proof}

\begin{observation}
\label{obs.St.is.semiadd}
The 1-category $\iota_1\St \in \iota_1 \what{\Cat}$ is semiadditive. Hence by \Cref{ex.self.enr.of.sadd.is.sadd}, $\ul{\iota_1\St} \in \what{\Cat}(\iota_1 \St)$ is semiadditive. Now, by \Cref{obs.change.of.enr.takes.sadd.to.sadd} and \Cref{lem.St.is.self.enr.of.iota.one.St}, we find that 
\[
\St \simeq (\iota_1\St)^{\iota_1\Cat\enr} \in \iota_1\what{\Cat}_2
\]
is semiadditive.
\end{observation}

\begin{proposition}
\label{prop.Fun.K.St.is.presentable}
For any small 2-category $\cB \in \Cat_2$, the 2-category $\Fun(\cB,\St) \in \what{\Cat}_2$ is presentable.
\end{proposition}

\begin{proof}
Let us write $\iota_1 \cB \xhookrightarrow{i} \cB$ for the inclusion. Because the functor $\St \xra{\fgt} \Cat$ is 1-full, we have a pullback square
\begin{equation}
\label{pullback.square.for.functors.to.St}
\begin{tikzcd}[column sep = 2.5cm, row sep = 2cm]
\Fun ( \cB , \St)
\arrow[hook]{r}{\Fun ( \cB , \fgt)}
\arrow{d}[swap]{\Fun ( i , \St)}
&
\Fun ( \cB , \Cat)
\arrow{d}{\Fun ( i , \Cat)}
\\
\Fun ( \iota_1 \cB , \St)
\arrow[hook]{r}[swap]{\Fun ( \iota_1 \cB , \fgt)}
&
\Fun(\iota_1 \cB , \Cat)
\end{tikzcd}
\end{equation}
in $\what{\Cat}_2$. We first claim that the pullback square $\iota_1\Cref{pullback.square.for.functors.to.St}$ in $\what{\Cat}$ lies in $\iota_1 \PrR$. 
By \Cref{prop.Fun.B.Cat.is.prbl.and.tensoring.is.ptwise}, the category $\Fun(\cB, \Cat)$ is presentable, and clearly
\[
\iota_1 \Fun ( \iota_1 \cB , \St)
\simeq
\Fun ( \iota_1 \cB , \iota_1 \St)
\qquad
\text{and}
\qquad
\iota_1 \Fun ( \iota_1 \cB , \Cat)
\simeq
\Fun ( \iota_1 \cB , \iota_1 \Cat)
\]
are presentable. Hence, it suffices to observe the existence of the left adjoints
\begin{equation}
\label{cospan.on.one.cats.from.pullback.square.for.functors.to.St}
\begin{tikzcd}[column sep = 2.5cm, row sep = 2cm]
&
\iota_1 \Fun ( \cB , \Cat)
\arrow[leftarrow, dashed, xshift=-0.9ex]{d}[swap]{i_!}
\arrow[xshift=0.9ex]{d}{\iota_1 \Fun ( i , \Cat)}[swap, xshift=0.2ex]{\rotatebox{90}{$\bot$}}
\\
\iota_1 \Fun ( \iota_1 \cB , \St)
\arrow[hook, yshift=-0.9ex]{r}[yshift=-0.2ex]{\bot}[swap]{\iota_1 \Fun ( \iota_1 \cB , \fgt)}
\arrow[dashed, leftarrow, yshift=0.9ex]{r}{\iota_1 \Fun ( \iota_1 \cB , \cP_\Spectra^\fin)}
&
\iota_1 \Fun(\iota_1 \cB , \Cat)
\end{tikzcd}
~,
\end{equation}
where the vertical left adjoint exists by \cite[Proposition 3.3.1]{GHL} (see also \cite[Remark 3.3.4]{GHL}) and the horizontal left adjoint exists by \Cref{obs.St.CGSM.and.get.iota.1.St.iota.1.Cat.enr}. 
Now, passing to left adjoints in the pullback square $\iota_1\Cref{pullback.square.for.functors.to.St}$ in $\iota_1 \PrR$ yields a pushout square in $\iota_1 \PrL$. We claim that this is in fact a pushout square in $\Mod_{\iota_1 \Cat}(\iota_1 \PrL)$, i.e.\! that the left adjoints in diagram \Cref{cospan.on.one.cats.from.pullback.square.for.functors.to.St} are $\iota_1 \Cat$-linear.\footnote{Note that the forgetful functor $\Mod_{\iota_1 \Cat}(\iota_1 \PrL) \xra{\fgt} \iota_1 \PrL$ admits a right adjoint (using that $\iota_1 \PrL$ is closed symmetric monoidal) and so preserves colimits.} The $\iota_1\Cat$-linearity of the horizontal left adjoint $\iota_1 \Fun ( \iota_1 \cB , \cP_\Spectra^\fin)$ follows from \Cref{obs.St.CGSM.and.get.iota.1.St.iota.1.Cat.enr}, while the $\iota_1\Cat$-linearity of the vertical left adjoint $i_!$ follows from Lemmas \ref{lem.presble.V.cats.admit.cotensors} and \ref{lem.G.a.right.adjoint.between.V.categories.iff}. So, the claim follows from  \Cref{thm.presble.V.cats.are.V.mods}.
\end{proof}

\subsection{Stably-enriched 2-categories}
\label{subsec.stably.enriched.2.cats}

In this subsection we study stably-enriched 2-categories; the primary output is the consistency check \Cref{prop.only.one.Cat.St}.

\begin{definition}
\label{defn.Cat.of.Stable}
A 2-category is called \bit{stably-enriched} if its hom-categories are stable and its composition bifunctors are biexact (which is merely a condition by \Cref{obs.condition.for.map.of.stable.cats.to.factor.thru.univ.biexact}). Given two stably-enriched 2-categories $\cX$ and $\cY$, we say that a functor
\[
\cX
\xlongra{F}
\cY
\]
is \bit{2-exact} if for all $A,B \in \cX$ the functor
\[
\hom_\cX(A,B)
\longra
\hom_\cY(F A , F B)
\]
between stable categories is exact (i.e.\! it lies in $\St \subseteq \Cat$). We write
\[
\Cat(\St)
\subseteq
\Cat_2
\]
for the 1-full sub-3-category of small stably-enriched 2-categories and 2-exact functors.\footnote{Note that this is not quite an instance of \Cref{notn.V.enr.cats}, which only applies to monoidal 1-categories (as opposed to monoidal 2-categories).}
Given stably-enriched 2-categories $\cX,\cY \in \Cat(\St)$, we write 
\[
\Fun^\twoex(\cX, \cY) := \hom_{\Cat(\St)}(\cX, \cY) \subseteq \Fun(\cX, \cY)
\]
for the 2-category of 2-exact functors between them.
\end{definition}

\begin{observation}
\label{obs.products.and.coproducts.in.stably.enriched.two.cat}
By Observations \ref{obs.St.is.semiadd} and \ref{obs.semiadd.enr.gives.semiadd}, if $\cX \in \Cat(\St)$, then finite products and coproducts in $\cX$ agree (so that we refer to them as finite sums). Moreover, by \Cref{obs.semiadd.enriched.implies.automatic.finite.sum.preservation}, 2-exact functors automatically preserve finite sums.
\end{observation}

\begin{observation}
\label{obs.stably.enriched.is.unambiguous.on.iota.one}
The functor
\[
\Cat(\iota_1 \St)
\xra{\Cat(\iota_1 \fgt)}
\Cat(\iota_1 \Cat)
=:
\iota_1 \Cat_2
\]
is a monomorphism in $\iota_1\what{\Cat}$ (using \Cref{obs.condition.for.map.of.stable.cats.to.factor.thru.univ.biexact} and the fact that $\iota_1\St \xra{\iota_1\fgt} \iota_1\Cat$ is a monomorphism in $\iota_1 \what{\Cat}$). It follows that $\Cat(\St) \subseteq \Cat_2$ is the 1-full sub-3-category generated by the image of the composite functor
\[
\Cat(\iota_1 \St)
\xra{\Cat(\iota_1 \fgt)}
\Cat(\iota_1 \Cat)
=:
\iota_1 \Cat_2
\longhookra
\Cat_2
~,
\]
and moreover that we have an equivalence of categories 
\[
\iota_1\Cat(\St) \simeq \Cat(\iota_1\St)~.
\]
\end{observation}

\begin{observation}
\label{obs.iota.one.Cat.St.CGSM.and.get.iota.one.Cat.St.iota.one.cat.two.enr}
The categories $\iota_1 \Cat_2$ and $\iota_1\Cat(\St)$ participate in an adjunction
\begin{equation}
\label{adjn.between.cat2.and.cat.st.on.underlying}
\begin{tikzcd}[column sep = 2cm]
\iota_1 \Cat_2 := \Cat(\iota_1 \Cat) 
\arrow[yshift=0.9ex]{r}{\Cat(\cP_\Spectra^\fin)}
\arrow[hookleftarrow, yshift=-0.9ex]{r}[yshift=-0.2ex]{\bot}[swap]{\Cat(\fgt)}
&
\Cat(\iota_1\St)
\simeq 
\iota_1\Cat(\St)
\end{tikzcd}
\end{equation}
(using the equivalence of \Cref{obs.stably.enriched.is.unambiguous.on.iota.one}). Moreover, the left adjoint $\Cat(\cP_\Spectra^\fin)$ defines a morphism in $\CAlg(\iota_1\PrL_\omega)$ by \Cref{lem.if.V.is.cgsm.then.Cat.V.is.cgsm}. By \Cref{thm.presble.V.cats.are.V.mods} and \Cref{obs.consequences.of.morphism.in.CAlg.PrL}\Cref{obs.for.Cathat.of.radjt.to.presly.s.m.fctr.both.directions.agree}, we obtain objects 
\[
\hspace{-1.3cm}
\ul{\iota_1\Cat(\St)} \in \iota_1 \PrL_{\iota_1\Cat(\St)} 
\qquad 
\text{and}
\qquad 
(\iota_1\Cat(\St))^{\iota_1\Cat_2\enr} 
:= 
\what{\Cat}(\Cat(\fgt))(\ul{\iota_1\Cat(\St)}) 
\in 
\iota_1\PrL_{\iota_1\Cat_2} 
=:
\iota_1\PrL_3 
\subset 
\iota_1\what{\Cat_3}
\]
as well as an adjunction 
\begin{equation} 
\label{adj.between.Cat.2.and.iota.1.Cat.St.iota.1.Cat.2.enr}
\begin{tikzcd}[column sep = 1.5cm]
\Cat_2
:= 
\ul{\iota_1\Cat_2}
\arrow[yshift=0.9ex]{r}{\what{\Cat}(\cP_\Spectra^\fin)}
\arrow[hookleftarrow, yshift=-0.9ex]{r}[yshift=-0.2ex]{\bot}[swap]{\what{\Cat}(\fgt)}
&
(\iota_1\Cat(\St))^{\iota_1\Cat_2\enr}
\end{tikzcd}
\end{equation}
in $\iota_1\what{\Cat}_3$ (which recovers the adjunction \Cref{adjn.between.cat2.and.cat.st.on.underlying} upon applying $\iota_1\what{\Cat}_3 \xra{\iota_1} \iota_1 \what{\Cat}$).
\end{observation}

\begin{notation}
We write $\otimes$ for the symmetric monoidal structure on $\ul{\iota_1\Cat(\St)} \in \iota_1 \PrL_{\iota_1 \Cat(\St)}$ that results from \Cref{obs.iota.one.Cat.St.CGSM.and.get.iota.one.Cat.St.iota.one.cat.two.enr}.
\end{notation}

We have the following consistency check, that the two 3-categorical enhancements of $\iota_1\Cat(\St)$ -- that given in \Cref{defn.Cat.of.Stable} and that of \Cref{obs.iota.one.Cat.St.CGSM.and.get.iota.one.Cat.St.iota.one.cat.two.enr} -- agree.

\begin{proposition}
\label{prop.only.one.Cat.St}
The right adjoint of adjunction \Cref{adj.between.Cat.2.and.iota.1.Cat.St.iota.1.Cat.2.enr} factors as an equivalence 
\[
\begin{tikzcd} 
\Cat_2 
\arrow[hookleftarrow]{r}{\fgt}
&
(\iota_1\Cat(\St))^{\iota_1\Cat_2\enr}
\arrow[dashed]{dl}[sloped, anchor=north]{\sim}
\\
\Cat(\St)
\arrow[hook]{u}{\fgt} 
&
\end{tikzcd}
\]
in $\iota_1 \what{\Cat}_3$. In particular, $\Fun^\twoex(-,-)$ is adjoint to $(-) \otimes (-)$.
\end{proposition}

\begin{observation}
\label{obs.functors.into.a.stably.enriched.are.stably.enriched}
By \Cref{lem.presble.V.cats.admit.cotensors}, $(\iota_1\Cat(\St))^{\iota_1\Cat_2\enr} \in \iota_1\PrL_3$ admits cotensors, and moreover by \Cref{obs.radjt.in.presbly.V.enriched.adjn.commutes.with.cotensors} these commute with the right adjoint $\Cat_2 \xla{\Cat(\fgt)} (\iota_1\Cat(\St))^{\iota_1\Cat_2\enr}$. In other words, for any $\cK \in \iota_1\Cat_2$ and any $\cC \in (\iota_1\Cat(\St))^{\iota_1\Cat_2\enr}$, we have a stably-enriched category $\cK \cotensor \cC \in (\iota_1\Cat(\St))^{\iota_1\Cat_2\enr}$ and an equivalence 
\[
\Cat(\fgt)(\cK \cotensor \cC) 
\simeq 
\cK \cotensor \Cat(\fgt)(\cC)
=:
\Fun(\cK, \Cat(\fgt)(\cC))
\in 
\Cat_2
~.\footnote{In particular, the 2-category $\Fun(\cK, \Cat(\fgt)(\cC))$ is stably-enriched.}
\]
\end{observation}

\begin{proof}[Proof of \Cref{prop.only.one.Cat.St}]
It is clear that the asserted factorization exists, and it is an equivalence on underlying 1-categories by \Cref{obs.stably.enriched.is.unambiguous.on.iota.one}. So, it remains to check that it induces an equivalence on hom 2-categories. For this, let $\cX,\cY \in \iota_1 \Cat(\St) \simeq \Cat(\iota_1 \St)$ be stably-enriched 2-categories. For any $\cK \in \iota_1 \Cat_2$, we have the two inclusions
\begin{align}
\nonumber
\hom_{\iota_1 \Cat_2} ( \cK , \hom_{(\iota_1\Cat(\St))^{\iota_1 \Cat_2\enr} } ( \cX , \cY ) )
& \simeq
\hom_{\iota_1\Cat(\St)} \left( \cX , \cK \cotensor^\flat \cY \right)
\\
\label{use.that.cotensor.from.a.two.cat.into.a.stably.enriched.two.cat.is.functors}
& \simeq
\hom_{\iota_1\Cat(\St)} ( \cX , \Fun(\cK,\cY) )
\\
\label{include.homs.in.Cat.of.iota.one.St.to.homs.in.Cat.of.iota.one.Cat}
& \subseteq
\hom_{\iota_1 \Cat_2} ( \cX , \Fun ( \cK , \cY ) )
\\
\nonumber
& \simeq
\hom_{\iota_1 \Cat_2} ( \cK , \Fun ( \cX , \cY ) )
\\
\label{include.functors.into.Fun.two.ex}
& \supseteq
\hom_{\iota_1 \Cat_2} ( \cK , \Fun^\twoex ( \cX , \cY ) )
~,
\end{align}
where equivalence \Cref{use.that.cotensor.from.a.two.cat.into.a.stably.enriched.two.cat.is.functors} (as well as the existence of the weak cotensor $\cK \cotensor^\flat \cY \in (\iota_1\Cat(\St))^{\iota_1\Cat_2\enr}$) follows from \Cref{obs.functors.into.a.stably.enriched.are.stably.enriched}. Now, the inclusion \Cref{include.functors.into.Fun.two.ex} is of the subspace on those functors $\cK \xra{F} \Fun ( \cX , \cY)$ such that for all $K \in \cK$ the functor $\cX \xra{F(K)} \cY$ is 2-exact. On the other hand, the inclusion \Cref{include.homs.in.Cat.of.iota.one.St.to.homs.in.Cat.of.iota.one.Cat} is of the subspace on those functors $\cX \xra{G} \Fun(\cK,\cY)$ that are 2-exact, and by \Cref{lem.zero.morphisms.and.exact.squares.in.Fun.K.Y} these are characterized by the requirement that for all $K \in \cK$ the composite $\cX \xra{G} \Fun(\cK,\cY) \xra{\ev_K} \cY$ is 2-exact. Thus these inclusions are of the same subspace.
\end{proof}

The following result was used in the proof of \Cref{prop.only.one.Cat.St}.

\begin{lemma}
\label{lem.zero.morphisms.and.exact.squares.in.Fun.K.Y}
Let $\cK, \cY \in \Cat_2$ be 2-categories, and suppose that $\cY$ is stably-enriched.
\begin{enumerate}

\item\label{ev.of.two.catl.functors.is.two.exact} For every object $K \in \cK$, the functor
\[
\Fun(\cK,\cY)
\xra{\ev_K}
\cY
\]
is 2-exact.

\item\label{joint.conservativity.and.exactness.of.ev.K} For any $F,G \in \Fun(\cK,\cY)$, the functors
\begin{equation}
\label{local.action.of.ev.k}
\hom_{\Fun(\cK,\cY)}(F,G)
\xra{\ev_K}
\hom_\cY(F(K),G(K))
\end{equation}
are (exact by part \Cref{ev.of.two.catl.functors.is.two.exact} and) jointly conservative (taken over all $K \in \cK$).

\end{enumerate}
\end{lemma}

\begin{proof}
Part \Cref{ev.of.two.catl.functors.is.two.exact} follows from \Cref{obs.functors.into.a.stably.enriched.are.stably.enriched} (and its evident naturality in the variable $\cK$). So, we turn to part \Cref{joint.conservativity.and.exactness.of.ev.K}. By part \Cref{ev.of.two.catl.functors.is.two.exact}, it suffices to show that for any $\alpha \in \hom_{\Fun(\cK,\cY)}(F,G) \in \St$, if $\ev_K(\alpha) \in \hom_\cY(F(K),G(K)) \in \St$ is zero for all $K \in \cK$ then $\alpha$ is zero. So suppose that $\ev_K(\alpha)$ is zero for all $K$, and consider the morphism
\begin{equation}
\label{morphism.from.zero.in.hom.stablecat.in.Fun.K.Y}
0_{\hom_{\Fun(\cK,\cY)}(F,G)}
\longra
\alpha
\end{equation}
in $\hom_{\Fun(\cK,\cY)}(F,G)$. Applying $\ev_K$, we obtain a morphism
\[
0_{\hom_\cY(F(K),G(K))}
\simeq
\ev_K(0_{\hom_{\Fun(\cK,\cY)}(F,G)})
\xra{\ev_K\Cref{morphism.from.zero.in.hom.stablecat.in.Fun.K.Y}}
\ev_K(\alpha)
~.
\]
By assumption, this is an equivalence for all $K \in \cK$. It follows from \Cref{lem.extensionalityish} that the morphism \Cref{morphism.from.zero.in.hom.stablecat.in.Fun.K.Y} is likewise an equivalence.
\end{proof}

The following result was used in the proof of \Cref{lem.zero.morphisms.and.exact.squares.in.Fun.K.Y}.

\begin{lemma}
\label{lem.extensionalityish}
Let $\cA, \cB \in \Cat_2$ be 2-categories. A 2-morphism
\[ \begin{tikzcd}[column sep = 1.5cm]
F
\arrow[bend left]{r}{\alpha}[swap, xshift=-0.1cm, yshift=-0.2cm]{\varphi \Downarrow}
\arrow[bend right]{r}[swap]{\beta}
&
G
\end{tikzcd} \]
in $\Fun (\cA , \cB)$ is invertible if and only if for every $A \in \cA$ the 2-morphism
\[ \begin{tikzcd}[column sep = 1.5cm]
F(A)
\arrow[bend left]{r}{\alpha(A)}[swap, xshift=-0.3cm, yshift=-0.4 cm]{\varphi(A) \Downarrow}
\arrow[bend right]{r}[swap]{\beta(A)}
&
G(A)
\end{tikzcd} \]
in $\cB$ is invertible.
\end{lemma}

\begin{proof}
The ``only if'' direction is immediate, so we address the ``if'' direction.

By Yoneda, we have a fully faithful embedding $\cB \hookrightarrow \Fun ( \cB^\op , \Cat)$, and so by adjunction we get a fully faithful embedding
\[
\Fun ( \cA , \cB)
\longhookra
\Fun ( \cA , \Fun ( \cB^\op , \Cat))
\simeq
\Fun ( \cA \times \cB^\op , \Cat)
~.
\]
Unwinding the definitions, we see that this reduces us to the case that $\cB = \Cat$ (i.e.\! we may replace $\cA$ by $\cA \times \cB^\op$ and $\cB$ by $\Cat$).

We use the equivalence
\[
\Fun(\cA,\Cat)
\underset{\sim}{\xlongra{\Gr}}
\coCart^\inn_\cA \xhookrightarrow{\text{1-full}} (\Cat_2)_{/\cA}
\]
of \cite[Theorem 1]{GHL}. Using \Cref{prop.Fun.B.Cat.is.prbl.and.tensoring.is.ptwise}, we see that the 2-morphism $\varphi$ is equivalently specified by a 1-morphism
\[
[1] \tensor \Gr(F)
\simeq
([1] \times \cA) \times_\cA \Gr(F)
\simeq
[1] \times \Gr(F)
\xlongra{\varphi}
\Gr(G)
\]
in $\coCart^\inn_\cA$, and it is invertible if and only if this factors through the projection $[1] \times \Gr(F) \xra{\pi} \Gr(F)$. Observe that the morphism $\pi$ admits a right adjoint $\pi^R$ in $\coCart^\inn_\cA$ (in fact it admits both adjoints). The unit gives a comparison morphism $\varphi \ra \varphi \pi^R \pi$ in $\hom_{\coCart^\inn_\cA}([1] \times \Gr(F),\Gr(G)) \in \Cat$, which is an equivalence if and only if it is so for each object of $[1] \times \Gr(F)$ by \cite[Lemma 2.4.8]{macpherson2020bivariant}. This proves the claim.
\end{proof}

{\small

\nocite{*}
\bibliographystyle{amsalpha}
\bibliography{K2}

\providecommand{\bysame}{\leavevmode\hbox to3em{\hrulefill}\thinspace}
\providecommand{\MR}{\relax\ifhmode\unskip\space\fi MR }
\providecommand{\MRhref}[2]{%
  \href{http://www.ams.org/mathscinet-getitem?mr=#1}{#2}
}
\providecommand{\href}[2]{#2}
\begin{thebibliography}{AMGR17b}

\bibitem[AFH23]{AyaFraHow-Symm}
David Ayala, John Francis, and Adam Howard, \emph{Natural symmetries of
  secondary {H}oschchild homology}, arXiv preprint (2023),
  \href{https://arxiv.org/abs/2111.08798}.

\bibitem[AFMGR]{AMGRenriched}
David Ayala, John Francis, Aaron Mazel-Gee, and Nick Rozenblyum, \emph{Enriched
  factorization homology in dimension 1}, to appear.

\bibitem[AFR18]{AFR}
David Ayala, John Francis, and Nick Rozenblyum, \emph{Factorization homology
  {I}: {H}igher categories}, Adv. Math. \textbf{333} (2018), 1042--1177.

\bibitem[AGH09]{AngGerHess-trunc}
Vigleik Angeltveit, Teena Gerhardt, and Lars Hesselholt, \emph{On the
  {$K$}-theory of truncated polynomial algebras over the integers}, J. Topol.
  \textbf{2} (2009), no.~2, 277--294.

\bibitem[AGHL14]{AngGerHillLind-trunc}
Vigleik Angeltveit, Teena Gerhardt, Michael~A. Hill, and Ayelet Lindenstrauss,
  \emph{On the algebraic {$K$}-theory of truncated polynomial algebras in
  several variables}, J. K-Theory \textbf{13} (2014), no.~1, 57--81.

\bibitem[AMGR17a]{AMGRfact}
David Ayala, Aaron Mazel-Gee, and Nick Rozenblyum, \emph{Factorization homology
  of enriched $\infty$-categories}, arXiv preprint (2017),
  \href{http://www.arxiv.org/abs/1710.06414}{arXiv:1710.06414}.

\bibitem[AMGR17b]{AMGRtrace}
\bysame, \emph{The geometry of the cyclotomic trace}, arXiv preprint (2017),
  \href{http://www.arxiv.org/abs/1710.06409}{arXiv:1710.06409}.

\bibitem[AMGR17c]{AMGRcyclo}
\bysame, \emph{A naive approach to genuine {$G$}-spectra and cyclotomic
  spectra}, arXiv preprint (2017),
  \href{http://www.arxiv.org/abs/1710.06416}{arXiv:1710.06416}.

\bibitem[AMGR20]{AMGRstrat}
\bysame, \emph{Stratified noncommutative geometry}, arXiv preprint (2020),
  \href{http://www.arxiv.org/abs/1910.14602}{arXiv:1910.14602}, v2.

\bibitem[AR02]{AusRog-topological}
Christian Ausoni and John Rognes, \emph{Algebraic {$K$}-theory of topological
  {$K$}-theory}, Acta Math. \textbf{188} (2002), no.~1, 1--39.

\bibitem[AR12]{AusRog-rational}
\bysame, \emph{Rational algebraic {$K$}-theory of topological {$K$}-theory},
  Geom. Topol. \textbf{16} (2012), no.~4, 2037--2065.

\bibitem[Bar11]{Bartlett-geometry}
Bruce Bartlett, \emph{The geometry of unitary 2-representations of finite
  groups and their 2-characters}, Appl. Categ. Structures \textbf{19} (2011),
  no.~1, 175--232.

\bibitem[Bar16]{BarwickAlgKTheoryofHigherCats}
Clark Barwick, \emph{On the algebraic \textsc{K}-theory of higher categories},
  Journal of Topology \textbf{9} (2016), no.~1, 245--347.

\bibitem[BC21]{BC-Ainftytwocats}
Nathaniel Bottman and Shachar Carmeli, \emph{{$(A_\infty,2)$}-categories and
  relative 2-operads}, High. Struct. \textbf{5} (2021), no.~1, 401--421.

\bibitem[BCD10]{BCD-cov}
Morten Brun, Gunnar Carlsson, and Bj\o rn~Ian Dundas, \emph{Covering homology},
  Adv. Math. \textbf{225} (2010), no.~6, 3166--3213.

\bibitem[BDR04]{Baas2004}
Nils~A. Baas, Bj{\o}rn~Ian Dundas, and John Rognes, \emph{Two-vector bundles
  and forms of elliptic cohomology}, Topology, Geometry and Quantum Field
  Theory, Cambridge University Press, June 2004, pp.~18--45.

\bibitem[BGT13]{BGT}
Andrew~J Blumberg, David Gepner, and Gon{\c{c}}alo Tabuada, \emph{A universal
  characterization of higher algebraic k-theory}, Geom. Topol. \textbf{17}
  (2013), no.~2, 733--838.

\bibitem[BLL04]{BondalLarsenLunts}
Alexey Bondal, Michael Larsen, and Valery Lunts, \emph{{Grothendieck Ring of
  Pretriangulated Categories}}, International Mathematics Research Notices
  \textbf{2004} (2004), no.~29, 1461.

\bibitem[BM19]{BluMan-AKTS}
Andrew Blumberg and Michael Mandell, \emph{The homotopy groups of the algebraic
  {K}-theory of the sphere spectrum}, Geom. Topol. \textbf{23} (2019),
  101--134.

\bibitem[Bot19]{Bottman-assoc}
Nathaniel Bottman, \emph{2-associahedra}, Algebr. Geom. Topol. \textbf{19}
  (2019), no.~2, 743--806.

\bibitem[BSP21]{BarSPunicity}
Clark Barwick and Christopher Schommer-Pries, \emph{On the unicity of the
  theory of higher categories}, J. Amer. Math. Soc. \textbf{34} (2021), no.~4,
  1011--1058.

\bibitem[BZFN10]{BZFNintegraldrinfeld}
David Ben-Zvi, John Francis, and David Nadler, \emph{{Integral transforms and
  Drinfeld centers in derived algebraic geometry}}, Journal of the American
  Mathematical Society \textbf{23} (2010), no.~4, 909--966.

\bibitem[BZN13]{BZNsecondarytraces}
David Ben-Zvi and David Nadler, \emph{{Secondary Traces}}, arXiv preprint
  (2013), \href{http://www.arxiv.org/abs/1305.7177}{arXiv:1305.7177}, v3.

\bibitem[BZN21]{BZNnonlineartraces}
David Ben-Zvi and David Nadler, \emph{Nonlinear traces}, Derived algebraic
  geometry, Panor. Synth\`eses, vol.~55, Soc. Math. France, Paris, [2021]
  \copyright2021, pp.~39--84.

\bibitem[Cam19]{Campbell2019}
Jonathan~A. Campbell, \emph{The {K}-theory spectrum of varieties}, Transactions
  of the American Mathematical Society \textbf{371} (2019), no.~11, 7845--7884.

\bibitem[CDD11]{CDD-higherTC}
Gunnar Carlsson, Christopher~L. Douglas, and Bj\o rn~Ian Dundas, \emph{Higher
  topological cyclic homology and the {S}egal conjecture for tori}, Adv. Math.
  \textbf{226} (2011), no.~2, 1823--1874.

\bibitem[CP22]{CampbellPontoTraces}
Jonathan~A. Campbell and Kate Ponto, \emph{Iterated traces in 2-categories and
  {L}efschetz theorems}, Algebr. Geom. Topol. \textbf{22} (2022), no.~2,
  815--879.

\bibitem[CT11]{CisTab-Nonconnective}
Denis-Charles Cisinski and Gon\c{c}alo Tabuada, \emph{Non-connective
  {$K$}-theory via universal invariants}, Compos. Math. \textbf{147} (2011),
  no.~4, 1281--1320.

\bibitem[CZ22]{CampbellZakharevichDevissage}
Jonathan~A. Campbell and Inna Zakharevich, \emph{D\'evissage and localization
  for the {G}rothendieck spectrum of varieties}, Adv. Math. \textbf{411}
  (2022), Paper No. 108710, 80.

\bibitem[DGM13]{DGMlocal}
Bj{\o}rn~Ian Dundas, Thomas Goodwillie, and Randy McCarthy, \emph{The local
  structure of algebraic \textsc{K}-theory}, Springer, London New York, 2013.

\bibitem[DK19]{DyckKap-higherSeg}
Tobias Dyckerhoff and Mikhail Kapranov, \emph{Higher {S}egal spaces}, Lecture
  Notes in Mathematics, vol. 2244, Springer, Cham, 2019.

\bibitem[Dou05]{Douglas-thesis}
Christopher~Lee Douglas, \emph{Twisted stable homotopy theory}, ProQuest LLC,
  Ann Arbor, MI, 2005, Thesis (Ph.D.)--Massachusetts Institute of Technology.

\bibitem[DR18]{DR-fusiontwocats}
Christopher~L. Douglas and David~J. Reutter, \emph{Fusion 2-categories and a
  state-sum invariant for 4-manifolds}, arXiv preprint (2018),
  \href{http://www.arxiv.org/abs/1812.11933}{arXiv:1812.11933}.

\bibitem[ES22]{elmanto2020nilpotent}
Elden Elmanto and Vladimir Sosnilo, \emph{On nilpotent extensions of
  {$\infty$}-categories and the cyclotomic trace}, Int. Math. Res. Not. IMRN
  (2022), no.~21, 16569--16633.

\bibitem[Gai15]{gaitsgory2013sheaves}
Dennis Gaitsgory, \emph{Sheaves of categories and the notion of 1-affineness},
  Stacks and categories in geometry, topology, and algebra, Contemp. Math.,
  vol. 643, Amer. Math. Soc., Providence, RI, 2015, pp.~127--225.

\bibitem[GH15]{GH}
David Gepner and Rune Haugseng, \emph{Enriched $\infty$-categories via
  non-symmetric $\infty$-operads}, Advances in Mathematics \textbf{279} (2015),
  575--716.

\bibitem[GHL20]{GHL}
Andrea Gagna, Yonatan Harpaz, and Edoardo Lanari, \emph{{F}ibrations and lax
  limits of $(\infty,2)$-categories}, arXiv preprint (2020),
  \href{http://www.arxiv.org/abs/2012.04537}{arXiv:2012.04537}.

\bibitem[GJF19]{GJF-Condensations}
Davide Gaiotto and Theo Johnson-Freyd, \emph{Condensations in higher
  categories}, arXiv preprint (2019),
  \href{http://www.arxiv.org/abs/1905.09566}{arXiv:1905.09566}.

\bibitem[GK08]{GanterKapranov2008}
Nora Ganter and Mikhail Kapranov, \emph{Representation and character theory in
  2-categories}, Advances in Mathematics \textbf{217} (2008), no.~5,
  2268--2300.

\bibitem[Gro16]{Groth-stable}
Moritz Groth, \emph{{C}haracterizations of abstract stable homotopy theories},
  arXiv preprint (2016),
  \href{http://www.arxiv.org/abs/1602.07632}{arXiv:1602.07632}.

\bibitem[Hau16]{Haugsengbimods}
Rune Haugseng, \emph{Bimodules and natural transformations for enriched
  {$\infty$}-categories}, Homology Homotopy Appl. \textbf{18} (2016), no.~1,
  71--98.

\bibitem[Hei20]{Heine-enriched}
Hadrian Heine, \emph{An equivalence between enriched $\infty$-categories and
  $\infty$-categories with weak action}, arXiv preprint (2020),
  \href{https://arxiv.org/abs/2009.02428}{arXiv:2009.02428}.

\bibitem[Hei24]{Heine-Weighted}
\bysame, \emph{The higher algerbra of weighted colimits}, arXiv preprint
  (2024), \href{https://arxiv.org/abs/2406.08925}{arXiv:2406.08925}.

\bibitem[Hin20]{Hinichyonedainfty}
Vladimir Hinich, \emph{Yoneda lemma for enriched {$\infty$}-categories}, Adv.
  Math. \textbf{367} (2020), 107129, 119.

\bibitem[HM97a]{HessMad-trunc}
Lars Hesselholt and Ib~Madsen, \emph{Cyclic polytopes and the {$K$}-theory of
  truncated polynomial algebras}, Invent. Math. \textbf{130} (1997), no.~1,
  73--97.

\bibitem[HM97b]{HessMad-Witt}
\bysame, \emph{On the {$K$}-theory of finite algebras over {W}itt vectors of
  perfect fields}, Topology \textbf{36} (1997), no.~1, 29--101.

\bibitem[HM03]{HessMad-local}
\bysame, \emph{On the {$K$}-theory of local fields}, Ann. of Math. (2)
  \textbf{158} (2003), no.~1, 1--113.

\bibitem[HM04]{HessMad-dRW}
\bysame, \emph{On the {D}e {R}ham-{W}itt complex in mixed characteristic}, Ann.
  Sci. \'Ecole Norm. Sup. (4) \textbf{37} (2004), no.~1, 1--43.

\bibitem[HSS17]{HStwo}
Marc Hoyois, Sarah Scherotzke, and Nicol{\`o} Sibilla, \emph{Higher traces,
  noncommutative motives, and the categorified {C}hern character}, Advances in
  Mathematics \textbf{309} (2017), 97--154.

\bibitem[HSSS21]{HSthree}
Marc Hoyois, Pavel Safronov, Sarah Scherotzke, and Nicol\`o Sibilla, \emph{The
  categorified {G}rothendieck-{R}iemann-{R}och theorem}, Compos. Math.
  \textbf{157} (2021), no.~1, 154--214.

\bibitem[HW22]{hahn2020redshift}
Jeremy Hahn and Dylan Wilson, \emph{Redshift and multiplication for truncated
  {B}rown-{P}eterson spectra}, Ann. of Math. (2) \textbf{196} (2022), no.~3,
  1277--1351.

\bibitem[KR97]{KleinRog-fib}
John~R. Klein and John Rognes, \emph{The fiber of the linearization map
  {$A(*)\to K({\bf Z})$}}, Topology \textbf{36} (1997), no.~4, 829--848.

\bibitem[Lau12]{Lauda-intro}
Aaron~D. Lauda, \emph{An introduction to diagrammatic algebra and categorified
  quantum {$\mathfrak{sl}_2$}}, Bull. Inst. Math. Acad. Sin. (N.S.) \textbf{7}
  (2012), no.~2, 165--270.

\bibitem[LMSM86]{LMS}
L.~G. Lewis, Jr., J.~P. May, M.~Steinberger, and J.~E. McClure,
  \emph{Equivariant stable homotopy theory}, Lecture Notes in Mathematics, vol.
  1213, Springer-Verlag, Berlin, 1986, With contributions by J. E. McClure.

\bibitem[Loo02]{looijenga-motivic}
Eduard Looijenga, \emph{Motivic measures}, no. 276, 2002, S\'{e}minaire
  Bourbaki, Vol. 1999/2000, pp.~267--297.

\bibitem[LSW20]{LinSatWes-Twisted}
John~A. Lind, Hisham Sati, and Craig Westerland, \emph{Twisted iterated
  algebraic {$K$}-theory and topological {T}-duality for sphere bundles}, Ann.
  K-Theory \textbf{5} (2020), no.~1, 1--42.

\bibitem[Lur09]{HTT}
Jacob Lurie, \emph{{Higher Topos Theory}}, Princeton University Press, 2009.

\bibitem[Lur17]{HA}
\bysame, \emph{{Higher Algebra}}, 2017, Available on author's webpage:
  \url{https://www.math.ias.edu/~lurie/papers/HA.pdf}.

\bibitem[Mac20]{macpherson2020bivariant}
Andrew~W. Macpherson, \emph{A bivariant {Y}oneda lemma and
  $(\infty,2)$-categories of correspondences}, arXiv preprint (2020),
  \href{http://www.arxiv.org/abs/2005.10496}{arXiv:2005.10496}.

\bibitem[Mac21]{macphersonoperad}
Andrew~W. Macpherson, \emph{The operad that co-represents enrichment}, Homology
  Homotopy Appl. \textbf{23} (2021), no.~1, 387--401.

\bibitem[Maz12]{Mazurchuk-lectures}
Volodymyr Mazorchuk, \emph{Lectures on algebraic categorification}, QGM Master
  Class Series, European Mathematical Society (EMS), Z\"{u}rich, 2012.

\bibitem[MG16]{MG-qadjns}
Aaron Mazel-Gee, \emph{Quillen adjunctions induce adjunctions of
  quasicategories}, New York J. Math. \textbf{22} (2016), 57--93.

\bibitem[NVY22]{nakano2021noncommutative}
Daniel~K. Nakano, Kent~B. Vashaw, and Milen~T. Yakimov, \emph{Noncommutative
  tensor triangular geometry}, Amer. J. Math. \textbf{144} (2022), no.~6,
  1681--1724.

\bibitem[NZ09]{NadZas}
David Nadler and Eric Zaslow, \emph{Constructible sheaves and the {F}ukaya
  category}, J. Amer. Math. Soc. \textbf{22} (2009), no.~1, 233--286.

\bibitem[Poo02]{Poonen-notadomain}
Bjorn Poonen, \emph{The {G}rothendieck ring of varieties is not a domain},
  Math. Res. Lett. \textbf{9} (2002), no.~4, 493--497.

\bibitem[Ras18]{RaskinDGM}
Sam Raskin, \emph{{On the Dundas-Goodwillie-McCarthy theorem}}, arXiv preprint
  (2018), \href{http://www.arxiv.org/abs/1807.06709}{arXiv:1807.06709}.

\bibitem[Rez10]{Rezk-ncats}
Charles Rezk, \emph{A {C}artesian presentation of weak {$n$}-categories}, Geom.
  Topol. \textbf{14} (2010), no.~1, 521--571.

\bibitem[Rob15]{Robalo-bridge}
Marco Robalo, \emph{{$K$}-theory and the bridge from motives to noncommutative
  motives}, Adv. Math. \textbf{269} (2015), 399--550.

\bibitem[Rog03]{Rog-White}
John Rognes, \emph{The smooth {W}hitehead spectrum of a point at odd regular
  primes}, Geom. Topol. \textbf{7} (2003), 155--184.

\bibitem[ST04]{StolzTeichner-ell}
Stephan Stolz and Peter Teichner, \emph{What is an elliptic object?}, Topology,
  geometry and quantum field theory, London Math. Soc. Lecture Note Ser., vol.
  308, Cambridge Univ. Press, Cambridge, 2004, pp.~247--343.

\bibitem[Ste20]{Stef-Presentable}
Germ{\'a}n Stefanich, \emph{Presentable {$(\infty, n)$}-categories}, arXiv
  preprint (2020), \href{https://arxiv.org/abs/2011.03035}{arXiv:2011.03035}.

\bibitem[Sus83]{Sus-algebraically}
A.~Suslin, \emph{On the {$K$}-theory of algebraically closed fields}, Invent.
  Math. \textbf{73} (1983), no.~2, 241--245.

\bibitem[Tab08]{Tabuada-HigherKTheory}
Gon\c{c}alo Tabuada, \emph{Higher {$K$}-theory via universal invariants}, Duke
  Math. J. \textbf{145} (2008), no.~1, 121--206.

\bibitem[Tab16a]{tabuada2016note}
Gon{\c{c}}alo Tabuada, \emph{A note on secondary \textsc{K}-theory}, Algebra \&
  Number Theory \textbf{10} (2016), no.~4, 887--906.

\bibitem[Tab16b]{tabuada2016note2}
\bysame, \emph{A note on secondary \textsc{K}-theory \textsc{II}}, arXiv
  preprint (2016),
  \href{http://www.arxiv.org/abs/1607.03094}{arXiv:1607.03094}.

\bibitem[Tab20]{tabuada-embedding}
Gon\c{c}alo Tabuada, \emph{Embedding of the derived {B}rauer group into the
  secondary {$K$}-theory ring}, J. Noncommut. Geom. \textbf{14} (2020), no.~2,
  773--788.

\bibitem[To{\"e}06]{ToenDGcats}
Bertrand To{\"e}n, \emph{{The homotopy theory of dg-categories and derived
  Morita theory}}, Inventiones mathematicae \textbf{167} (2006), no.~3,
  615--667.

\bibitem[To{\"e}11]{toenDGcatlectures}
\bysame, \emph{{Lectures on dg-categories}}, {Topics in Algebraic and
  Topological K-Theory}, Lecture Notes in Mathematics 2008, {Springer Berlin
  Heidelberg}, 2011, pp.~243--301.

\bibitem[TV09]{TV2009}
Bertrand To{\"e}n and Gabriele Vezzosi, \emph{Chern character, loop spaces and
  derived algebraic geometry}, Algebraic topology, Springer, 2009,
  pp.~331--354.

\bibitem[TV15]{TV-caracteres}
Bertrand To\"{e}n and Gabriele Vezzosi, \emph{Caract\`eres de {C}hern, traces
  \'{e}quivariantes et g\'{e}om\'{e}trie alg\'{e}brique d\'{e}riv\'{e}e},
  Selecta Math. (N.S.) \textbf{21} (2015), no.~2, 449--554.

\bibitem[Wal85]{Wald-AKTofspaces}
Friedhelm Waldhausen, \emph{Algebraic {$K$}-theory of spaces}, Algebraic and
  geometric topology ({N}ew {B}runswick, {N}.{J}., 1983), Lecture Notes in
  Math., vol. 1126, Springer, Berlin, 1985, pp.~318--419.

\end{thebibliography}

}

\end{document}